\documentclass[11pt,english]{article}

\RequirePackage{amsthm,amsmath,amsfonts,amssymb}
\RequirePackage[numbers]{natbib}

\usepackage[utf8]{inputenc}
\usepackage[a4paper]{geometry}
\usepackage{color}
\usepackage{babel}
\usepackage{verbatim}
\usepackage{mathtools}
\usepackage{amsmath}
\usepackage{amsthm}
\usepackage{amssymb}
\usepackage{cancel}
\usepackage{enumitem}
\usepackage[unicode=true,pdfusetitle,bookmarks=true,bookmarksnumbered=false,bookmarksopen=false,
 breaklinks=false,pdfborder={0 0 1},backref=false,colorlinks=false]{hyperref}

\usepackage{bbm}
\usepackage{mathrsfs}

\makeatletter

\usepackage{titlesec}
\usepackage{xcolor}
\usepackage{graphicx}

\theoremstyle{plain}

\newtheorem{theorem}{Theorem}
\newtheorem{proposition}[theorem]{Proposition}
\newtheorem{lemma}[theorem]{Lemma}
\newtheorem{corollary}[theorem]{Corollary}

\theoremstyle{remark}

\newtheorem{example}[theorem]{Example}

\newtheorem{remark}[theorem]{Remark}

\newtheorem{manualtheoreminner}{Assumption}
\newenvironment{assumption}[1]{%
  \renewcommand\themanualtheoreminner{#1}%
  \manualtheoreminner
}{\endmanualtheoreminner}

\DeclareMathOperator{\E}{{\mathbb E}}
\DeclareMathOperator{\R}{{\mathbb R}}
\DeclareMathOperator{\N}{{\mathbb N}}

\DeclareMathOperator{\Var}{Var} 

\DeclareMathOperator{\dom}{dom}

\makeatother

\def\cF{\mathcal{F}}
\def\cG{\mathcal{G}}
\def\cH{\mathcal{H}}
\def\cI{\mathcal{I}}

\def\cN{\mathcal{N}}
\def\cO{\mathcal{O}}

\def\bN{\mathbb{N}}

\def\bP{\mathbb{P}}

\def\bR{\mathbb{R}}

\def\sF{\mathscr{F}}

\newcommand{\1}{\mathbbm{1}}            
\newcommand{\set}[1]{\{#1\}}            

\DeclareMathOperator*{\w}{w-\!\!} 
\DeclareMathOperator{\dif}{d \!}        

\definecolor{Red}{rgb}{0.9,0,0.0}
\definecolor{Blue}{rgb}{0,0.0,1.0}

\begin{document}
\global\long\def\epsilon{\varepsilon}

\global\long\def\E{\mathbb{E}}

\global\long\def\I{\mathbf{1}}

\global\long\def\N{\mathbb{N}}

\global\long\def\R{\mathbb{R}}

\global\long\def\C{\mathbb{C}}

\global\long\def\Q{\mathbb{Q}}

\global\long\def\P{\mathbb{P}}

\global\long\def\D{\Delta_{n}}

\global\long\def\dom{\operatorname{dom}}

\global\long\def\b#1{\mathbb{#1}}

\global\long\def\c#1{\mathcal{#1}}

\global\long\def\s#1{{\scriptstyle #1}}

\global\long\def\u#1#2{\underset{#2}{\underbrace{#1}}}

\global\long\def\r#1{\xrightarrow{#1}}

\global\long\def\mr#1{\mathrel{\raisebox{-2pt}{\ensuremath{\xrightarrow{#1}}}}}

\global\long\def\t#1{\left.#1\right|}

\global\long\def\l#1{\left.#1\right|}

\global\long\def\f#1{\lfloor#1\rfloor}

\global\long\def\sc#1#2{\langle#1,#2\rangle}

\global\long\def\abs#1{\lvert#1\rvert}

\global\long\def\bnorm#1{\Bigl\lVert#1\Bigr\rVert}

\global\long\def\wraum{(\Omega,\c F,\P)}

\global\long\def\fwraum{(\Omega,\c F,\P,(\c F_{t}))}

\global\long\def\norm#1{\lVert#1\rVert}

\global\long\def\theta{\vartheta}

\title{Parameter estimation for semilinear SPDEs from local measurements}

\author{Randolf Altmeyer\thanks{Institut für Mathematik, Humboldt-Universität zu Berlin, Unter den Linden 6, 10099 Berlin, Germany. Email: altmeyrx@math.hu-berlin.de}, Igor Cialenco\thanks{Department of Applied Mathematics, Illinois Institute of Technology, 10 W 32nd Str, Building REC, Room 220, Chicago, IL 60616, USA, Email: cialenco@iit.edu}, Gregor Pasemann\thanks{Institut f\"ur Mathematik, Humboldt-Universit\"at zu Berlin, Unter den Linden 6, 10099 Berlin, Germany, E-mail: gregor.pasemann@hu-berlin.de}}

\maketitle

\begin{abstract}

\noindent This work contributes to the limited literature on estimating the diffusivity or drift coefficient of nonlinear SPDEs driven by additive noise. Assuming that the solution is measured locally in space and over a finite time interval, we show that the augmented maximum likelihood estimator introduced in \cite{AltmeyerReiss2019} for linear SPDEs remains rate-optimal  when applied to a large class of semilinear SPDEs. The obtained abstract results are applied to several important classes of SPDEs, including stochastic reaction-diffusion equations. Moreover, we also study the stochastic Burgers equation, as an example with first order nonlinearity, which is a borderline case of the general results. The optimal statistical results are obtained through a precise control of the spatial regularity of the solution and by using higher order fractional $L^p$-Sobolev type spaces. We conclude with numerical examples that validate the theoretical results. 

\medskip


\smallskip

\end{abstract}

\section{Introduction}

While the statistical analysis of stochastic evolution equations, and stochastic partial differential equations (SPDEs) in particular,  is becoming a mature research field, there are many problems left open that broadly can be streamlined into two directions, both undertaken in this paper: a) to consider larger and more diverse classes of equations, usually dictated by specific and practically important  models; b) to develop new statistical methods and techniques that are theoretically  sound and practically relevant. 

We consider a general class of second order \textit{semi-linear SPDEs} of the form\footnote{The equation is strictly defined in Section~\ref{sec:Prelims}.}  
\begin{equation}\label{eq:introSPDE}
\dif X(t)=\vartheta\Delta X(t)\dif t + F(t,X(t)) \dif t+B \dif W(t), \quad 0<t\le T, \ X(0)=X_0,\\
\end{equation}
defined on an appropriate Hilbert space, endowed with zero boundary conditions on a bounded domain $\Lambda \subset\bR^d$, and where $\vartheta$ is the parameter of interest, $F$ is a (nonlinear) function, $B$ a linear operator, and $W$ a cylindrical Brownian motion.

Up until recently, most of the literature on parameter estimation for SPDEs was rooted in the so-called spectral approach by assuming that the observations are obtained in the Fourier space over some finite time interval. For details on this classical method, as well as for general  historical developments in this field, we refer to the survey \cite{Cialenco2018}. Recently, new methods have been developed to study statistical inference problems for linear SPDEs, notably the methodology based on local measurements introduced in \cite{AltmeyerReiss2019}, as well as several approaches dedicated to discrete sampling (cf. \cite{CialencoHuang2017,BibingerTrabs2017,BibingerTrabs2019,Chong2019,Chong2019a,CialencoDelgado-VencesKim2019,KainoMasayukiUchida2019,KhalilTudor2019,CialencoKim2020}), data assimilation (\cite{CotterCrisanHolmEtAl2019,NueskenReichRozdeba2019}) and Bayesian inference (\cite{ReichRozdeba2020}, \cite{Yan2019}). 
While many SPDEs of practical relevance are inherently nonlinear, such equations are considered only in few works (\cite{IgorNathanAditiveNS2010,PasemannStannat2019}, \cite{Pasemann2020}), all within the spectral approach.

The main goal of this paper is to study the estimation of the diffusivity (drift) parameter $\vartheta$ of the \textit{nonlinear SPDE} \eqref{eq:introSPDE} in the context of the \textit{local measurements} framework of \cite{AltmeyerReiss2019}.  We take as an ansatz that the \textit{augmented maximum likelihood estimator} (augmented MLE) of $\vartheta$ introduced in \cite{AltmeyerReiss2019} for linear SPDEs and defined by
$$
\widehat{\vartheta}_{\delta}=\frac{\int_{0}^{T}X_{\delta,x_{0}}^{\Delta}(t)\dif X_{\delta,x_{0}}\left(t\right)}{\int_{0}^{T}(X_{\delta,x_{0}}^{\Delta}(t))^{2}\dif t}, 
$$
has desired asymptotic properties when applied to nonlinear SPDEs, where the observables $X_{\delta,x_0}(t)$, and respectively $X^\Delta_{\delta,x_{0}}(t)$, are obtained from integrating the solution $X$ against a kernel $K_{\delta,x_0}$, and respectively against $\Delta K_{\delta,x_0}$, assuming that $K_{\delta,x_0}$ has support in a $\delta$-neighborhood of a fixed spatial point $x_0$ (hence local measurements).

In the main result of this paper, we prove under some minimal assumptions satisfied by a large class of SPDEs that $\widehat{\vartheta}_{\delta}$, as $\delta\to0$, is a consistent and asymptotically normal estimator of $\vartheta$. Statistically, this shows that spatially localized measurements of semilinear SPDEs contain enough information to identify the coefficient next to the highest order derivative, which is in line with the conclusion of \cite{AltmeyerReiss2019}, as well as with the literature on discrete sampling\footnote{It was shown that to estimate the diffusivity coefficient in a stochastic heat equation driven by an additive noise it is enough to sample the solution at one spacial point over a finite time interval.} listed above, but contrary to the spectral approach, where by its very nature the solution has to be observed everywhere in the physical domain. For an application of the augmented MLE with multiple local measurements to experimental data from cell biology see \cite{altmeyer2020parameter}.

In a nutshell, we establish the exact rate of convergence of the augmented MLE, that depends on the regularity gap (the extra regularity of the nonlinear part of the solution) or the order of the nonlinearity $F$ comparative to the Laplacian. We show that this rate of convergence is not specific to $\widehat{\theta}_{\delta}$ by proving that it is the best possible rate in the minimax sense for any admissible estimator and any sufficiently regular nonlinearity. The derivation of the main results fundamentally exploits  in a novel way fine analytical properties of the solution through a precise control of the spatial regularity of the solution by using higher order fractional $L^p$-Sobolev type spaces. 

The augmented MLE is remarkably flexible. It does not depend on the geometry of the domain $\Lambda$ nor its dimension. Moreover, the estimation procedure remains valid even when the nonlinearity $F$, the covariance operator $B$ or the initial data $X_0$ are unknown or misspecified, as is often the case in practice. We also note that the operator $B$ is not required to commute with the Laplacian $\Delta$, which is one of the core assumptions in the spectral approach. On the other hand, we treat only the parametric case, compared to \cite{AltmeyerReiss2019}, but the extension to nonparametric $\vartheta$ is straightforward, yet computationally significantly more involved. 

The main contributions of this paper can be summarized as follows: First, we present abstract conditions on $F,B,K$ and $X_0$, that guarantee the above mentioned asymptotic properties for $\widehat{\vartheta}_\delta$; Section~\ref{sec:Prelims}. 
We show that these structural conditions are minimal and cover a wide range of SPDEs. Second, we discuss some classical examples of SPDEs proving that the abstract conditions are fulfilled.   
This includes the stochastic reaction-diffusion equations and the stochastic Burgers equation.  Third, we show that equations with first order nonlinearities, such as the stochastic Burgers equation, constitute the extreme case, to which the general asymptotic normality results do not apply while the consistency still holds true. We treat this case separately, by combining the regularity analysis of the solution with its Wiener chaos expansion; cf. Section~\ref{sec:BurgersMain}. Forth, the results for stochastic Allen-Cahn and stochastic Burgers equation are illustrated numerically in Section~\ref{sec:Discussion-and-numerical}.

Thorough discussions on the nature of the proofs, the form of the imposed conditions and comparison to other existing methods, are presented throughout the paper as well as in the concluding Section~\ref{sec:discussion}. Due to the technical nature of the proofs, to streamline the presentation, the vast majority of the results are proved in the Appendix.  Although the well-posedness and regularity properties of the solution are at the core of our analysis, we postpone them to Section~\ref{append:wellPos}, where for the sake of completeness, we also provide a self-contained treatment of well-posedness of SPDEs relevant to the purposes of our study.

\section{Preliminaries and the main problem}\label{sec:Prelims}

\subsection{Notation}

Let $\Lambda$ be an open and bounded set in $\R^{d}$ with smooth boundary $\partial\Lambda$ and let $\langle\cdot,\cdot\rangle$ be the inner product in $L^2(\Lambda)$. For $p>1$ and any linear operator $A:L^p(U)\rightarrow L^p(U)$, where $U\subset \R^d$ is open, let $\norm{A}_{L^{p}(U)}$ denote its operator norm. For $k\in \N_0$, $H^k(\R^d)$ are the usual $L^2$-Sobolev spaces. Let $\Delta z=\sum_{i=1}^d \partial_i^2 z$ denote the Laplace operator on $L^p(\Lambda)$, $p>1$, with zero boundary conditions. To describe higher regularities we consider for $s\in \R$ the fractional Laplacians $(-\Delta)^{s/2}$ on $L^{p}(\Lambda)$, cf. \cite{Yagi10}, and denote their domains by $W^{s,p}(\Lambda):=\{u\in L^{p}(\Lambda):\norm{u}_{s,p}<\infty\}$, where $\norm{\,\cdot\,}_{s,p}:=\norm{\,\cdot\,}_{W^{s, p}(\Lambda)}:=\norm{(-\Delta)^{s/2}\,\cdot\,}_{L^{p}(\Lambda)}$. We also set $W^{s}(\Lambda):=W^{s,2}(\Lambda)$ and $\norm{\cdot}_s:=\norm{\cdot}_{s,2}$, $\norm{\cdot}:=\norm{\cdot}_{2}$. The spaces $W^{s,p}(\Lambda)$ differ from the  Sobolev spaces as defined, for example,  in \cite{Adams}, but they are subspaces of the classical Bessel potential spaces and allow for a Sobolev embedding theorem; for details, see \cite{Triebel1983book}, \cite{DebusscheMoorHofmanova2015},  \cite[Section 16.5]{Yagi10}. Similarly, $(-\Delta_0)^{s/2}$  will stand for the fractional negative Laplace operator on $\R^d$.

We fix a constant $\vartheta\in\bR_+:=(0,\infty)$, that will play the role of the parameter of interest, and denote by $(S_{\vartheta}(t))_{t\geq 0}$ the semigroup generated by $\vartheta\Delta$ on $L^2(\Lambda)$. Respectively,  $(e^{t\Delta_0})_{{t\geq 0}}$ is the heat semigroup on $\R^d$ generated by $\Delta_0$.

Throughout this work we fix a finite time horizon $T>0$, and let $(\Omega,\sF,(\sF_{t})_{0\le t\le T},\P)$ be a filtered probability space supporting a cylindrical Brownian motion $W$ on $L^{2}(\Lambda)$. Informally, $\dot{W}$ is referred to as \textit{space-time} white noise.  Throughout, all equalities and inequalities, unless otherwise mentioned, will be understood in the $\bP$-a.s. sense. As usual, we will denote by $\r{\P}$ the convergence in probability, and $\w\lim$ or $\xrightarrow{d}$ will stand for the convergence in distribution. Correspondingly, for two sequences of random variables $(a_n)_{n\geq 1}$, $(b_n)_{n\geq 1}$, by definition  $a_n=o_{\P}(b_n)$ and $a_n=\cO_\bP(b_n)$, if $a_n/b_n \r{\P} 0$ as $n\rightarrow \infty$ and, respectively,  $\sup_{n\in\N}\bP(|a_n/b_n|>M)\rightarrow 0$, as $M\rightarrow\infty$.

\subsection{The SPDE model}\label{sec:SPDEmodel}
Consider the semilinear stochastic partial differential equation
\begin{equation}
\begin{cases}
\dif X(t)=\vartheta\Delta X(t)\dif t + F(t,X(t)) \dif t+B \dif W(t), \quad 0\leq t\le T,\\
X(0)=X_{0}\in L^2(\Lambda),\\
X(t)|_{\partial\Lambda}=0, \quad 0<t\le T,
\end{cases}\label{eq:SPDE}
\end{equation}
where $F:[0,+\infty)\times \cH\to L^2(\Lambda)$  is a Borel measurable function with a suitable Hilbert space $\cH\subset L^2(\Lambda)$ and a linear operator  $B:L^2(\Lambda)\rightarrow L^2(\Lambda)$. The initial data $X_0$ is $\cF_0$-measurable.

In what follows, we always assume that \eqref{eq:SPDE} has a mild solution, namely that there exists an  adapted process $X=(X(t))_{0\leq t\leq T}$ with values in $L^{2}(\Lambda)$ and such that
\begin{equation}
X(t) = S_{\vartheta}(t)X_0 + \int_0^t S_{\vartheta}(t-s)F(s,X(s))\dif s + \int_0^t S_{\vartheta}(t-s)B \dif W(s). \label{eq:mildSolution}
\end{equation}
In particular, we implicitly assume that all integrals in \eqref{eq:mildSolution} are well-defined.  Sufficient conditions for the existence and uniqueness of mild solutions are well-known (cf.~\cite{DaPratoZabczykBook2014}) and will be discussed for specific equations in Section ~\ref{sec:particularEquations}. The choice to work with mild solutions is primarily dictated by the methods we use to establish fine analytical properties of $X$ that are  needed for the statistical analysis below. 

On the other hand, the statistical experiment, which will be introduced  in the next section, is based only on functionals of the form $\sc{X(t)}z$ for some test functions $z$. We therefore assume that $X$ is also a weak solution to \eqref{eq:SPDE}, that is, $X$ is an $L^{2}(\Lambda)$-valued adapted process such that for any test function $z\in W^{2}(\Lambda)$
\begin{equation}
\sc{X(t)}z=  \sc{X_0}z+\int_{0}^{t}\sc{X(s)}{\vartheta \Delta z}\dif s+\int_{0}^{t}\sc{F(s,X(s))}z \dif s+\sc{W(t)}{B^{*}z}.\label{eq:weakSolution}
\end{equation}
This holds for the process $X$ in \eqref{eq:mildSolution} under mild assumptions, which will be satisfied in all examples below; cf. ~\cite[Theorem 5.4]{DaPratoZabczykBook2014} and \cite[Proposition G.0.5]{LiuRoeckner2015}. 
Generally speaking, considering a weak solution will allow for a larger class of operators $B$, including $B$ being the identity operator and thus \eqref{eq:SPDE} driven by a space-time white noise. We also believe that all results on statistical inference in this paper hold true assuming only the existence of a weak solution in $L^2(\Lambda)$, as long as the splitting argument in Section \ref{sec:MainAssumptions} below together with an analysis of the spatial regularity of the involved processes can be performed, and detailed proofs of this are postponed to future works.

\subsection{Statistical experiment}\label{sec:StatExperiment}

Following the setup from \cite{AltmeyerReiss2019}, we fix a spatial point $x_{0}\in\Lambda$ around which the local measurements of the solution will be performed.  Throughout, we will use the following notations: for $p\geq 2$, $z\in L^{p}(\R^{d})$ and $\delta>0$, 
\begin{align*}
\Lambda_{\delta,x_{0}} & :=\delta^{-1}(\Lambda-x_{0})=\{\delta^{-1}(x-x_{0})\,:\,x\in\Lambda\}, \\
z_{\delta,x_{0}}(x) & :=\delta^{-d/2}z(\delta^{-1}(x-x_{0})),\quad x\in\R^{d},
\end{align*}
and we also set $\Lambda_{0,x_{0}}:=\R^{d}$. For $\delta > 0$, denote by $\Delta_{\delta,x_{0}}$ the Laplace operator on $L^p(\Lambda_{\delta,x_{0}})$ and by $(-\Delta_{\delta,x_0})^{s/2}$, $s\in\R$, the fractional Laplacian with domain $W^{s,p}(\Lambda_{\delta,x_0})$. Correspondingly,  $(S_{\vartheta,\delta,x_{0}}(t))_{t\geq0}$ is the semigroup generated by $\vartheta\Delta_{\delta,x_0}$ on $L^{2}(\Lambda_{\delta, x_0})$.

The measurements are obtained with respect to a fixed function (or kernel) $K\in H^{2}(\R^{d})$ with compact support in $\Lambda_{\delta,x_0}$ for $\delta\leq 1$ such that $K_{\delta,x_0}\in W^{2}(\Lambda)$. Local measurements for the solution $X$ of \eqref{eq:SPDE} at $x_{0}$ with resolution level $\delta$ on the time interval $[0,T]$ are given by the real-valued processes $X_{\delta,x_{0}}=(X_{\delta,x_{0}}(t))_{0\le t\le T}$, and $X_{\delta,x_{0}}^{\Delta}=(X_{\delta,x_{0}}^{\Delta}(t))_{0\le t\le T}$, where
\begin{align}
X_{\delta,x_{0}}(t) & :=\left\langle X(t),K_{\delta,x_{0}}\right\rangle ,\label{EqXdelta}\\
X_{\delta,x_{0}}^{\Delta}(t) & :=\left\langle X(t),\Delta K_{\delta,x_{0}}\right\rangle.\label{EqXdeltaDelta}
\end{align}
Note that $X_{\delta,x_0}^{\Delta}(t)=\Delta X_{\delta,\cdot}(t)|_{x=x_0}$ by convolution, and thus,  $X_{\delta,x_0}^{\Delta}(t)$ can be computed by observing $X_{\delta,x}(t)$ for $x$ in a neighborhood of $x_0$.

The statistical analysis requires additional assumptions on $K$ which will be imposed below.  We will also show that the variance of the proposed estimator depends on the choice of $K$;   cf. Theorem~\ref{thm:rates}. For typical examples of $K$ see Section~\ref{sec:Discussion-and-numerical}.

\subsection{The estimator}

As noticed in \cite{IgorNathanAditiveNS2010}, and consequently used and generalized in \cite{PasemannStannat2019}, the estimator of the diffusivity coefficient $\vartheta$ for linear SPDEs derived within the so-called spectral approach retains its asymptotic properties when applied to a nonlinear SPDE, given that the nonlinear part does not `dominate' the linear part.  Thus, for the local measurements $X_{\delta,x_0}$, $X^{\Delta}_{\delta,x_0}$, we take as ansatz the \textit{augmented maximum likelihood estimator} (augmented MLE) of $\vartheta$ introduced in \cite{AltmeyerReiss2019} for linear SPDEs, which is defined by 
\begin{equation}\label{eq:Estimator}
\widehat{\vartheta}_{\delta}=\frac{\int_{0}^{T}X_{\delta,x_{0}}^{\Delta}(t)\dif X_{\delta,x_{0}}\left(t\right)}{\int_{0}^{T}(X_{\delta,x_{0}}^{\Delta}(t))^{2}\dif t}.
\end{equation}
As discussed in \cite{AltmeyerReiss2019}, this estimator is closely related to, but different from, the actual MLE, which cannot be computed in closed form, even for linear equations and constant $\vartheta$. We also note that $\widehat{\vartheta}_{\delta}$ makes no explicit reference to $F$ or $B$, which are generally unknown to the observer and therefore treated here as nuisance.

From \eqref{eq:weakSolution}, clearly the dynamics of $X_{\delta,x_0}$ are given by
\begin{equation}
  \dif X_{\delta,x_0} = \vartheta X^\Delta_{\delta,x_0} \dif t +
  \langle  F(t,X(t)), K_{\delta,x_0} \rangle\dif t + \norm{B^*K_{\delta,x_0}} \dif  \bar{w}(t),\label{eq:X_delta_eq}
\end{equation}
where $\bar{w}(t):= \langle W(t), B^* K_{\delta,x_0}\rangle/ \norm{B^*K_{\delta,x_0}}$  is a scalar Brownian motion, as long as $\norm{B^*K_{\delta,x_0}}$ does not vanish, which is guaranteed to be true for small $\delta>0$ (cf. Assumption~\ref{assump:B} and the discussion therein). Using \eqref{eq:Estimator} and \eqref{eq:X_delta_eq}, we obtain the error decomposition
\begin{equation}\label{eq:augm_decomp}
\widehat{\vartheta}_{\delta}=\vartheta+(\cI_{\delta})^{-1}R_{\delta}+(\cI_{\delta})^{-1}M_{\delta},
\end{equation}
where
\begin{align*}
\cI_{\delta} & :=\norm{B^{*}K_{\delta,x_{0}}}^{-2}\int_{0}^{T}(X_{\delta,x_{0}}^{\Delta}(t))^{2}\dif t,\tag*{\textrm{(observed Fisher information)}}\\
R_{\delta} & := \norm{B^{*}K_{\delta,x_{0}}}^{-2}\int_{0}^{T}X_{\delta,x_{0}}^{\Delta}(t)\sc{F(t,X(t))}{K_{\delta,x_{0}}}\dif t, \tag*{\textrm{(nonlinear bias)}} \\
M_{\delta} & := \norm{B^{*}K_{\delta,x_{0}}}^{-1}\int_{0}^{T}X_{\delta,x_{0}}^{\Delta}(t)\dif\bar{w}(t).\tag*{\textrm{(martingale part)}}
\end{align*}
The nonlinear bias $R_\delta$ accounts for not observing $(\sc{F(t,X(t))}{K_{\delta,x_0}})_{0\leq t\leq T}$. The observed Fisher information $\c{I}_\delta$ does not correspond to the Fisher information of the statistical model, although it plays a similar role here in the sense that $\cI_\delta \rightarrow \infty$ means `increasing information', and hence yields consistent estimation. In view of \eqref{eq:X_delta_eq}, the decomposition \eqref{eq:augm_decomp} is essentially obtained from the `whitened' process $X_{\delta,x_0}/\norm{B^*K_{\delta,x_0}}$. The statistical performance of $\widehat{\vartheta}_{\delta}$ is therefore not affected by $B$, as $\delta\rightarrow 0$, as we will see below. This is in stark contrast to the regularity properties of $X$, which improve as $B$ becomes more smoothing.

Using the decomposition \eqref{eq:augm_decomp}, to prove consistency, it is enough to show that $(\cI_{\delta})^{-1}R_{\delta}$ and  $(\cI_{\delta})^{-1}M_{\delta}$ vanish, as $\delta\to0$, and to prove asymptotic normality, we will show that $\delta^{-1}(\cI_{\delta})^{-1}R_{\delta}\to0$, while  $\delta^{-1}(\cI_{\delta})^{-1}M_{\delta}$ converges in distribution to a Gaussian random variable.

\subsection{The splitting argument and main model assumptions} \label{sec:MainAssumptions}

In this section, we list high level structural assumptions on the model inputs that will guarantee the desired asymptotic properties of $\widehat{\vartheta}_{\delta}$. These assumptions will be implied by verifiable conditions on the nonlinear term $F$, the operator $B$ and the initial condition $X_0$ in Sections \ref{sec:higherReg} and \ref{sec:particularEquations}.

Similar to \cite{IgorNathanAditiveNS2010,PasemannStannat2019} we use the `splitting of the solution' argument. Namely, consider the $L^{2}(\Lambda)$-valued process $\bar{X}=(\bar X(t))_{0\leq t\leq T}$ given by 
\begin{equation}\label{eq:mildSolutionLinear}
\bar{X}(t) = \int_0^t S_{\vartheta}(t-s)B \dif W(s).
\end{equation}
Analogous to \eqref{eq:mildSolution}, $\bar X$ is a mild solution to the corresponding linear equation
\begin{equation}
\dif \bar{X}(t) = \vartheta \Delta \bar{X}(t)\dif t + B\dif W(t), \quad 0<t\leq T, \quad \bar{X}(0)=0. \label{eq:SPDElinear}
\end{equation}
Then, the nonlinear part $\widetilde{X}:=X-\bar{X}$ satisfies
\begin{equation}\label{eq:NonlinearMild}
\widetilde X(t) = S_{\vartheta}(t)X_0 + \int_0^t S_{\vartheta}(t-s)F(s,\bar{X}(s)+\widetilde{X}(s))\mathrm{d}s, \quad 0 \leq t \leq T,
\end{equation}
namely, it solves the partial differential equation with random coefficients given by
\begin{equation}
\frac{\dif}{\dif t} \widetilde{X}(t) = \vartheta \Delta \widetilde{X}(t) + F(t,\bar {X}(t) + \widetilde{X}(t)), \quad 0<t\leq T, \quad \widetilde{X}(0)=X_0. \label{eq:SPDENonlinear}
\end{equation}
With this at hand, the statistical properties of the local measurements in \eqref{EqXdelta} and \eqref{EqXdeltaDelta} can be studied separately for the linear parts $\bar{X}_{\delta,x_{0}}(t) :=\langle \bar{X}(t),K_{\delta,x_{0}}\rangle$, $\bar{X}_{\delta,x_{0}}^{\Delta}(t) :=\langle \bar{X}(t),\Delta K_{\delta,x_{0}}\rangle$ and the corresponding nonlinear parts $\widetilde{X}_{\delta,x_{0}}(t)$, $\widetilde{X}_{\delta,x_{0}}^{\Delta}(t)$. 

Using \eqref{eq:mildSolutionLinear}, we first note that $\bar{X}_{\delta,x_0}$, $\bar{X}_{\delta,x_0}^{\Delta}$ are centered Gaussian processes. Following similar arguments as in \cite{AltmeyerReiss2019}, exact limits of their covariance functions, as $\delta\rightarrow 0$, will be obtained after appropriate scaling by analyzing the actions of $\vartheta\Delta$ and $S_{\vartheta}(t)$ on the localized functions $z_{\delta,x_0}$; see Section~\ref{append:ScaleCov}. These limits are non-degenerate only under certain scaling assumptions on $B$ and $K$. In view of the error decomposition \eqref{eq:augm_decomp}, we  further impose mild conditions on $\widetilde X$ and $F$ that allow to reduce the entire line of reasoning to the linear case; see Proposition~\ref{prop:pseudo_Fisher_info}. 

\begin{assumption}{\textit{B}}\label{assump:B} 
There exists a constant $\gamma > d/4-1/2$, $\gamma \geq 0$, such that $B:L^2(\Lambda)\rightarrow W^{2\gamma}(\Lambda)$ is an isomorphism. Further, there is a family of linear and bounded operators $(B_{\delta,x_{0}}, \ 0\leq\delta\leq 1)$,  $B_{\delta,x_{0}}:L^2(\R^{d})\rightarrow L^{2}(\R^{d})$ such that
\begin{equation}\label{eq:B_scaling}
B^{*}(-\Delta)^{\gamma}z{}_{\delta,x_{0}}=(B_{\delta,x_{0}}^{*}z)_{\delta,x_{0}},\quad 0<\delta\leq 1,
\end{equation}
for any smooth function $z$ supported in $\Lambda_{\delta,x_{0}}$, and such that
$B^{*}_{\delta,x_{0}}z\rightarrow B^{*}_{0,x_{0}}z$ in $L^{2}(\R^{d})$, for $\delta\rightarrow0$ and $z\in L^{2}(\R^{d})$.
\end{assumption}

\begin{assumption}{\textit{K}}\label{assump:K}
There exists a function $\widetilde{K}\in H^{2\lceil\gamma\rceil+2}(\R^{d})$ with compact support in $\Lambda_{\delta,x_0}$ for $\delta\leq 1$ such that $K= (-\Delta)^{\lceil\gamma\rceil}\widetilde{K}$.
\end{assumption}

\begin{assumption}{\textit{ND}}\label{assump:ND}
With $\Psi(z)  :=\int_{0}^{\infty}\norm{B_{0,x_{0}}^{*}e^{s\Delta_{0}}z}^2_{L^{2}(\R^{d})} \dif s$, assume that $\norm{B_{0,x_{0}}^{*}(-\Delta_{0})^{-\gamma} K}_{L^{2}(\R^{d})}>0$,
$\Psi((-\Delta_{0})^{1-\gamma} K)>0$.
\end{assumption}

\begin{assumption}{\textit{F}}\label{assump:F} 
There exists $\nu > 0$ such that 
\begin{align}
	\int_{0}^{T}(\widetilde{X}_{\delta,x_0}^{\Delta}(t))^{2}\mathrm{d}t &= o_{\P}(\delta^{-2+4\gamma}), \label{eq:assumptionF_tilde}\\
	\int_{0}^{T}\sc{F(t,X(t))}{K_{\delta,x_0}}^{2}\dif t &= \mathcal{O}_{\P}(\delta^{2\nu-2+4\gamma}). \label{eq:assumptionF_F}
\end{align}
\end{assumption}

Next, let us discuss these assumptions in the context of the analytical and statistical properties of the underlying SPDE model. Assumption \ref{assump:B} requires only that $B^*$ scales as the fractional Laplacian $(-\Delta)^{-\gamma}$ when applied to localized functions $z_{\delta,x_0}$. In particular, it is not required that $B^*$ commutes with $\Delta$. The parameter $\gamma$ determines the spatial regularity of $X(t)$. For $\gamma > d/4-1/2$ the linear process \eqref{eq:mildSolutionLinear}  takes values in $L^2(\Lambda)$; see  Proposition~\ref{prop:LinearRegularity}. From the scaling of the fractional Laplacian on localized functions $z_{\delta,x_0}$ (see Lemma~\ref{lem:delta_neg_frac_scaling}) it follows that there exists at most one $\gamma$ satisfying  \eqref{eq:B_scaling} with a non-degenerate operator $B_{0,x_0}$. Moreover, $\gamma$ can be estimated from the observed data. Indeed, having a continuous path of $X_{\delta,x_0}$, for $\delta>0$, at our disposal, one can compute its quadratic variation, which equals $T\norm{B^*{K}_{\delta,x_0}}^2$, cf. \eqref{eq:weakSolution} or \eqref{eq:X_delta_eq}. Finally, $\delta^{2\gamma} T\norm{B^*{K}_{\delta,x_0}}$ converges by \eqref{eq:B_K_convergence} as $\delta\to 0$ to a non-degenerate limit, from which $\gamma$ can be uniquely determined.

Assumptions \ref{assump:K} and \ref{assump:ND} are necessary to ensure non-degenerate variances for $\widehat{\vartheta}_{\delta}$; see Theorem~\ref{thm:rates} and the fact that 
\begin{align*}
\Psi((-\Delta_{0})^{1-\gamma}K)  \leq & \norm{B^*_{0,x_0}}^2_{L^2(\R^d)} \int_{0}^\infty \norm{e^{s\Delta_0}\Delta_0(-\Delta_{0})^{\lceil\gamma\rceil-\gamma}\widetilde{K}}^2_{L^2(\R^d)} \dif s \\
=&  \norm{B^*_{0,x_0}}^2_{L^2(\R^d)} \frac{1}{2} \norm{(-\Delta_{0})^{1/2+\lceil\gamma\rceil-\gamma}\widetilde{K}}^2_{L^2(\R^d)} < \infty,
\end{align*}
concluding by Lemma~\ref{lem:delta_frac_bound_conv} and $1/2+\lceil \gamma\rceil - \gamma>0$. 

Since $X_{\delta,x_0}(t)  = \delta^{2\lceil\gamma\rceil}(-\Delta)^{\lceil\gamma\rceil}\sc{X(t)}{\widetilde{K}_{\delta,\cdot}}|_{x=x_0}$, practically speaking Assumption \ref{assump:K} is not restrictive. 
Thus, analogous to the remark after \eqref{EqXdeltaDelta}, the local measurement  $X_{\delta,x_0}(t)$ in \eqref{EqXdelta} can be obtained by observing $\sc{X(t)}{\widetilde{K}_{\delta,x}}$ for a kernel $\widetilde{K}$ and for $x$ in a neighborhood of $x_0$.  

Next we present a few examples illustrating Assumptions~\ref{assump:B} and \ref{assump:ND}. 

\begin{example}\label{ex:B}

\noindent(i) Let $\gamma$ be as in Assumption \ref{assump:B}. For a smooth function $\sigma\in C^{\infty}(\R^d)$ define the multiplication operator $M_{\sigma}z = \sigma \cdot z$, and consider the linear operator $B=M_{\sigma}(-\Delta)^{-\gamma}$. A larger $\gamma$ corresponds to a smoother noise,  while $\sigma$ controls locally the noise level.  Note that $B$ does not commute with $\Delta$ nor with the semigroup $S_{\vartheta}(t)$, unless $\sigma$ is constant. Then $B^{*}=(-\Delta)^{-\gamma}M_{\sigma}$ and according to Lemmas~\ref{lem:delta_neg_frac_scaling} and \ref{lem:G_delta} we have
\[
B_{\delta,x_{0}}^{*}z=(-\Delta_{\delta,x_{0}})^{-\gamma}M_{\sigma(\delta\cdot+x_{0})}(-\Delta_{\delta,x_{0}})^{\gamma}z,\quad z\in C^{\infty}_{c}(\overline{\Lambda}_{\delta,x_0}).
\]
By Lemma \ref{lem:G_delta}, $B_{\delta,x_{0}}^{*}$ extends
to a bounded operator $B_{\delta,x_{0}}^{*}:L^{2}(\R^{d})\rightarrow L^{2}(\R^{d})$ satisfying $B_{\delta,x_{0}}^{*}z\rightarrow B_{0,x_{0}}^{*}z:=M_{\sigma(x_{0})}z$ for $z\in L^2(\R^{d})$. Moreover, 
\begin{align*}
\Psi((-\Delta_{0})^{\lceil\gamma\rceil-\gamma}\Delta\widetilde{K}) & = \frac{\sigma^2(x_0)}{2} \norm{(-\Delta_{0})^{1/2+\lceil\gamma\rceil-\gamma}\widetilde{K}}^2_{L^2(\R^d)}.
\end{align*}
Assumptions \ref{assump:B} and \ref{assump:ND} are satisfied as long as $(-\Delta_0)^{1/2+\lceil\gamma\rceil-\gamma}\widetilde{K}$ is not identically zero and $\sigma(x_0) \neq 0$. For integer $\gamma$ and using integration by parts the last display simplifies to $\frac{\sigma^2(x_0)}{2}\norm{\nabla\widetilde{K}}^2_{L^2(\R^d)}$.

\bigskip\noindent
(ii) Let now $B=(-\Delta)^{-\gamma}M_\sigma$ for a $\gamma$ as in Assumption \ref{assump:B} and $\sigma\in C(\overline{\Lambda})$, $\sigma(x_0)\neq 0$. Clearly,  $B^*=M_\sigma (-\Delta)^{-\gamma}$ and we immediately obtain $B^*_{\delta,x_0}=M_\sigma$, $B^*_{0,x_0}=M_{\sigma(x_0)}$, and $\Psi$ is as in (i). 

\bigskip\noindent
(iii) With $\gamma$ and $\sigma$ as in (i) let $B=M_\sigma(-\Delta)^{-\gamma}+(-A)^{-\gamma'}$, where $\gamma < \gamma'$ and $A=\Delta - b$ for a constant $b >0$. Note that $A^* = A$ and by Lemma~\ref{lem:scaling}, $Az_{\delta,x_0} = \delta^{-2}(\Delta_{\delta,x_0}z - \delta^2 bz)_{\delta,x_0}$. Moreover, $\norm{(-A)^{-\gamma'}(-\Delta)^{\gamma}z_{\delta,x_0}}_{L^2(\R^d)} \rightarrow 0$ by Lemmas \ref{lem:delta_neg_frac_scaling}, \ref{lem:delta_frac_bound_conv}. Therefore, $B^*_{\delta,x_0}z$ is as in (i), up to a perturbation of order $o(1)$ that may depend on $z$, and hence $B^*_{0,x_0}$, $\Psi$ are again as in (i).  
\end{example}

Assumption \ref{assump:F} is satisfied under sufficient spatial regularity of $\widetilde{X}$ and $F(\cdot,X(\cdot))$, as the next lemma shows. In Section~\ref{sec:higherReg} we show that these regularity properties hold under a general growth condition on $F$. 

\begin{lemma}
\label{lem:regularity}
Grant Assumption \ref{assump:K} and let $r:=s^* + \nu + d/p$, where $s^* := 1+2\gamma-d/2$, $p\geq 2$, $\nu>0$. Assumption~\ref{assump:F} holds true as soon as 
\[
	\widetilde{X}\in C([0,T];W^{r,p}(\Lambda)),\quad  F(\cdot,X(\cdot))\in C([0,T];W^{r-2,p}(\Lambda)).
\]
\end{lemma}
\begin{proof}
Applying Lemma \ref{lem:scalar_ineq} below with $q=p/(p-1)$, we have 
\begin{align*}
	\int_{0}^{T}\widetilde{X}_{\delta,x_0}^{\Delta}(t)^{2}\mathrm{d}t & = \int_{0}^{T}\langle\widetilde{X}(t), \delta^{-2} (\Delta K)_{\delta,x_0}\rangle^{2}\mathrm{d}t \\
& \leq C \delta^{-4+2r+2d(\frac{1}{2}-\frac{1}{p})}
\norm{(-\Delta_{\delta,x_0})^{-\frac{r-2}{2}}K}_{L^{q}(\Lambda_{\delta,x_0})}^{2},
\end{align*}
for a constant $c$. That \eqref{eq:assumptionF_tilde} is satisfied follows from $-4+2r+2d(1/2-1/p)=2\nu-2+4\gamma$ and $\nu>0$, as well as from noting that the norm in the last display is bounded uniformly in $0<\delta\leq 1$ according to Lemma~\ref{lem:fracDelta_K}(i), as $r-2 < 2\lceil\gamma\rceil + d/p$. In the same way, \eqref{eq:assumptionF_F} is obtained from
\begin{align*}
	\int_{0}^{T}\sc{F(X(t))}{K_{\delta,x_0}}^{2}\mathrm{d}t & = \int_{0}^{T}\sc{(-\Delta)^{-1} F(X(t))}{\delta^{-2} (\Delta K)_{\delta,x_0}}^{2}\mathrm{d}t.\qedhere
	\end{align*}
\end{proof}

\section{Main results}\label{sec:Main-results}
In this section we present the main results of this paper, starting with the asymptotic properties 
of the augmented MLE $\widehat{\vartheta}_{\delta}$ pertinent to \eqref{eq:SPDE} in its abstract form, and then discussing refinements to Assumption~\ref{assump:F}. In the third part, we consider several important classes of particular equations, and in teh forth part we focus on the stochastic Burgers equation as an important test case not covered by the general theory, and which is treated by a different approach. Proofs of technical results are postponed to Appendix~\ref{append:ProofMainResults}.

\subsection{Asymptotic analysis of the estimator} \label{sec:AsympAnEst}

We study first the observed Fisher information. In view of the splitting argument let
\[
\bar{\c I}_{\delta}:=\norm{B^{*}K_{\delta,x_{0}}}^{-2}\int_{0}^{T}(\bar{X}_{\delta,x_{0}}^{\Delta}(t))^{2}\dif t
\]
denote the observed Fisher information corresponding to the linear part. 

\begin{proposition} \label{prop:pseudo_Fisher_info}
Assume that Assumptions~\ref{assump:B}, \ref{assump:K} and \ref{assump:ND} are satisfied.
Then, as $\delta\rightarrow0$, the following asymptotics hold true:
\begin{enumerate}
\item $\delta^{2}\E[\bar{\c I}_{\delta}]\rightarrow (\vartheta\Sigma)^{-1}$, where 
\[
\Sigma:=T^{-1} \norm{B_{0,x_{0}}^{*}(-\Delta_{0})^{\lceil\gamma\rceil-\gamma}\widetilde{K}}_{L^{2}(\R^{d})}^{2}\Psi((-\Delta_{0})^{\lceil\gamma\rceil-\gamma}\Delta\widetilde{K})^{-1}.
\]
\item $\c{\bar{I}}_{\delta}/\E[\c{\bar{I}}_{\delta}]\r{\P}1$.
\end{enumerate}
In addition, if Assumption \ref{assump:F} is satisfied, then: 
\begin{enumerate}[resume]
\item $\c I_{\delta}=\bar{\c I}_{\delta}+o_{\P}(\delta^{-2})$.
\item $\c I_{\delta}^{-1}R_{\delta}= \mathcal{O}_{\P}(\delta^{\nu})$.
\end{enumerate}
\end{proposition}
\begin{proof}
	The proof is deferred to Appendix~\ref{append:ProofMainResults}. 
\end{proof}

Now we are in the position to present our first main result.
\begin{theorem} \label{thm:rates}
Assume that Assumptions~\ref{assump:B}, \ref{assump:K}, \ref{assump:ND} and~\ref{assump:F} are satisfied. Then the following assertions hold true:
\begin{enumerate}
\item $\widehat{\vartheta}_{\delta}$  is a weakly consistent estimator of $\vartheta$ and 
\begin{equation}\label{eq:thMain1-cons}
\widehat{\vartheta}_{\delta} =\vartheta + \mathcal{O}_{\P}(\delta^{\nu\wedge1}). 
\end{equation}
\item If $\nu>1$, then $\widehat{\vartheta}_{\delta}$  is asymptotically normal and, with $\Sigma$ from Proposition~\ref{prop:pseudo_Fisher_info}(i),
\begin{equation}\label{eq:ThMain-AsNorm}
\w\lim_{\delta\to0}\delta^{-1}(\widehat{\vartheta}_{\delta}-\vartheta)  =\cN\left(0,\vartheta\Sigma\right).
\end{equation}
\end{enumerate}
\end{theorem}

\begin{proof} Consider the error decomposition \eqref{eq:augm_decomp} and let
$$
Y_{t}^{(\delta)}:=\norm{B^{*}K_{\delta,x_{0}}}^{-1}X_{\delta,x_{0}}^{\Delta}(t)/\E[\bar{\c I}_{\delta}]^{1/2}.
$$
Thus, $M_{\delta}/\E[\bar{\c I}_{\delta}]^{1/2} = \int_{0}^{T}Y_{t}^{(\delta)}d\bar{w}(t)$. By Proposition~\ref{prop:pseudo_Fisher_info}(i)-(iii) we obtain that  
\[
\c I_{\delta}/\E[\bar{\c I}_{\delta}] = (\bar{\c I}_{\delta}+o_{\P}(\delta^{-2}))/\E[\bar{\c I}_{\delta}]\r{\P}1,
\]
such that the quadratic variation of $M_{\delta}/\E[\bar{\c I}_{\delta}]^{1/2}$ satisfies $\int_{0}^{T}(Y_{t}^{(\delta)})^{2} \dif t=\c I_{\delta}/\E[\bar{\c I}_{\delta}]\r{\P}1$.
From here, by a standard central limit theorem for continuous martingales (cf. \cite[Theorem 5.5.4]{LiptserShiryayevBookMartingales}),
we obtain that $M_{\delta}/\E[\bar{\c I}_{\delta}]^{1/2}\xrightarrow{d}\cN(0,1)$. 
We also note that in view of Proposition~\ref{prop:pseudo_Fisher_info}(i)-(ii), $\delta\E[\bar{\c I}_\delta]^{1/2}\rightarrow (\vartheta \Sigma)^{-1/2}$, as $\delta\to0$. Using the above, as well as \eqref{eq:augm_decomp} and Proposition~\ref{prop:pseudo_Fisher_info}(iv), the identity \eqref{eq:thMain1-cons} follows at once.
Similarly and by employing Slutsky's Lemma, we obtain \eqref{eq:ThMain-AsNorm}.
The proof is complete.
\end{proof}

For $\nu<1$ the error in \eqref{eq:thMain1-cons} is dominated by the nonlinear contribution $\c I_{\delta}^{-1}R_{\delta}$ and asymptotic normality does not hold. It is interesting to note that the nonlinear bias will generally \textit{not} decrease with larger $T$, as opposed to the martingale term, which is of order $T^{-1/2}$, see Proposition \ref{prop:pseudo_Fisher_info} and the lower bound in Theorem \ref{thm:minimax} below. Obtaining a central limit theorem in the critical case $\nu=1$ is a challenging problem, and generally speaking has to be treated on case-by-case basis; one such example is the stochastic Burgers equations discussed in Section~\ref{sec:particularEquations}. For $\nu>1$, there is no asymptotic bias in \eqref{eq:ThMain-AsNorm} and since the asymptotic variance depends linearly on the unknown parameter, one can easily deduce an asymptotic confidence interval for $\vartheta$.
\begin{corollary}
Assume that Assumptions~\ref{assump:B}, \ref{assump:K}, \ref{assump:ND} and~\ref{assump:F} are satisfied for $\nu>1$. For $0<\alpha<1$,  let 
\begin{equation*}
I_{1-\alpha} = \left[\widehat{\vartheta}_{\delta}-\cI_{\delta}^{-1/2}q_{1-\alpha/2}, \widehat{\vartheta}_{\delta}+\cI_{\delta}^{-1/2}q_{1-\alpha/2}\right],
\end{equation*}
where $q_\beta$ is the $\beta$-quantile of the standard normal distribution. Then, $I_{1-\alpha}$ is a confidence interval for $\vartheta$ with asymptotic coverage $1-\alpha$, as $\delta\rightarrow 0$.
\end{corollary}
\begin{proof}
By Proposition \ref{prop:pseudo_Fisher_info} we have $\delta^2\cI_{\delta}\rightarrow (\vartheta\Sigma)^{-1}$. Theorem \ref{thm:rates}(ii) and Slutsky's lemma show 
\begin{equation*}
\w\lim_{\delta\to0}\cI_{\delta}^{1/2}\left(\widehat{\vartheta}_{\delta}-\vartheta\right) = \cN(0,1).
\end{equation*}
This yields $\lim_{\delta\to0}\bP(\vartheta \in I_{1-\alpha}) = 1-\alpha$.
\end{proof}

It is worth pointing out that the rate of convergence $\delta^{\nu\wedge1}$ in \eqref{eq:thMain1-cons} does not depend on the `smoothing' parameter $\gamma$. Moreover, as the next result shows, the rate is even minimax optimal. For $\nu > 0$ and $\gamma>d/4-1/2$, let $ \Theta_{\nu,\gamma}$ be the set of all admissible model inputs $\kappa = (\theta,F,B,X_0)$ in \eqref{eq:SPDE} such that $\theta>0$ and Assumptions \ref{assump:B}, \ref{assump:F} are satisfied. We denote by $\P_{\kappa}$ the law of $X_{\delta,x_0}$ on the canonical  space $C([0,T])$, equipped with the Borel sigma algebra corresponding to the sup norm on $[0,T]$, and by $\E_{\kappa}$ its expectation.

\begin{theorem}\label{thm:minimax}
Let $0<\nu\leq 2$ and let $K$ be as in Assumption \ref{assump:K}. When $\nu>1$ let $\gamma>d/4-1/2$, and when $\nu\leq 1$ let $\gamma > 1/2+d/4$. Then, as $\delta\rightarrow 0$, we have the following asymptotic lower bound of the root mean squared error   
\[
	\inf_{\hat{\theta}} \sup_{\kappa\in \Theta_{\nu,\gamma}} \E_{\kappa}\left[(\hat{\vartheta}-\theta)^2 \right]^{1/2} \geq c_1 T^{-1/2}\delta \1_{\{\nu> 1\}} + c_2\delta^{\nu} \1_{\{\nu \leq 1\}},
\]
for some constants $c_1,c_2>0$, and where the infimum is taken over all estimators $\hat{\theta}$ based on observing $X_{\delta,x_0}$. 
\end{theorem}
\begin{proof}
	The proof is deferred to Appendix~\ref{append:ProofMainResults}. 
\end{proof}

The broad specifications of $F$, $B$ and $K$ allow for application of the asymptotic results  to a wide range of SPDEs.  We also emphasize that the asymptotic variance $\vartheta\Sigma$ in Theorem \ref{thm:rates} for $\nu>1$ does not depend on $F$ at all. Therefore, the augmented MLE is robust to the misspecification of $F$, which practically speaking is often difficult to model exactly. As far as $B$ is concerned, similar to Example~\ref{ex:B}(i-iii), only the scaling with respect to $\gamma$ appears in $\Sigma$.

\begin{example}\label{ex:concreteCLT} For $B$ as in Example~\ref{ex:B}(i-iii), grant Assumptions~\ref{assump:K} and \ref{assump:F} for $\nu>1$. In this case, $\Sigma$ can be computed explicitly,  it is independent of $B^*_{0,x_0}=M_{\sigma(x_0)}$, and we have that 
\begin{equation*}
\w\lim_{\delta\to0}\delta^{-1}(\widehat{\vartheta}_{\delta}-\vartheta)  =\cN\left(0,\frac{2 \vartheta \norm{(-\Delta_{0})^{\lceil\gamma\rceil-\gamma}\widetilde{K}}_{L^{2}(\R^{d})}^{2}}{T \norm{(-\Delta_{0})^{1/2+\lceil\gamma\rceil-\gamma}\widetilde{K}}^2_{L^2(\R^d)}} \right).
\end{equation*}
Moreover, for integer $\gamma$, the asymptotic variance is equal to $2\vartheta T^{-1}\norm{\widetilde{K}}^2_{L^2(\R^d)} \norm{\nabla\widetilde{K}}^{-2}_{L^2(\R^d)}$.
\end{example}

\subsection{Higher regularity of the perturbation process}\label{sec:higherReg}

We give now sufficient conditions to verify Assumption~\ref{assump:F}. Inspired by the perturbation argument of \cite{IgorNathanAditiveNS2010,PasemannStannat2019}, we study the spatial regularity of the processes $\bar{X}$ and $\widetilde{X}$. Aiming to obtain optimal regularity that exploits the localization under the kernel $K$, we consider the spaces $W^{s,p}(\Lambda)$ introduced in Section~\ref{sec:Prelims}.

For $p\geq 2$, denote  by $\bar s(p)$ the $L^{p}$-regularity index of the linear process, namely 
\begin{equation}
	\bar s(p) = \sup \set{ s\in\bR\, : \, \bar X\in C([0,T]; W^{s, p}(\Lambda))},\quad \P\text{-a.s..}\label{eq:sBar}
\end{equation}
Under Assumption \ref{assump:B} it can be shown (see Supplement~\ref{app:RegularityLinear}) that 
\begin{equation*}
\max(s^* - d/2 + d/p,0) \leq \bar{s}(p)\leq s^*=1+2\gamma-d/2. 
\end{equation*}
The constant $s^*$ should be viewed as the `optimal expected spatial regularity' of $\bar{X}$, while $\bar{s}(p)$ depends on the geometry of the domain $\Lambda$ and strict inequality may occur. Nevertheless, $\bar{s}(p)=s^*$ for rectangular domains in any dimension, and thus in particular if $d=1$. Note that Theorem \ref{thm:rates} and Theorem \ref{thm:regularity} below can be shown to hold also for non-smooth boundaries $\partial \Lambda$, as long as the eigenfunctions of the Laplacian are smooth on $\bar{\Lambda}$, which is true for rectangular domains.

Let us introduce the following common growth condition on $F$, parametrized by $s,\eta\in\R$, $p\geq 2$.
\begin{assumption}{$A_{s, \eta, p}$}\label{assump:A}
We have $F(t,u)\equiv F(u)$, and there exist $\epsilon>0$ and a continuous function $g:[0,\infty)\rightarrow[0,\infty)$ such that 
\begin{equation*}
\norm{F(u)}_{s+\eta-2+\epsilon, p}\leq g(\norm{u}_{s, p}),\quad u\in W^{s,p}(\Lambda).
\end{equation*}
\end{assumption}
Without loss of generality, we can assume that $g$ is non-decreasing (otherwise replace $g$ with $x\mapsto \sup_{0\leq y\leq x}g(y)$). As we will see in the next section, the term $2-\eta$ should be understood as the order of $F$ in the sense of a differential operator.

\begin{proposition}\label{prop:ExcessRegularity} 
Let $p \geq 2$, $2\leq p_1 \leq p$ and let $0\leq s_1 < \bar{s}(p)$. Assume that
\[
X_0\in W^{\bar{s}(p)+\eta,p}(\Lambda),\quad \widetilde X\in C([0,T]; W^{s_1, p_1}(\Lambda)),
\]
and suppose that Assumption \hyperref[assump:A]{$A_{s,\eta,p'}$} holds true for some $\eta>0$ and all $s_1\leq s < \bar{s}(p)$, $p_1 \leq p' \leq p$. Then $\widetilde X\in C([0,T]; W^{\bar{s}(p)+\eta, p}(\Lambda))$. In particular, $X \in C([0,T]; W^{s, p}(\Lambda))$ for all $s<\bar{s}(p)$. 
\end{proposition}
\begin{proof}
The proof is deferred to Appendix~\ref{append:ProofMainResults}.
\end{proof}

This shows that $\widetilde{X}$ is more regular in space than $\bar{X}$ with \textit{excess regularity} $\eta$. Note that existence results for semilinear SPDEs typically provide some minimal spatial $L^2$-Sobolev regularity for the solution $X$, and thus for $\widetilde{X}$; see \cite{LiuRoeckner2015} or Lemma  \ref{lem:SemilinearWellPosedGlobalInTime} below, assuming additional local Lipschitz and coercivity conditions.

\begin{theorem}
\label{thm:regularity}
Grant Assumption \ref{assump:K} and the assumptions of Proposition \ref{prop:ExcessRegularity}. Suppose that $\eta>s^* - \bar s(p) + d/p$.
Then Assumption~\ref{assump:F} holds true with
\begin{equation*}
		\nu = \left(\eta - \left(s^* - \bar s(p) + d/p\right)\right) \wedge 5/4.
\end{equation*}
\end{theorem}
\begin{proof}
Fix $\nu>0$ as in the statement and let $r:= s^* + \nu +d/p \leq \bar{s}(p)+\eta$. According to Lemma \ref{lem:regularity} we only have to observe that $\widetilde{X}\in C([0,T];W^{r,p}(\Lambda))$ by Proposition \ref{prop:ExcessRegularity} and that 
\begin{equation*}
\sup_{0\leq t\leq T}\norm{F(X(t))}_{r-2, p} \leq g\left(\sup_{0\leq t\leq T}\norm{X(t)}_{s, p} \right) < \infty,
\end{equation*} 
with $\epsilon$ and $g$ from Assumption \ref{assump:A} with $s=\bar{s}(p)-\epsilon'$, $0<\epsilon'<\epsilon$  such that $r-2<s+\eta-2+\epsilon$ and noting $X\in C([0,T]; W^{s, p}(\Lambda))$ by Proposition \ref{prop:ExcessRegularity}.
\end{proof}

In the setting of Theorem \ref{thm:rates}, this result means that asymptotic normality of $\widehat{\vartheta}_{\delta}$ holds as soon as the excess regularity $\eta$ is larger than $s^*-\bar{s}(p)+d/p+1$, while $\widehat{\vartheta}_{\delta}$ is consistent if $\eta>s^*-\bar{s}(p)+d/p$. 
If $\bar{s}(p)=s^*$ and if Proposition \ref{prop:ExcessRegularity} can be applied for all $p\geq 2$, then $\nu = \eta \wedge 5/4$ is independent of the dimension $d$. Compared to this, the $L^2$-perturbation results for the spectral approach of \cite{IgorNathanAditiveNS2010}, \cite{PasemannStannat2019} depend heavily on the dimension, with slower convergence rates for estimators of $\vartheta$ in higher dimensions. It is an interesting question if $L^p$-regularity for $p>2$ can improve results also for the spectral approach.

\subsection{Results for particular equations}\label{sec:particularEquations}

Let us apply Theorems \ref{thm:rates} and \ref{thm:regularity} to SPDEs  with specific nonlinearities. We always assume that Assumptions~\ref{assump:B}, \ref{assump:K}, \ref{assump:ND} are satisfied, which already implies well-posedness of the linear part $\bar{X}$ and allows us to define the `linear regularity gap'
\begin{align*}
	s_\mathrm{gap} = s^*-\inf_{p\geq 2}\bar s(p),
\end{align*}
which satisfies $0\leq s_\mathrm{gap}\leq d/2$; cf. Supplement  \ref{app:RegularityLinear}. Recall also that $s_{\mathrm{gap}}=0$ for rectangular domains, in particular when $d=1$. The initial value $X_0$ is always assumed to satisfy $X_0\in W^{\bar{s}(p)+\eta,p}(\Lambda)$ for all $p\geq 2$ and with $\eta$ to be determined, in order to apply Proposition \ref{prop:ExcessRegularity}. Verification of Assumption \ref{assump:A} will follow mainly by the following simple but convenient result.

\begin{lemma}\label{lemLpReactionDiffusion}
For $0\leq\alpha<2$, $m\in\N_0$ and $p\geq 2$, suppose that $F(u)=D_{\alpha}Q_{m}(u)$, where $Q_{m}:\R\rightarrow\R$ is a polynomial of degree at most $m$, and where $D_\alpha$ is a differential operator of order $\alpha$, i.e. $D_\alpha:W^{s+\alpha, p}(\Lambda)\rightarrow W^{s, p}(\Lambda)$ is bounded for all $s\in\R$. Then Assumption \ref{assump:A} holds for $0\leq\eta<2-\alpha$ and $s > d/p$. When $m\leq 1$, it holds for all $s\geq 0$.
\end{lemma}
\begin{proof}
Set $\epsilon:=2-\alpha-\eta$ such that $\norm{D_\alpha Q_m(u)}_{s+\eta-2+\epsilon, p}\leq C\norm{Q_m(u)}_{s, p}$ for an absolute constant $C<\infty$. This already implies the claim when $m\leq 1$. When $m>1$, it is enough to consider $Q_m(x)=x^m$. For $s>d/p$, the space $W^{s, p}(\Lambda)$ is closed under multiplication; cf.   \cite{Triebel1983book}. This yields $\norm{x^m}_{s, p}\leq \tilde{C} \norm{x}_{s, p}^m$ for another absolute constant $\tilde{C}<\infty$, implying Assumption \ref{assump:A}.
\end{proof}

The results discussed in this and the next section can be combined to apply to more general SPDEs by considering composite nonlinearities of the form $F(u)=a_1 F_1(u) + a_2 F_2(u)$ for sufficiently smooth functions $a_1,a_2$. In this case $a_1,a_2$ are pointwise multipliers on the Bessel potential spaces, cf. \cite[Theorem 3.3.2]{Triebel1983book}, and so $F$ satisfies Assumption $A_{s,\eta_1\wedge\eta_2,p}$ for $p\geq 2$ and $s,\eta_1,\eta_2 \in \R$, as soon as $F_1$ and $F_2$ satisfy Assumptions $A_{s, \eta_1, p}$ and $A_{s, \eta_2, p}$, respectively. In this sense, the results are robust under misspecification of certain lower order terms in the nonlinear part.

\subsubsection{Linear perturbations}\label{sec:linearPerturbations}

Let $D_{\alpha}$ be a differential operator of order $0\leq \alpha<2$ as in Lemma \ref{lemLpReactionDiffusion}, and consider the linear equation
\begin{align}\label{eq:linearSPDE}
	\mathrm{d}X(t) = (\vartheta\Delta X(t) + D_{\alpha}X(t))\mathrm{d}t + B\mathrm{d}W(t).
\end{align}
Examples for $D_{\alpha}$ are $(-\Delta)^{\alpha/2}$ or first order differential operators such as $u\mapsto \sc{b}{\nabla u}_{\R^d} + cu$ with $\alpha=1$, $b\in C^{\infty}(\R^d;\R^d)$, $c\in C^{\infty}(\R^d)$. For  applications of linear SPDEs see e.g. \cite{Walsh1981}, \cite{Frankignoul1985}, \cite{Cont2005}. 
Well-posedness follows as for $\bar{X}$, cf. Supplement \ref{app:RegularityLinear}, as long as the operator $\vartheta \Delta + D_{\alpha}$ generates an analytic semigroup. To satisfy Assumption \ref{assump:F} we further require $D_{\alpha}X(t)\in L^2(\Lambda)$, that is $X(t)\in W^{\alpha}(\Lambda)$. Note that $\vartheta$ is not identifiable for $\alpha=2$, for example when $D_{2}=b \Delta$ for unknown $b$. For simplicity, we consider only $s_{\mathrm{gap}}=0$.

\begin{theorem} Let $0\leq \alpha<2$, $s_{\mathrm{gap}}=0$ and assume that \eqref{eq:linearSPDE} is well-posed in $C([0,T];W^{s, p}(\Lambda))$ for all $\alpha<s<s^*$ and all $p\geq 2$. Then:
\begin{enumerate}
\item $\widehat\vartheta_\delta$ is a consistent estimator of $\vartheta$ with $\widehat\vartheta_\delta = \vartheta + O_{\P}(\delta^{(2-\alpha')\wedge 1})$ for any $\alpha'>\alpha$.
\item If $\alpha < 1$, then $\widehat\vartheta_\delta$ is also an asymptotically normal estimator of $\vartheta$ satisfying \eqref{eq:ThMain-AsNorm}. 
\end{enumerate}
\end{theorem}
\begin{proof} By Theorems \ref{thm:rates} and \ref{thm:regularity} choosing $p=2d/(\alpha'-\alpha)$, $\eta = 2-\alpha'+d/p<2-\alpha$ and $\nu=(2-\alpha')\wedge 5/4$ in the notation therein, it is enough to check Assumption \ref{assump:A}, which holds by Lemma \ref{lemLpReactionDiffusion} with $\eta < 2-\alpha$, $s\geq 0$ for all $p\geq 2$. 
\end{proof}

In the critical case $\alpha=1$, that is with $\nu=1$, it is \textit{a-priori} not clear if a CLT for $\widehat{\vartheta}_{\delta}$ holds at the optimal rate $\delta$. For the examples mentioned after \eqref{eq:linearSPDE}, however, this can be shown to be true by an explicit computation for the nonlinear bias as in \cite[Theorem 5.3]{AltmeyerReiss2019}, and we leave the details to the reader; cf. also the proof of Theorem \ref{thm:Burgers_CLT} below. It is worth mentioning that the results of \cite{AltmeyerReiss2019} are obtained for linear equations of the form \eqref{eq:linearSPDE} with $\vartheta\Delta+D_{\alpha}$ being a second order elliptic  operator, $\gamma=0$, any dimension $d$ and assuming only a $C^2$-boundary for $\Lambda$. 

\subsubsection{Stochastic reaction-diffusion equations}\label{sec:ReactionDiffMain}

Let us consider the equation
\begin{align}\label{eq:SPDERepeated}
	\mathrm{d}X(t,x) = (\vartheta\Delta X(t,x) + f(X(t,x)))\mathrm{d}t + B\mathrm{d}W(t,x),\quad x\in\Lambda,
\end{align}
where the nonlinearity $F(u)(x)=f(u(x))$ is a Nemytskii operator for a function $f:\R\rightarrow \R$. These equations are ubiquitous in physics, chemistry, biology and neuroscience, see e.g. \cite{AlonsoStangeBeta2018}, \cite{NagumoEtAl1962}, \cite{Fitzhugh1961}, \cite{Schloegl1972}, \cite{CahnAllen1977}.

Important examples are polynomial nonlinearities
\begin{equation}
	f(x) = a_m x^m + \dots + a_1x  + a_0,\quad x\in\R,\label{eq:poly}
\end{equation}
with $m\in 2\N+1$ and $a_m < 0$. For a numerical example see Section \ref{sec:Discussion-and-numerical}. Theorem \ref{thm:ExistenceReactionDiffusionComplete} gives sufficient conditions to guarantee that \eqref{eq:SPDERepeated} is well-posed in $C([0,T];W^{s, p}(\Lambda))$ for some $p\geq 2$ and $s>d/p$ in $d\leq 3$. 

For a second class of stochastic reaction diffusion equations consider $f\in C_b^\infty(\R)$, which is the space of smooth functions with bounded derivatives. For a concrete application see \cite{LockleyEtAl2015}. In this case, well-posedness of \eqref{eq:SPDERepeated} in $C([0,T];W^{s, p}(\Lambda))$ for some $p\geq 2$ and $s>d/p$ follows from Theorem \ref{thm:compOperators}.

\begin{theorem}\label{thm:reactDiff} Assume that \eqref{eq:SPDERepeated} is well-posed in $C([0,T];W^{s, p}(\Lambda))$ for all $p\geq 2$ and $d/p<s<s^*-s_\mathrm{gap}$, where $f$ is either as in \eqref{eq:poly} or $f\in C_b^\infty(\R)$. The following assertions hold true:
\begin{enumerate}
\item If $s_\mathrm{gap} < 2$, in particular if $d\leq 3$, then $\widehat\vartheta_\delta$ is a consistent estimator of $\vartheta$, and for $\nu < (2-s_\mathrm{gap})\wedge 1$, $\widehat\vartheta_\delta = \vartheta + O_{\P}(\delta^\nu)$, as $\delta\to0$. 
\item If $s_\mathrm{gap}<1$, in particular if $d=1$, then $\widehat\vartheta_\delta$ is also an asymptotically normal estimator of $\vartheta$ satisfying \eqref{eq:ThMain-AsNorm}. 
\end{enumerate}
\end{theorem}
\begin{proof}
It is enough to show that Assumption \ref{assump:A} holds for all $\eta$ close to $2$ and all $p\geq 2$, $s>d/p$, since then the result is obtained by Theorems \ref{thm:rates} and \ref{thm:regularity} (with $p=2d/(2-s_\mathrm{gap}-\nu)$ and $\eta=\nu + (s^*-\bar s(p)) + d/p \leq \nu + s_\mathrm{gap} + d/p < 2$ in the notation therein, where we can take any $1<\nu<(1-s_\mathrm{gap})\wedge 5/4$ in (ii)). With respect to $f$ in  \eqref{eq:poly}, this follows from Lemma \ref{lemLpReactionDiffusion} with $\alpha=0$, and for $f\in C_b^\infty(\R)$ from Lemma \ref{lem:ConditionsBoundedPerturbation}(i).
\end{proof} 

\subsection{An example for the critical case: The stochastic Burgers equation}\label{sec:BurgersMain}

As a prototypical example for an SPDE with first order nonlinearity, let us consider the stochastic Burgers equation in dimension $d=1$,
\begin{equation}\label{eq:Burgers}
\mathrm{d}X(t)=(\vartheta\Delta X(t)-X(t)\partial_{x}X(t))\mathrm{d}t+B\mathrm{d}W(t).
\end{equation}
This equation serves as a simple model for turbulence and is the one-dimensional analogue to the Navier-Stokes equations; for applications see e.g. the references in \cite{HairerVoss2011}. Note that the nonlinearity is given by 
\begin{equation}
F(u)=-u\partial_{x}u=\partial_{x}\left(-\frac{1}{2}u^{2}\right).
\end{equation}
It can be shown that \eqref{eq:Burgers} has a mild solution when $B=I$; cf. \cite{DaPratoDebusscheTemam1994}. In order to obtain higher regularity of the solution let us assume Assumption~\ref{assump:B} with $\gamma>\frac{1}{4}$. Theorem~\ref{thm:ExistenceBurgersComplete} yields $X\in C([0,T];W^{s,p}(\Lambda))$ for all $p\geq 2$ and $1<s<1/2+2\gamma$. With $d=1$ we find that $s_{\mathrm{gap}}=0$ and Lemma \ref{lemLpReactionDiffusion} with $\alpha=1$ implies Assumption \ref{assump:A} for any $\eta<1$. This is not enough to obtain asymptotic normality of $\widehat{\vartheta}_{\delta}$ using Theorems \ref{thm:rates} and \ref{thm:regularity}. 
Instead, we tackle the nonlinear bias $R_{\delta}$ directly and show that $\delta^{-1}\cI_{\delta}^{-1}R_{\delta}=o_{\P}(1)$. The proof is based on decomposing $F(X)$ using the splitting argument $X=\bar{X}+\tilde{X}$. Terms involving only $\bar{X}$ are treated by Gaussian calculus, similar to $\bar{\cI}_{\delta}$ in Proposition \ref{prop:pseudo_Fisher_info}. Moreover, $\bar{X}$ and $\tilde{X}$ are decoupled using the higher regularity of $\tilde{X}$ over $\bar{X}$ according to Proposition \ref{prop:ExcessRegularity} and by a Wiener-chaos decomposition of $\tilde{X}(t,x_0)$. For the proof we assume $B=(-\Delta)^{-\gamma}$ and a slightly stronger condition on the kernel $K$ to shorten technical arguments, but this can likely be relaxed; see also the numerical study in the next section. For a proof see Supplement \ref{sec:ProofOfBurgers}. 

\begin{theorem}\label{thm:Burgers_CLT} Assume $B=(-\Delta)^{-\gamma}$
for $\gamma>1/4$. Grant Assumption \ref{assump:K} and assume in
addition that $\widetilde{K}=\partial_{x}L$ for $L\in H^{2\lceil\gamma\rceil+3}(\R)$
having compact support. Then $\delta^{-1}\cI_{\delta}^{-1}R_{\delta}=o_{\P}(1)$. 
\end{theorem}

Combining this with the discussion above, Theorems \ref{thm:rates} and \ref{thm:regularity} (with any $0<\nu<1$, $p=2/(1-\nu)$ and $\eta = \nu+1/p<1$) yield immediately:

\begin{theorem}
Assume that \eqref{eq:Burgers} is well-posed in $C([0,T];W^{s, p}(\Lambda))$ for all $p\geq 2$ and $1<s<1/2+2\gamma$. Then the following holds:
\begin{enumerate}
\item The estimator $\widehat\vartheta_\delta$ is consistent with $\widehat\vartheta_\delta = \vartheta + O_\mathbb{P}(\delta^\nu)$ for any $\nu<1$.  
\item If the additional hypotheses from Theorem \ref{thm:Burgers_CLT} are satisfied, then $\widehat\vartheta_\delta$ is also an asymptotically normal estimator of $\vartheta$ satisfying \eqref{eq:ThMain-AsNorm}. 
\end{enumerate}
\end{theorem}

\section{Numerical examples}\label{sec:Discussion-and-numerical} 
In this section we illustrate the theoretical results by some simple numerical experiments. A detailed numerical analysis is beyond the scope of this manuscript. 

Let $T=1$ and $\Lambda=(0,1)$ and consider the stochastic Allen-Cahn equation, a stochastic reaction diffusion equation of the form
\begin{align}\label{eq:AllenCahn}
dX(t)=(\vartheta\Delta X(t)+10X(t)(1-X(t))(X(t)-0.5))\dif t+\sigma \dif W(t),
\end{align}
with zero boundary conditions, $\vartheta=0.01$, $\sigma=0.05$, and driven by space-time white noise, i.e. $B=\sigma I$ and $\gamma=0$. The initial value $X_0$ is assumed to be smooth, equal to $1$ on $[0.3,0.7]$ and vanishing outside of $[0.3-\epsilon,0.7+\epsilon]$ for a small $\epsilon>0$. 

To approximate the solution of \eqref{eq:AllenCahn} we use a finite difference scheme, cf. \cite[Example 10.31]{Lord2014}, with respect to a regular time-space grid $\{(t_{k},y_{j}):t_{k}=k/N,y_{j}=j/M,k=0,\dots,N,j=0,\dots,M\}$, with $M=500$, $N=10^{5}$. The heat map of a typical realization of the solution is presented in Figure~\ref{fig:1} (top left).  We see that the bistable nonlinearity in \eqref{eq:AllenCahn} leads to a persistent phase separation by the solution trajectory, up to stochastic fluctuations.

\begin{figure}
	\includegraphics[scale=0.5]{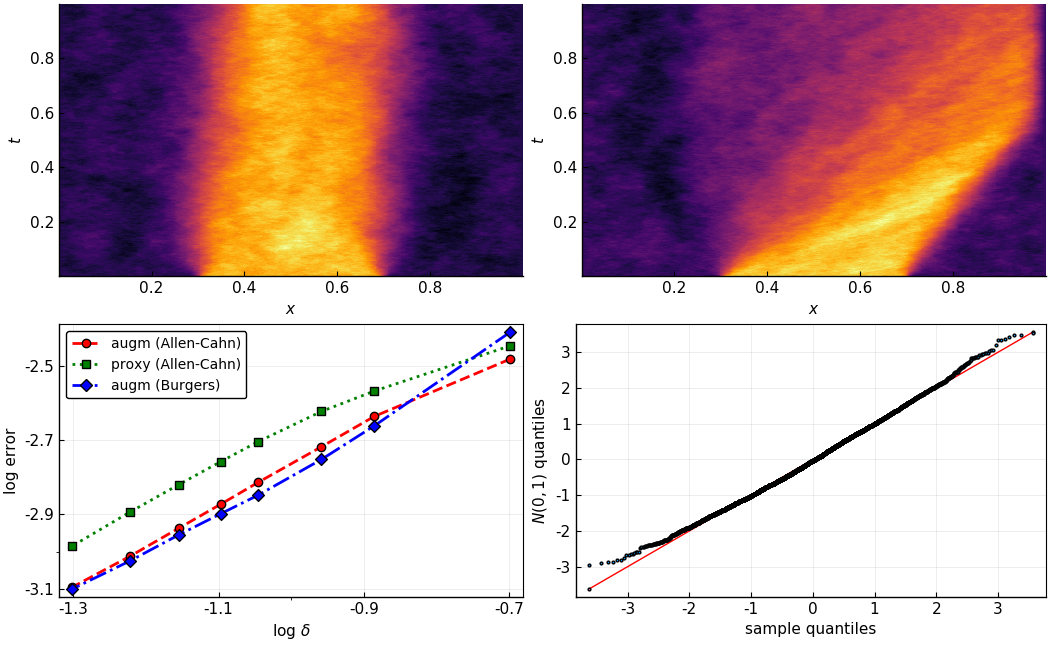}
	\caption{\label{fig:Estimation} Top row: heat maps for typical realizations of the stochastic Allen-Cahn (left panel) and Burgers (right panel) equations. Bottom row:  $\log_{10}$-$\log_{10}$ plot of root mean squared estimation errors at $x_0=0.4$ (left panel) and Normal Q-Q plot for Allen-Cahn at $x_0=0.4$, $\delta=0.05$ (right panel).}
	\label{fig:1}
\end{figure}

Consider the kernel $K=\widetilde{K}=\varphi'''$ from \cite{AltmeyerReiss2019} with a smooth bump function
\[
\varphi(x) :=\exp(-\frac{12}{1-x^{2}}),\,\,\,\,x\in(-1,1).
\]
For $\delta\in [0.05,0.2]$ and $x_{0}\in(0,1)$
we then obtain approximate local measurements $X_{\delta,x_{0}}$,
$X_{\delta,x_{0}}^{\Delta}$, from which the augmented MLE $\widehat{\vartheta}_{\delta}$ is computed. For $x_{0}<\delta$ set $K_{\delta,x_{0}}:=K_{\delta,\delta}$ and for $x_{0}>\delta$ set $K_{\delta,x_{0}}:=K_{\delta,1-\delta}$. Note that the theoretical asymptotic variance $\vartheta \Sigma$ of Theorem \ref{thm:reactDiff} is available by Example \ref{ex:concreteCLT}. 

Using $5000$ Monte-Carlo runs, in Figure \ref{fig:1} (bottom right) we display a Normal Q-Q plot for the approximate distribution of $(\vartheta \Sigma)^{-1/2}\delta^{-1}(\widehat{\vartheta}_{\delta}-\vartheta)$ obtained at $x_{0}=0.4$ and for $\delta=0.05$. Clearly, the sample distribution is very close to the theoretical asymptotic distribution. Moreover, in Figure \ref{fig:1} 
(bottom left)
we present a $\log_{10}$-$\log_{10}$ plot of root mean squared estimation errors for $\delta\rightarrow0$, demonstrating that the rate of convergence indeed approaches $\delta$ as the resolution tends to zero. For comparison, we also include results for another estimator - the proxy MLE introduced in \cite{AltmeyerReiss2019} - that is based 
on observing only $X_{\delta,x_0}$, with the same $K$ as above. 
We note that the performance of the proxy MLE  is comparable to the augmented MLE, which suggests that results similar to Theorem \ref{thm:rates} may hold true for the proxy MLE. 

At last, we consider the same steps for the stochastic Burgers equation \eqref{eq:Burgers} with the same $\vartheta = 0.01$ and  driven by the same noise as in \eqref{eq:AllenCahn}. The heat map for a typical realization is given in Figure~\ref{fig:1} (top right). Notice that the interface of the traveling wave therein is smooth as the equation is viscous (meaning that $\theta>0$). We remark that the finite difference scheme has to be adjusted, see \cite{HairerVoss2011} for details, but this adjustment does not affect the estimation of $\vartheta$, as it is of zero differential order and therefore negligible compared to the Laplacian under scaling with $\delta$; cf. Lemma~\ref{lem:scaling}.  The Normal Q-Q plot remains essentially unchanged (not shown). 
The root mean squared estimation errors for small $\delta$  are displayed in Figure~\ref{fig:1} (bottom left, diamond marked line), which essentially coincide with the one corresponding to the Allen-Cahn equation. Although Theorem~\ref{thm:Burgers_CLT} was proved under the additional assumption $\gamma > 1/4$, the numerical results suggest that the asymptotic results for $\widehat{\vartheta}_\delta$ remain valid  under weaker assumptions on the noise, in particular for $\gamma=0$. 

Similar results were obtained for other sets of parameters. The numerical simulations were performed using Julia and the source code can be obtained from the authors upon request. 

\section{Concluding remarks}\label{sec:discussion}
We showed that the augmented MLE provides a unified approach for estimating the diffusivity coefficient from local measurements for nonlinear SPDEs that can be used in modeling a large variety of dynamical phenomena. Remarkably, the proposed estimator does not depend on the specific parts of the nonlinearity $F$, the noise operator $B$ or the initial condition. 
Practically speaking, this estimator is easy to implement, and it was successfully applied recently to experimental data in cell biology \cite{altmeyer2020parameter} showing promising results in comparison to some more traditional fitting methods.

In contrast to the spectral approach to statistical inference for SPDEs, that is inherently based on global spatial measurements and global $L^2$-regularity properties of the solution, the augmented MLE uses only spatially localized measurements $X_{\delta,x_0}$ and $X_{\delta,x_0}^\Delta$, and can exploit the local regularity of the solution. On the other hand, the minimax-lower bound in Theorem \ref{thm:minimax} assumes only observation of $X_{\delta,x_0}$, and we conjecture that the ideas  developed here can be extended to the proxy MLE of \cite{AltmeyerReiss2019}.  

We have mostly focused on equations with nonlinearities satisfying the growth condition \ref{assump:A}, primarily because this allows for a straightforward regularity analysis. Assumption \ref{assump:F}, which is implied by \ref{assump:A}, holds likely in much more general situations, for example for non-Markovian dynamics (cf. \cite{Pasemann2020}) or for SPDEs with multiplicative noise. In addition, the augmented MLE is well-defined assuming only a weak solution of \eqref{eq:SPDE}. This suggests that the obtained results on statistical inference may hold also for SPDEs with rougher noise (e.g. space-time white noise).  



\begin{appendix}

\section{Proofs of the main results}\label{append:A}

Here we give the proof of our main results, further technical statements are formulated and proven in the supplement.
From now on, without loss of generality, we assume that $x_{0}=0$,
or formally we replace $\Lambda$ by $\Lambda-x_{0}$. To ease the
notations, we also remove $x_{0}$ whenever necessary, for example
by writing $\Lambda_{\delta}, \ z_{\delta}, \ \Delta_{\delta}, \ S_{\vartheta,\delta}$ instead of
$\Lambda_{\delta,x_{0}}$, $\ z_{\delta,x_{0}}$, $\Delta_{\delta,x_0}$, $S_{\vartheta,\delta,x_0}$. We further write $\Delta = \Delta_1$, $S(t) = S_{1,1}(t)$, $S_{\delta}(t) = S_{1,\delta}(t)$. Note that $S_{\vartheta,\delta}(t)=S_{\delta}(\vartheta t)$.  As usual, we will denote by $C$ a generic positive constant, which may change from line to line and depend on $T$, but not on $\delta$. In addition, $a\lesssim b$ means $a\leq Cb$ for $a,b\in\R$. If not mentioned otherwise, all limits are taken as $\delta\rightarrow0$. 

\subsection{On semigroups and the fractional Laplacian}\label{append:Semigroups}

The Laplacian and its semigroup satisfy a certain scaling property with respect
to localized functions. The proof is straightforward; see \cite[Lemma~ 3.1]{AltmeyerReiss2019}
for details when $p=2$, the general case is analogous. 

\begin{lemma}\label{lem:scaling}For $2\leq p <\infty$, $\delta>0$:
	\begin{enumerate}
		\item If $z\in W^{2,p}(\Lambda_{\delta})$,
		then $\Delta z_{\delta}=\delta^{-2}(\Delta_{\delta}z)_{\delta}$. 
		\item If $z\in L^{p}(\Lambda_{\delta})$, then $S(t)z_{\delta}=(S_{\delta}(t\delta^{-2})z)_{\delta}$,
		$t\geq0$.
	\end{enumerate}
\end{lemma}

In order to extend the scaling property to the fractional Laplacian, recall (from \cite[Chapter~2.6]{Pazy:1983us}, for example) that the fractional Laplacian can be represented as 
\begin{align}
(-\Delta)^{-h} & =\frac{1}{\Gamma(h)}\int_{0}^{\infty}t^{h-1}S(t)\dif t,\quad h>0,\label{eq:neg_frac_laplace_formula}\\
(-\Delta)^{h} & =-\frac{\sin\pi h}{\pi}\int_{0}^{\infty}t^{h-1}\Delta(t-\Delta)^{-1}\,\dif t\nonumber \\
& =-\Gamma(h)\frac{\sin\pi h}{\pi}\int_{0}^{\infty}(t')^{-h}\Delta S(t')\,\dif t',\quad0<h<1,\label{eq:frac_Laplace_formula}
\end{align}
where $\Gamma$ is the gamma function and using the resolvent
equality $(t-\Delta)^{-1}=\int_{0}^{\infty}e^{-tt'}S(t')\dif t'$
in the last line. Note that these formulas also apply to the fractional Laplacian on $\R^d$.

\begin{lemma} \label{lem:delta_neg_frac_scaling} Let $2\leq p < \infty$, $\delta>0$.
	If $h>0$ and $z\in L^{p}(\Lambda_{\delta})$ (or $h\leq0$ and $z\in W^{-2h,p}(\Lambda_{\delta})$),
	then $(-\Delta)^{-h}z_{\delta}=\delta^{2h}((-\Delta_{\delta})^{-h}z)_{\delta}$. 
	
\end{lemma}

\begin{proof} Let first $h>0$ and $z\in L^{p}(\Lambda_{\delta})$.
	In view of \eqref{eq:neg_frac_laplace_formula} and Lemma \ref{lem:scaling}(ii)
	we have 
	\begin{align*}
	(-\Delta)^{-h}z_{\delta} & =\left(\frac{1}{\Gamma(h)}\int_{0}^{\infty}t^{h-1}S_{\delta}(t\delta^{-2})z\,\dif t\right)_{\delta}\\
	\  & =\delta^{2h}\left(\frac{1}{\Gamma(h)}\int_{0}^{\infty}t^{h-1}S_{\delta}(t)z\,\dif t\right)_{\delta}=\delta^{2h}\left((-\Delta_{\delta})^{-h}z\right)_{\delta}.
	\end{align*}
	For $h\leq0$ it is enough to consider smooth $z$ supported in $\Lambda_{\delta}$.
	Set $\tilde{h}=-h\geq 0$. By Lemma \ref{lem:scaling}(i) and $(-\Delta)^{\tilde{h}}=(-\Delta)^{\tilde{h}-\lfloor\tilde{h}\rfloor}(-\Delta)^{\lfloor\tilde{h}\rfloor}$,
	the problem reduces to $0\leq\tilde{h}\leq1$. The result follows
	using \eqref{eq:frac_Laplace_formula} and Lemma \ref{lem:scaling}(i,ii)
	from 
	\begin{align*}
	(-\Delta)^{\tilde{h}}z_{\delta} & =\delta^{-2\tilde{h}}\left(-\Gamma(\tilde{h})\frac{\sin\pi\tilde{h}}{\pi}\int_{0}^{\infty}(t')^{-\tilde{h}}\Delta_{\delta}S_{\delta}(t')z\,\dif t'\right)_{\delta}=\delta^{-2\tilde{h}}((-\Delta_{\delta})^{\tilde{h}}z)_{\delta}.
	\end{align*}\end{proof}

The proof of this lemma suggests that convergence of the operators
$(-\Delta_{\delta})^{h}$ can be obtained from the underlying semigroup.

\begin{proposition} \label{prop:semigroup_props}Let $2\leq p<\infty$, $t>0$. Then: 
	\begin{enumerate}
		\item For any $h\geq 0$ there exists a universal constant $M_{h}<\infty$ such
		that $\sup_{t>0, 0 < \delta \leq 1}\norm{(-t\Delta_{\delta})^{h}S_{\delta}(t)}_{L^{p}(\Lambda_{\delta})}\le M_{h}$. 
		\item If $z\in L^{p}(\R^{d})$, then $S_{\delta}(t)(z|_{\Lambda_{\delta}})\rightarrow e^{t\Delta_{0}}z$
		in $L^{p}(\R^{d})$ as $\delta\rightarrow0$. 
	\end{enumerate}
\end{proposition}

\begin{proof} For $\delta=1$, (i) is a well-known result due to
	the spectrum of the Laplacian being bounded away from zero; cf. \cite[Theorem 2.6.13]{Pazy:1983us}.
	Recall the scaling properties from Lemmas \ref{lem:scaling}(ii) and \ref{lem:delta_neg_frac_scaling}.
	Using $\delta^{d(1/2-1/p)}\norm{z_{\delta}}_{0,p}=\norm z_{L^{p}(\Lambda_{\delta})}$
	for $z\in L^{p}(\Lambda_{\delta})$, we then have 
	\begin{align*}
	& \norm{(-t\Delta_{\delta})^{h}S_{\delta}(t)}_{L^{p}(\Lambda_{\delta})}  =\sup_{\norm z_{L^{p}(\Lambda_{\delta})}=1}\norm{(-(t\delta^{2})\delta^{-2}\Delta_{\delta})^{h}S_{\delta}(t)z}_{L^{p}(\Lambda_{\delta})}\\
	& \quad =\sup_{\norm z_{L^{p}(\Lambda_{\delta})}=1}\norm{((-(t\delta^{2})\delta^{-2}\Delta_{\delta})^{h}S_{\delta}(t)\delta^{d(\frac{1}{2}-\frac{1}{p})}z)_{\delta}}_{0,p}\\
	& \quad =\sup_{\delta^{d(1/2-1/p)}\norm{z_{\delta}}_{0,p}=1}\delta^{d(\frac{1}{2}-\frac{1}{p})}\norm{(-t\delta^{2}\Delta)^{h}S(t\delta^{2})z_{\delta}}_{0,p}\leq M_{h},
	\end{align*}
by applying the statement in (i) for $\delta=1$ in the last inequality. This proves (i). Part (ii) follows from Proposition 3.5(ii) of \cite{AltmeyerReiss2019} (with $A_{\vartheta,\delta,0}^{*}=\vartheta\Delta_{\delta}$) by replacing $L^2(\R^d)$ in the proof with $L^p(\R^d)$.
\end{proof}

\begin{lemma} \label{lem:delta_frac_bound_conv} Let $2\leq p<\infty$,
	$h\geq0$ and let $z\in L^{p}(\R^{d})$ have compact support in
	$\Lambda_{\delta'}$ for some $\delta'>0$. Then we have for $\delta\leq \delta'$:
\begin{enumerate}
\item If $z\in W^{2\lceil h\rceil,p}(\Lambda_{\delta'})$, then $(-\Delta_{\delta})^{h}z\rightarrow(-\Delta_{0})^{h}z$
	in $L^{p}(\R^{d})$ as $\delta\rightarrow0$. 
\item $\sup_{0<\delta\leq 1}\norm{(-\Delta_{\delta})^{-h}z}_{L^{p}(\Lambda_{\delta})}\lesssim\max(\norm z_{L^{1}(\R^{d})},\norm z_{L^{p}(\R^{d})})$, $h<\frac{d}{2}(1-\frac{1}{p})$.
\end{enumerate}
 
\end{lemma}

\begin{proof} (i) The claim is clear when $h\in\N_{0}$, because $\Delta_{\delta}^h z=\Delta^h z$ for all $\delta\leq \delta'$. For non-integer
	$h$ write $h=m+h'$ with $m\in\N_{0}$ and $0<h'<1$. Then $z'=\Delta^{m}z$
	and $(-\Delta_{\delta})^{h}z=(-\Delta_{\delta})^{h'}z'$.
	It is therefore enough to prove the claim for $0<h<1$. Recall the
	formula for the fractional Laplacian in \eqref{eq:frac_Laplace_formula}. By Proposition \ref{prop:semigroup_props}(ii), we have pointwise for fixed $t>0$ that $t^{-h}\Delta_\delta S_\delta(t)z\rightarrow t^{-h}\Delta_0S_0(t)z$ as $\delta\rightarrow 0$. Since the formula in \eqref{eq:frac_Laplace_formula} also holds
	for the fractional Laplacian $\Delta_{0}$ on $\R^{d}$, the result follows from the dominated convergence theorem, noting 
	\begin{align*}
		\norm{t^{-h}\Delta_\delta S_\delta(t)z}_{L^p(\Lambda_\delta)}
		&\leq \min\left(M_1\norm{z}_{L^p(\R^d)}t^{-h-1}, M_0\norm{\Delta z}_{L^p(\R^d)}t^{-h}\right)
	\end{align*}
	with $M_0, M_1$ from Proposition \ref{prop:semigroup_props}(i). 
	
	(ii)  Since $z\in L^{1}(\R^{d})$ by its compact support and because $(-\Delta_{\delta})^{-h}$ is a bounded operator on $L^{p}(\R^{d})$,
we can assume $z\in C(\bar{\Lambda}_{\delta'})$. By \eqref{eq:neg_frac_laplace_formula} it is enough to show for $t\geq0$ that
\[
\sup_{0<\delta\leq 1}\norm{S_{\delta}(t)z}_{L^{p}(\Lambda_{\delta})}\lesssim\min(1,t^{-\frac{d}{2}(1-\frac{1}{p})})\max(\norm z_{L^{1}(\R^{d})},\norm z_{L^{p}(\R^{d})}).
\]

Proposition \ref{prop:semigroup_props}(i) already gives the bound
$\sup_{0<\delta\leq 1}\norm{S_{\delta}(t)z}_{L^{p}(\Lambda_{\delta})}\lesssim\norm z_{L^{p}(\R^{d})}$
for $0\leq t\leq1$. For $t>1$, Proposition 3.5(i) of \cite{AltmeyerReiss2019}
shows $|(S_{\delta}(t)z)(x)|\leq c_{1}e^{c_{2}t\Delta_{0}}|z|(x)$,
$x\in\Lambda_{\delta}$, with universal constants $c_{1},c_{2}>0$.
The result follows therefore from representing the $e^{c_2t\Delta_0}$ as a convolution operator using the heat kernel on
$\R^{d}$ such that by Young's inequality and hypercontractivity of the heat kernel
\[
\norm{e^{c_{2}t\Delta_{0}}|z|}_{L^{p}(\R^{d})}\lesssim\min\left(\norm z_{L^{p}(\R^{d})},t^{-\frac{d}{2}(1-\frac{1}{p})}\norm z_{L^{1}(\R^{d})}\right).\qedhere
\]	
\end{proof}

Next lemma is a simple application of the scaling property of
the fractional Laplacian and relates regularity to decay as $\delta\rightarrow0$. 

\begin{lemma}\label{lem:scalar_ineq} Let $u\in W^{r,p}(\Lambda)$,
	$z\in W^{-r,q}(\Lambda_{\delta})$ for some $\delta>0$, $r\geq 0$ and
	$\frac{1}{p}+\frac{1}{q}=1$. Then 
	\begin{align*}
	|\sc u{z_{\delta}}| & \leq\delta^{r+d(\frac{1}{2}-\frac{1}{p})}\norm u_{r,p}\norm{(-\Delta_{\delta})^{-r/2}z}_{L^{q}(\Lambda_{\delta})}.
	\end{align*}
\end{lemma}

\begin{proof} The Hölder inequality shows 
	\begin{align*}
	|\sc u{z_{\delta}}| & =|\sc{(-\Delta)^{r/2}u}{(-\Delta)^{-r/2}z_{\delta}}|\leq\norm u_{r,p}\norm{z_{\delta}}_{-r,q}.
	\end{align*}
	Lemma \ref{lem:delta_neg_frac_scaling} yields the identity $\norm{z_{\delta}}_{-r,q}=\delta^{r}\norm{((-\Delta_{\delta})^{-r/2}z)_{\delta}}_{0,q}$
	and the result follows by a change of variables.
\end{proof}

\begin{lemma}\label{lem:fracDelta_K} Grant Assumption \ref{assump:K} and let $1<q<\infty$. The following hold true:
\begin{enumerate}
	\item If $r < 2\lceil\gamma\rceil+d(1-1/q)$, then $\sup_{0<\delta\leq 1}\norm{(-\Delta_{\delta})^{-r/2}K}_{L^{q}(\Lambda_{\delta})}<\infty$.
	
\item If $r\leq 2 \lceil\gamma\rceil$, then, as $\delta\rightarrow0$,  $(-\Delta_{\delta})^{-r/2}K\rightarrow(-\Delta_{0})^{\lceil\gamma\rceil-r/2}\widetilde{K}$
	and $(-\Delta_{\delta})^{-r/2}\Delta K\rightarrow(-\Delta_{0})^{\lceil\gamma\rceil-r/2}\Delta\widetilde{K}$
	in $L^{q}(\R^{d})$.
\end{enumerate}

\end{lemma}

\begin{proof} Note that $(-\Delta_{\delta})^{-r/2}K=(-\Delta_{\delta})^{\lceil\gamma\rceil-r/2}\widetilde{K}$,
	$(-\Delta_{\delta})^{-r/2}\Delta K=(-\Delta_{\delta})^{\lceil\gamma\rceil-r/2}\Delta\widetilde{K}$
	with $\widetilde{K}$ and $\Delta\widetilde{K}$ having compact support.
	The two claims follow therefore from Lemma \ref{lem:delta_frac_bound_conv}.
	
\end{proof}

\subsection{Scaling of the covariance function}\label{append:ScaleCov}

In this section, we study the properties of the covariance function of the Gaussian
process $(t,z)\mapsto\sc{\bar{X}(t)}z, t\geq0,z\in L^2(\Lambda)$,  for localized functions $z_{\delta}$, as well as its limit behavior when $\delta\to0$. Repeatedly and sometimes without mentioning it explicitly, we will use properties of the fractional Laplacian $\Delta_{\delta}$ and the semigroup operators $S_{\vartheta,\delta}(t)=S_{\delta}(\vartheta t)$, $t\geq 0$, from Section~\ref{append:Semigroups}. For $t,t'\geq 0$, we use the notations
\begin{align*}
c(t,z,t',z') & :=\text{Cov}(\sc{\bar{X}(t)}z,\sc{\bar{X}(t')}{z'}),& z,z'\in L^2(\Lambda),\\
f_{\delta}(t,u,t',u') & :=\sc{B_{\delta}^{*}S_{\vartheta,\delta}(t)u}{B_{\delta}^{*}S_{\vartheta,\delta}(t')u'}_{L^{2}(\Lambda_{\delta})},& u,u'\in L^{2}(\Lambda_{\delta}),
\end{align*}
and set $c(t,z):=c(t,z,t,z)$, $f_{\delta}(t,u):=f_{\delta}(t,u,t,u)$.

\begin{lemma}\label{lem:covFun1}Grant Assumption \ref{assump:B}
	and let $z,z'\in L^{2}(\Lambda_{\delta})$ for $\delta>0$. Then, for $0\leq t'\leq t\leq T$,
	\[
	c(t,z_{\delta},t',z'_{\delta}) = \delta^{2+4\gamma} \int_0^{t'\delta^{-2}}f_{\delta}((t-t')\delta^{-2}+s,(-\Delta_{\delta})^{-\gamma}z,s,(-\Delta_{\delta})^{-\gamma}z')\dif s.
	\]
\end{lemma}

\begin{proof}Assumption \ref{assump:B} combined with Lemmas~\ref{lem:scaling}(ii) and \ref{lem:delta_neg_frac_scaling} imply the  identities 
	\begin{align*}
	B^{*}S_{\vartheta}(t-s)z_{\delta} & =B^{*}(-\Delta)^{\gamma}S_{\vartheta}(t-s)(-\Delta)^{-\gamma}z_{\delta}\\
	& =\delta^{2\gamma}B^{*}(-\Delta)^{\gamma}(S_{\vartheta,\delta}((t-s)\delta^{-2})(-\Delta_{\delta})^{-\gamma}z)_{\delta}\\
	& =\delta^{2\gamma}(B_{\delta}^{*}S_{\vartheta,\delta}((t-s)\delta^{-2})(-\Delta_{\delta})^{-\gamma}z)_{\delta}.
	\end{align*}
This and Itô isometry (\cite[Proposition 4.28]{DaPratoZabczykBook2014}) yield
	\begin{align*}
	& c(t,z_{\delta},t',z'_{\delta}) =\int_{0}^{t'}\sc{B^{*}S_{\vartheta}(t-s)z_{\delta}}{B^{*}S_{\vartheta}(t'-s)z'_{\delta}}\dif s\\
	& \quad = \delta^{4\gamma}\int_{0}^{t'}f_{\delta}((t-s)\delta^{-2},(-\Delta_{\delta})^{-\gamma}z,(t'-s)\delta^{-2},(-\Delta_{\delta})^{-\gamma}z') \dif s.
	\end{align*}
The result follows by a change of variables. 
\end{proof}

\begin{lemma}\label{lem:covFun2} Grant Assumption \ref{assump:B}
	and let $z,z'\in L^{2}(\Lambda_{\delta})$ for $\delta>0$. Set $z^{(\delta)}:=(-\Delta_{\delta})^{-1/2-\gamma}z$ and 
	$z'^{(\delta)}:=(-\Delta_{\delta})^{-1/2-\gamma}z'$. Then
\begin{align}
\left|c(t,z_{\delta},t',z'_{\delta})\right| 
& \lesssim\delta^{2+4\gamma}\norm{z^{(\delta)}}_{L^{2}(\Lambda_{\delta})} \norm{z'^{(\delta)}}_{L^{2}(\Lambda_{\delta})},\quad 0\leq t'\leq t\leq T \label{eq:CovLinear}, \\
\int_{0}^{t}c(t,z_{\delta},t',z'_{\delta})^{2}\dif t' 
& \lesssim\delta^{6+8\gamma}\norm{(-\Delta_{\delta})^{-1/2}z^{(\delta)}}^2_{L^2(\Lambda_{\delta})} \norm{z'^{(\delta)}}^2_{L^2(\Lambda_{\delta})}. \label{eq:IntCovLinear}
\end{align}
\end{lemma}

\begin{proof} 
Assumption~\ref{assump:B} and the Banach-Steinhaus theorem imply 
\begin{equation}\label{eq:BanachSteinhaus}
\sup_{0\leq\delta\leq 1}\norm{B_{\delta }^{*}}_{L^2(\R^d)}<\infty.
\end{equation}
By Lemma~\ref{lem:covFun1} and the Cauchy-Schwarz inequality, $\delta^{-2-4\gamma}\left|c(t,z_{\delta},t',z'_{\delta})\right|$ is up to a constant bounded by
\begin{align}
 &  \left(\int_{0}^{t'\delta^{-2}} \norm{S_{\vartheta,\delta}(s)S_{\vartheta,\delta}((t-t')\delta^{-2})(-\Delta_{\delta})^{1/2}z^{(\delta)}}_{L^{2}(\Lambda_{\delta})}^{2}\dif s\right)^{1/2} \nonumber\\ 
& \qquad \cdot \left(\int_{0}^{t'\delta^{-2}} \norm{S_{\vartheta,\delta}(s)(-\Delta_{\delta})^{1/2}z'^{(\delta)}}_{L^{2}(\Lambda_{\delta})}^{2}\dif s\right)^{1/2}. \label{eq:Temp1}
\end{align}
Note that 
$\int_{0}^{a}S_{\vartheta,\delta}(2s')u\dif s' = \frac{1}{2}(I-S_{\vartheta,\delta}(2a))(-\vartheta\Delta_{\delta})^{-1}u
$ for $a>0$ and $u\in L^{2}(\Lambda_{\delta})$, 
which consequently implies that
\begin{align}
& \int_{0}^{a}\norm{S_{\vartheta,\delta}(s')u}_{L^{2}(\Lambda_{\delta})}^{2}\dif s' 
 = \int_{0}^{a}\sc{S_{\vartheta,\delta}(2s')u}u_{L^{2}(\Lambda_{\delta})}\dif s' \nonumber\\
& \quad = \frac{1}{2}\sc{(-\vartheta\Delta_{\delta})^{-1}u}u_{L^{2}(\Lambda_{\delta})} -\frac{1}{2}\sc{S_{\vartheta,\delta}(2a)(-\vartheta\Delta_{\delta})^{-1}u}u_{L^{2}(\Lambda_{\delta})} \nonumber \\
& \quad \leq\frac{1}{2}\norm{(-\vartheta\Delta_{\delta})^{-1/2}u}_{L^{2}(\Lambda_{\delta})}^{2}. 
\label{eq:Temp2}
\end{align}
Applying this to \eqref{eq:Temp1} yields
\begin{equation}
\delta^{-2-4\gamma}|c(t,z_{\delta},t',z'_{\delta})|\lesssim\norm{S_{\vartheta,\delta}((t-t')\delta^{-2})z^{(\delta)}}_{L^{2}(\Lambda_{\delta})}\norm{z'^{(\delta)}}_{L^{2}(\Lambda_{\delta})}.\label{eq:covInequ}
\end{equation}
Clearly, \eqref{eq:CovLinear} follows from \eqref{eq:covInequ} by taking $t=t'$. On the other hand, by integrating  \eqref{eq:covInequ} with respect to $t'$, making a change of variables, and applying  \eqref{eq:Temp2}, we obtain
\begin{align*}
& \frac{1}{\delta^{4+8\gamma}}\int_{0}^{t}c(t,z_{\delta},t',z'_{\delta})^{2}\dif t'\lesssim\delta^{2}\int_{0}^{t\delta^{-2}}\norm{S_{\vartheta,\delta}(t')z^{(\delta)}}_{L^{2}(\Lambda_{\delta})}^{2}\dif t'\norm{z'^{(\delta)}}_{L^{2}(\Lambda_{\delta})}^{2}\\
	& \quad\lesssim\delta^{2}\norm{(-\Delta_{\delta})^{-1/2}z{}^{(\delta)}}_{L^{2}(\Lambda_{\delta})}^{2}\norm{z'^{(\delta)}}_{L^{2}(\Lambda_{\delta})}^{2},
	\end{align*}
which implies \eqref{eq:IntCovLinear} at once. 
This concludes the proof. 
\end{proof}

\subsection{Proofs of results in Section~\ref{sec:Main-results}}
\label{append:ProofMainResults}

\begin{proposition} \label{prop:var_cov} Grant Assumptions \ref{assump:B},
	\ref{assump:K}. Then with $\Psi$ from Assumption \ref{assump:ND}, we have as $\delta\rightarrow0$:
	\begin{enumerate}
		\item $\delta^{2-4\gamma}\int_{0}^{T}\E[\bar{X}_{\delta}^{\Delta}(t)^{2}]\dif t\rightarrow T\vartheta{}^{-1}\Psi((-\Delta_{0})^{\lceil\gamma\rceil-\gamma}\Delta\widetilde{K})$, 
		\item $\Var(\int_{0}^{T}\bar{X}_{\delta}^{\Delta}(t)^{2}\dif t)\lesssim\delta^{-2+8\gamma}$. 
	\end{enumerate}
\end{proposition}

\begin{proof}(i) By Lemma~\ref{lem:covFun1} applied to $t=t'$ and $z_{\delta}=z'_{\delta}=\delta^{-2}(\Delta K)_{\delta}$ we write 
$$
\delta^{2-4\gamma}\int_{0}^{T}\E[\bar{X}_{\delta}^{\Delta}(t)^{2}]\dif t  
=\int_{0}^{T}\int_{0}^{\infty}f_{\delta}(s,(-\Delta_{\delta})^{-\gamma}\Delta K)\I_{\{s\leq t\delta^{-2}\}}\dif s\dif t.
$$
Next, we set  $f(s):=\norm{B_{0}^{*}e^{s\vartheta\Delta_{0}}(-\Delta_{0})^{\lceil\gamma\rceil-\gamma}\Delta\widetilde{K}}^{2}$, and note that $\int_{0}^{\infty}f(s)ds = \vartheta{}^{-1}\Psi((-\Delta_{0})^{\lceil\gamma\rceil-\gamma}\Delta\widetilde{K})$,
which clearly follows after substituting $s'=\vartheta s$. 
Recalling that $\sup_{0<\delta\leq 1}\norm{B_{\delta}^{*}}<\infty$ from \eqref{eq:BanachSteinhaus}, and since $\Delta K=\Delta_{\delta}K$, by 
Proposition~\ref{prop:semigroup_props}(i) obtain with $f_\delta(s)\equiv f_\delta(s, (-\Delta_\delta)^{-\gamma}\Delta K)$:
\begin{align*}
|f_{\delta}(s)| & \lesssim \norm{\Delta_{\delta}S_{\vartheta,\delta}(s)(-\Delta_{\delta})^{-\gamma}K}_{L^{2}(\Lambda_{\delta})}^{2}\\
	& \lesssim(1\wedge s^{-2}) \left(\norm{(-\Delta_{\delta})^{-\gamma}K}_{L^{2}(\Lambda_{\delta})}^{2}+\norm{(-\Delta_{\delta})^{-\gamma}\Delta K}_{L^{2}(\Lambda_{\delta})}^{2}\right).
\end{align*}
Consequently, by Lemma~\ref{lem:fracDelta_K},  $|f_{\delta}(s)| \lesssim 1\wedge s^{-2}$, uniformly in $0<\delta\leq 1$, and thus $\sup_{0<\delta \leq 1}|f_{\delta}|\in L^{1}([0,\infty))$.
Setting $K^{(\delta)}:=(-\Delta_{\delta})^{-\gamma}\Delta K$, $K^{(0)}:=(-\Delta_{0})^{-\gamma}\Delta K$, we further have
\begin{align*}
 \norm{B_{\delta}^{*}S_{\vartheta,\delta}(s)K^{(\delta)}-B_{0}^{*}&e^{\vartheta s\Delta_{0}} K^{(0)}}_{L^{2}(\R^{d})}  \lesssim  
\norm{S_{\vartheta,\delta}(s)(K^{(0)}|_{\Lambda_{\delta}})-e^{\vartheta s\Delta_{0}}K^{(0)}}_{L^{2}(\R^{d})} \\
& + \norm{K^{(\delta)}-K^{(0)}}_{L^{2}(\R^{d})} 
 +\norm{(B_{\delta}^{*}-B_{0}^{*}) e^{\vartheta s\Delta_{0}}K^{(0)})}_{L^{2}(\R^{d})}.
\end{align*}
Therefore, the pointwise convergence $f_{\delta}(s)\rightarrow f(s)$, as $\delta\rightarrow0$, follows from Proposition~\ref{prop:semigroup_props}(ii), Lemma~\ref{lem:fracDelta_K}, and Assumption~\ref{assump:B}. Finally, by the dominated convergence theorem (i) is proved.

\medskip \noindent	
(ii) Note that the random variables $\{\bar{X}_{\delta}^{\Delta}(t)\,|\,t\ge0\}$ are centered and jointly Gaussian. Thus, in view of Wick's formula, cf. \cite[Theorem 1.28]{Janson:1997uy}, it follows that 
\begin{align*}
\Var(\int_{0}^{T}\bar{X}_{\delta}^{\Delta}(t)^{2}dt) &=\int_{0}^{T}\int_{0}^{T}\text{Cov}(\bar{X}_{\delta}^{\Delta}(t)^{2},\bar{X}_{\delta}^{\Delta}(t')^{2})\dif t'\dif t\\
& = 4\int_{0}^{T}\int_{0}^{t}\text{Cov}(\bar{X}_{\delta}^{\Delta}(t), \bar{X}_{\delta}^{\Delta}(t'))^{2}\dif t'\dif t\\
& =4\delta^{-8}\int_{0}^{T}\int_{0}^{t}\E[\sc{\bar{X}(t)}{(\Delta K)_{\delta}}\sc{\bar{X}(t')}{(\Delta K)_{\delta}}]^{2}\dif t'\dif t.
\end{align*}
Consequently, by Lemma~\ref{lem:covFun2} with $z=z'=\Delta K$, we continue 
\[
\Var(\int_{0}^{T}\bar{X}_{\delta}^{\Delta}(t)^{2}dt)\lesssim\delta^{-2+8\gamma}\norm{(-\Delta_{\delta})^{-\gamma}K}^2_{L^{2}(\Lambda_{\delta})}\norm{(-\Delta_{\delta})^{1/2-\gamma}K}^2_{L^{2}(\Lambda_{\delta})}.
\]
Invoking Lemma \ref{lem:fracDelta_K}, we conclude the proof. 
\end{proof}

\begin{proof}[Proof of Proposition \ref{prop:pseudo_Fisher_info}]
(i) By Assumption~\ref{assump:B}, Assumption~\ref{assump:K} and Lemma~\ref{lem:delta_neg_frac_scaling}, it follows that 
\begin{align*}
\delta^{-2\gamma}\norm{B^{*}K_{\delta}} & =\norm{B^{*}(-\Delta)^{\gamma}((-\Delta_{\delta})^{-\gamma}K)_{\delta}}=\norm{B_{\delta}^{*}(-\Delta_{\delta})^{\lceil\gamma\rceil-\gamma}\widetilde{K}}_{L^{2}(\Lambda_{\delta})}.
	\end{align*}
Since $\sup_{0<\delta\leq 1}\norm{B_{\delta}^{*}}<\infty$,  cf. \eqref{eq:BanachSteinhaus}, using 
Lemma~\ref{lem:fracDelta_K} we have 
\begin{equation}
	\delta^{-2\gamma}\norm{B^{*}K_{\delta}}\rightarrow\norm{B_{0}^{*}(-\Delta_{0})^{\lceil\gamma\rceil-\gamma}\widetilde{K}}_{L^{2}(\R^{d})}.\label{eq:B_K_convergence}
\end{equation}
Noting that 
\[ \delta^{2}\E[\bar{\c I}_{\delta}] = (\delta^{-2\gamma}\norm{B^{*}K_{\delta}})^{-2}\delta^{2-4\gamma}\int_{0}^{T}\E[\bar{X}_{\delta}^{\Delta}(t)^{2}]\dif t,
\]
and using \eqref{eq:B_K_convergence} and Proposition~\ref{prop:var_cov}(i), the desired result follows at once.

\medskip\noindent
(ii) Convergence~\eqref{eq:B_K_convergence} and Proposition~\ref{prop:var_cov}(ii) imply 
\begin{align*}
\text{\ensuremath{\delta^{4}}Var}(\bar{\c I}_{\delta}) &  = (\delta^{-2\gamma}\norm{B^{*}K_{\delta}})^{-4}\delta^{4-8\gamma}\text{Var}(\int_{0}^{T}\bar{X}_{\delta}^{\Delta}(t)^{2}\dif t)\rightarrow0.
\end{align*}
From this and (i) we get $\Var(\bar{\c I}_{\delta})/\E[\bar{\c I}_{\delta}]^{2}\rightarrow 0$, which gives the result.
	
\medskip\noindent
(iii)  Using the decomposition $X_{\delta}^{\Delta} = \bar{X}_{\delta}^{\Delta}+\widetilde{X}_{\delta}^{\Delta}$ we write 
	\[
	\mathcal{I}_{\delta}-\mathcal{\bar{I}}_{\delta}=\norm{B^{*}K_{\delta}}^{-2}\int_{0}^{T}(\widetilde{X}_{\delta}^{\Delta}(t)^{2}+2\widetilde{X}_{\delta}^{\Delta}(t)\bar{X}_{\delta}^{\Delta}(t))\mathrm{d}t.
	\]
Hence, by (i), \eqref{eq:B_K_convergence} and the Cauchy-Schwarz inequality
	it is enough to have 
	\begin{equation}
	\delta^{2-4\gamma}\int_{0}^{T}\widetilde{X}_{\delta}^{\Delta}(t)^{2}\mathrm{d}t\r{\P}0,
	\end{equation}
which is \eqref{eq:assumptionF_tilde} from Assumption~\ref{assump:F}.

\medskip\noindent
(iv)  The Cauchy-Schwarz inequality implies that  
\[
\c I_{\delta}^{-1}|R_{\delta}|\lesssim\c I_{\delta}^{-1/2}\norm{B^{*}K_{\delta}}^{-1}\left(\int_{0}^{T}\sc{F(t,X(t))}{K_{\delta}}^{2}dt\right)^{1/2}.
\]
By \eqref{eq:B_K_convergence} and items (ii,iii) we find that $\c I_{\delta}^{-1/2}\norm{B^{*}K_{\delta}}^{-1}=O_{\P}(\delta^{1-2\gamma})$.
	Similar to the previous case, (\ref{eq:assumptionF_F}) gives the result.
\end{proof}

\begin{proof}[Proof of Theorem \ref{thm:minimax}]
For both regimes of $\nu$ it is enough to find a lower bound for 
\begin{equation}
	\inf_{\hat{\theta}}\sup_{\kappa\in\{\kappa_1,\kappa_2\}} \E_{\kappa}\left[(\hat{\vartheta}-\vartheta)^2 \right]^{1/2}\label{eq:lowerBound_1}
\end{equation}
 for two suitable alternatives $\kappa_1$, $\kappa_2\in \Theta_{\nu,\gamma}$. When $\nu>1$, fix any $\theta_0>0$ and consider for some $c>0$ the alternatives $\kappa_1=(\theta_0,0,(-\Delta)^{-\gamma},X^{\theta_0}_0)$, $\kappa_2=(\theta_0+c\delta,0,(-\Delta)^{-\gamma},X^{\theta_0+c\delta}_0)$ for the random initial conditions $X^{\theta}_0 = \int_{-\infty}^0 S_{\theta}(-s)(-\Delta)^{-\gamma}dW(s)$ with $\theta\in\{\theta_0,\theta_0+c\delta\}$ and where $(W(s))_{s\in\R}$ is now a two sided cylindrical Brownian motion. As in Proposition \ref{prop:LinearRegularity} one shows that $X_0$ is square integrable. With this initial condition the process $X$ in \eqref{eq:mildSolution} is stationary under $\P_{\theta}$ and Assumption \ref{assump:F} is satisfied, because $F=0$ and $\widetilde{X}(t)=S_{\theta}(t)X_0$ such that  for $\kappa\in\{\kappa_1,\kappa_2\}$
\begin{align*}
	& \E_{\kappa}\left[\int_0^T (\widetilde{X}^{\Delta}_{\delta}(t))^2\mathrm{d}t \right]
	= \int_0^T \E_{\kappa}\left[ \left(\int_{-\infty}^{0}\sc{S_{\theta}(t-s)(-\Delta)^{1-\gamma} K_{\delta}}{dW(s)}\right)^2\right] \mathrm{d}t\\
	& \quad = -\frac{1}{2\theta} \int_0^T \sc{S_{\theta}(2t)(-\Delta)^{1-\gamma} K_{\delta}}{(-\Delta)^{-\gamma} K_{\delta}} \mathrm{d}t
	\leq \frac{1}{4\theta^2} \norm{(-\Delta)^{-\gamma} K_{\delta}}^2\lesssim \delta^{4\gamma},
\end{align*}
concluding by the scaling in Lemma \ref{lem:scaling} and using Lemma \ref{lem:fracDelta_K}. Writing $\bar{K}=\delta^{2\gamma}(-\Delta_{\delta})^{-\gamma}K$ such that $(-\Delta)^{-\gamma}K_{\delta} =\bar{K}_{\delta}$, we observe that also
\[
	X_{\delta}(t) = \int_{-\infty}^t \sc{S_{\theta}(t-s)\bar{K}_{\delta}}{\mathrm{d}W(s)}
\]
is stationary. Arguing exactly as in the proof of Proposition 5.12 and Lemma A.1 of \cite{AltmeyerReiss2019} with respect to $\bar{ K}$ we then obtain for sufficiently small $\delta$ and a suitable constant $c>0$ for \eqref{eq:lowerBound_1} the lower bound
\[
	c\frac{\norm{(I-\Delta_{\delta})^{-1}\bar{K}}^2_{L^2(\Lambda_{\delta})}}{\sqrt{T}\norm{(-\Delta_{\delta})^{1/2}\bar K}^2_{L^2(\Lambda_{\delta})}} \delta = c\frac{\norm{(I-\Delta_{\delta})^{-1}(-\Delta_{\delta})^{-\gamma}K}^2_{L^2(\Lambda_{\delta})}}{\sqrt{T}\norm{(-\Delta_{\delta})^{1/2-\gamma}K}^2_{L^2(\Lambda_{\delta})}} \delta.
\]
The assumed compact support of $\bar{K}$ in Lemma A.1 of \cite{AltmeyerReiss2019} is only necessary to find the limit of the expression before $\delta$ in the last display as $\delta\rightarrow 0$. Here, however, Lemma \ref{lem:fracDelta_K} shows the $L^2(\R^d)$-convergence of $(-\Delta_{\delta})^{1/2-\gamma}K$ and $(-\Delta_{\delta})^{-\gamma}K$, which in turn implies by dominated convergence and the resolvent identity the $L^2(\R^d)$-convergence of
\[
	(I-\Delta_{\delta})^{-1} (-\Delta_{\delta})^{-\gamma}K = \int_0^{\infty}e^{-t}S_{\delta}(t)(-\Delta_{\delta})^{-\gamma}K \mathrm{d}t.
\]
This proves the wanted lower bound for $\nu>1$. 

Let now $\nu\leq 1$ and consider the alternatives $\kappa_1=(\theta_0,\delta^{\nu}\Delta,(-\Delta)^{-\gamma},0)$, $\kappa_2=(\theta_0+\delta^{\nu},0,(-\Delta)^{-\gamma},0)$. Note that the two alternatives correspond to the same linear SPDE \eqref{eq:SPDE}, and the mild solutions \eqref{eq:mildSolution} coincide. In particular, $\P_{\kappa_1}=\P_{\kappa_2}$. We infer from Proposition \ref{prop:LinearRegularity} with $\gamma>1/2+d/4$ that $F(t,X(t))=\delta^{\nu}\Delta X(t)\in L^2(\Lambda)$.  For $\kappa_2$, Assumption \ref{assump:F} is clearly satisfied, and for $\kappa_1$ we see
\begin{align*}
& \E_{\kappa_1}\left[\sc{F(t,X(t))}{K_{\delta,x_0}}^{2}\right] 
	= \delta^{2\nu}\E_{\kappa_2}\left[\bar{X}^{\Delta}_{\delta}(t)^{2}\right]\lesssim \delta^{2\nu-2+4\gamma},
\end{align*}
cf. Proposition \ref{prop:var_cov}(i). Moreover, by the stochastic Fubini theorem (\cite[Theorem 4.33]{DaPratoZabczykBook2014}) and $S_{\theta_0}(t-s)S_{\theta_0+\delta^{\nu}}(s-r)=S_{\theta_0}(t-r)S_{\delta^{\nu}}(s-r)$ we have
\begin{align*}
	\widetilde{X}(t) & =\int_0^t S_{\theta_0}(t-s)\delta^{\nu}\Delta X(s)\mathrm{d}s\\
		& = \int_0^t S_{\theta_0}(t-s)\delta^{\nu}\Delta \int_0^s S_{\theta_0+\delta^\nu}(s-r)(-\Delta)^{-\gamma}\mathrm{d}W(r)\mathrm{d}s\\
		& = \int_0^t S_{\theta_0}(t-r)(S_{\delta^{\nu}}(t-r)-I)(-\Delta)^{-\gamma}\mathrm{d}W(r).
\end{align*}
Itô's isometry and the scaling in Lemmas \ref{lem:scaling} and \ref{lem:delta_neg_frac_scaling} therefore imply
\begin{align*}
	\E_{\kappa_1}\left[(\widetilde{X}^{\Delta}_{\delta}(t))^2 \right]
	& = \delta^{4\gamma-2}\int_0^{t\delta^{-2}} \norm{(S_{1,\delta}(\delta^{\nu}r)-I)S_{\theta_0,\delta}(r)(-\Delta_{\delta})^{1-\gamma} K}^2_{L^2(\Lambda_{\delta})} \mathrm{d}r,
\end{align*}
which is of order $o(\delta^{4\gamma-2})$ by the dominated convergence theorem, using the first parts in Proposition \ref{prop:semigroup_props} and Lemma \ref{lem:fracDelta_K} to find a dominating integrand, and with pointwise convergence following from the second parts of the same results. This means Assumption \ref{assump:F} holds also with respect to $\kappa_1$. Since the two alternatives induce the same law, the total variation distance between $\P_{\kappa_1}$ and $\P_{\kappa_2}$ vanishes. The result follows from equation (2.9) and \cite[Theorem~2.2(i)]{tsybakov2008introduction}, and noting that the two alternatives have distance $|\theta_0+\delta^{\nu}-\theta_0|=\delta^{\nu}$.
\end{proof}

\begin{proof}[Proof of Proposition \ref{prop:ExcessRegularity}] 
We start by proving a general statement. Choose $\epsilon$ and $g$ as in Assumption~\ref{assump:A}. Without loss of generality we can assume that $\epsilon < 2$. Then for any $s_1 \leq s < \bar{s}(p)$, $0\leq \epsilon'<\epsilon$ such that $s+\epsilon'\leq \bar{s}(p)$ and  $p_1\leq p' \leq p$ the following implication holds
\begin{equation}
\sup_{0\leq t\leq T}\norm{\widetilde X(t)}_{s, p'} < \infty \Rightarrow \sup_{0\leq t\leq T}\norm{\widetilde X(t)}_{s+\eta+\epsilon', p'} < \infty.\label{eq:tildeXImplication}
\end{equation}
We proceed as in \cite{DaPratoDebusscheTemam1994}. Use Proposition~\ref{prop:semigroup_props}(i) for $\delta = 1$ to deduce for any $t\in[0,T]$ that
\begin{align*}
\norm{\widetilde X(t)}_{s+\eta+\epsilon', p'} &\leq \norm{S_{\vartheta}(t)X_0}_{s+\eta+\epsilon', p'} + \int_0^t\norm{S_{\vartheta}(t-r)F(X(r))}_{s+\eta+\epsilon', p'}\mathrm{d}r \\
&\lesssim \norm{X_0}_{\bar{s}(p)+\eta, p} + \int_0^t(t-r)^{-1+\frac{\epsilon-\epsilon'}{2}}\norm{F(X(r))}_{s+\eta-2+\epsilon, p'}\mathrm{d}r.
\end{align*}
Assumption \hyperref[assump:A]{$A_{s,\eta,p'}$} and the monotonicity of $g$ allow upper bounding this by
\begin{equation*}
\norm{X_0}_{\bar{s}(p)+\eta, p} + \frac{2}{\epsilon-\epsilon'}T^{\frac{\epsilon-\epsilon'}{2}} g\left(\sup_{0\leq t\leq T}\norm{\bar X(t)}_{s, p'} + \sup_{0\leq t\leq T}\norm{\widetilde X(t)}_{s, p'}\right).
\end{equation*}
Since $X_0\in W^{\bar{s}(p)+\eta,p}(\Lambda)$, $\bar{X}\in C([0,T]; W^{s, p}(\Lambda))$, we obtain \eqref{eq:tildeXImplication}.

Let us now prove the theorem. Applying \eqref{eq:tildeXImplication} iteratively to $p'=p_1$ and all $s_1\leq s < \bar{s}(p)$ gives $\widetilde{X}\in C([0,T]; W^{s+\eta+\epsilon', p_1}(\Lambda))$ for all sufficiently small $\epsilon'\geq 0$ , and thus $\widetilde{X}\in C([0,T]; W^{s_1, p''}(\Lambda))$ by the Sobolev embedding for some suitable $p''>p_1$. Repeating these steps with $p'' > p_1$ instead of $p_1$ until $p''\geq p$ is reached, yields $\widetilde{X}\in C([0,T]; W^{\bar{s}(p)+\eta, p}(\Lambda))$ and $X \in C([0,T]; W^{s,p}(\Lambda))$. 
\end{proof}

\section{Additional Proofs}

\subsection{Local Asymptotics for the Multiplication Operator}

\begin{lemma}\label{lem:G_delta} Let $r \geq 0$, $p>1$ and consider the multiplication operator $M_{\sigma}u=\sigma \cdot u$ with $\sigma\in C^{2r'}(\R^{d})$ for $r'>r$. Let $z\in C^{\infty}(\R^{d})$ with compact support in $\Lambda_{\delta'}$ 
	for some $\delta'>0$ and define for $0<\delta\leq\delta'$ the operator
	$G_{\delta}(\sigma)z:=(-\Delta_{\delta})^{-r}M_{\sigma(\delta\cdot)}(-\Delta_{\delta})^{r}z$.
	Then: 
	\begin{enumerate}
		\item $\sup_{0<\delta\leq \delta'}\norm{G_{\delta}(\sigma)z}_{L^{p}(\R^{d})}\leq C\norm{\sigma}_{C^{2r'}(\R^d)}$, and in particular, $G_{\delta}(\sigma)$ extends to a bounded operator $G_{\delta}(\sigma):L^{p}(\Lambda_{\delta})\rightarrow L^{p}(\Lambda_{\delta})$. 
		\item If $\sigma\in C^{2r'+1}(\R^{d})$, then $G_{\delta}(\sigma)z\r{}G_{0}(\sigma)z:=\sigma(0)z$
		in $L^{2}(\R^{d})$ for $z\in L^{2}(\R^{d})$ as $\delta\rightarrow0$. 
	\end{enumerate}
\end{lemma}

\begin{proof} (i) For $r'$ as in the statement, $\sigma$ induces a bounded multiplication operator $M_{\sigma}$ on $W^{-2r,p}(\Lambda)$; cf. \cite[Theorem 3.3.2]{Triebel1983book}. This means
	\begin{align*}
	\norm{M_{\sigma}u}_{-2r,p} & \leq C\norm{\sigma}_{C^{2r'}(\R^d)}\norm u_{-2r,p},\,\,\,\,u\in W^{-2r,p}(\Lambda).
	\end{align*}
	Therefore we have by Lemma \ref{lem:delta_neg_frac_scaling} 
	\begin{align*}
    & \norm{G_{\delta}(\sigma)z}_{L^{p}(\Lambda_{\delta})} 
	 =\delta^{d(\frac{1}{2}-\frac{1}{p})}\norm{(-\Delta)^{-r}M_{\sigma}(-\Delta)^{r}z_{\delta}}_{L^p(\Lambda)}\\
	 & \quad =\delta^{d(\frac{1}{2}-\frac{1}{p})} \norm{M_{\sigma}(-\Delta)^{r}z_{\delta}}_{-2r,p} \leq C \norm{\sigma}_{C^{2r'}(\R^d)} \norm z_{L^{p}(\R^{d})}.
	\end{align*}
	
	(ii) Because of (i) it is enough to consider $z\in C^{\infty}(\R^{d})$ with
	compact support in $\Lambda_{\delta'}$. We can further restrict to
	$z=\Delta\widetilde{z}$ with $\widetilde{z}\in C^{\infty}(\R^{d})$
	also having compact support in $\Lambda_{\delta'}$. Indeed, assuming
	this holds, let $z\in C_{c}^{\infty}(\bar{\Lambda}_{\delta'})$. 
	Using the Fourier transform $\c Fu$ for $u\in L^{2}(\R^{d})$ define
	with $\epsilon>0$ functions 
	\[
	v_{\epsilon}:=\c F^{-1}[u_{\epsilon}](x)\,\,\,\text{with}\,u_{\epsilon}(\omega):=\frac{1}{\epsilon+|i\omega|^{2}}\c Fz(\omega),\,\,\omega\in\R^{d}.
	\]
	Note that $\c F(\Delta v_{\epsilon})(\omega)=|i\omega|^{2}(\epsilon+|i\omega|^{2})^{-1}\c Fz(\omega)$
	and therefore $\Delta v_{\epsilon}\rightarrow z$ in $L^{2}(\R^{d})$
	as $\epsilon\rightarrow0$. By the Paley-Wiener Theorem, \cite[Theorem II.7.22]{rudin2006functional}, $z$ satisfies the exponential growth condition
	$|\c Fz|(\omega)\leq\gamma_{N}(1+|\omega|)^{-N}\exp((\delta')^{-1}|\text{Im}(\omega)|)$,
	$\omega\in\mathbb{C}^{d}$, for all $N\in\N$ and suitable constants
	$\gamma_{N}$. A reverse application of the same theorem shows that
	$u_{\epsilon}\in C_{c}^{\infty}(\R^{d})$ is also supported in $\Lambda_{\delta'}$.
	Since both $G_{\delta}$ and $G_{0}$ are continuous, this means 
	\begin{align*}
	\norm{(G_{\delta}(\sigma)-G_{0}(\sigma))z}_{L^{2}(\R^{d})} & \lesssim\norm{z-\Delta v_{\epsilon}}+\norm{(G_{\delta}(\sigma)-G_{0}(\sigma))\Delta v_{\epsilon}}.
	\end{align*}
	The result follows from letting first $\delta\rightarrow0$ and then
	$\epsilon\rightarrow0$. 
	
	Assume therefore now that $z=\Delta\widetilde{z}$ with $\widetilde{z}$
	as above. By Taylor's theorem and Lemma \ref{lem:delta_neg_frac_scaling}
	we have 
	\begin{align*}
	(G_{\delta}(\sigma)z-G_{0}(\sigma)z)_{\delta} & =(-\Delta)^{-r}M_{\sigma(\cdot)-\sigma(0)}(-\Delta)^{r}z_{\delta}\\
	& =\sum_{i=1}^{d}\int_{0}^{1}(-\Delta)^{-r}M_{\partial_{i}\sigma(h\cdot)x_{i}}(-\Delta)^{r}z_{\delta}\,dh.
	\end{align*}
	From $M_{\partial_{i}\sigma(h\cdot)x_{i}}=M_{\partial_{i}\sigma(h\cdot)}(-\Delta)^{r}(-\Delta)^{-r}M_{x_{i}}$
	and (i) we find that 
	\[
	\norm{(G_{\delta}(\sigma)-G_{0}(\sigma))z}_{L^{2}(\R^{d})}\leq C \norm{\sigma}_{C^{2r'+1}(\R^d)} \sum_{i=1}^{d}\norm{(-\Delta)^{-r}M_{x_{i}}(-\Delta)^{r}z_{\delta}}.
	\]
	To prove the claim it is enough to show $\norm{(-\Delta)^{-r}M_{x_{i}}(-\Delta)^{r}z_{\delta}}\rightarrow0$,
	$i=1,\dots,d$ as $\delta\rightarrow0$. For this, write $r=m+r'$
	with $m\in\N_{0}$ and $0\leq r'<1$. Iterating the identity $x_{i}\Delta u=\Delta(x_{i}u)-2\partial_{i}u$
	for smooth $u$, we find $x_{i}\Delta^{m+1}\widetilde{z}_{\delta}=\Delta^{m+1}(x_{i}\widetilde{z}_{\delta})-2(m+1)\Delta^{m}\partial_{i}\widetilde{z}_{\delta}$
	such that
	\begin{align*}
	& (-\Delta)^{-r}M_{x_{i}}(-\Delta)^{r}z_{\delta}=\delta^2 (-1)^m (-\Delta)^{-r}M_{x_{i}}(-\Delta)^{r'}\Delta^{m+1}\widetilde{z}_{\delta}=:J_{1,\delta}+J_{2,\delta},\\
	& \quad \text{with}\quad J_{1,\delta}:=\delta^{2}(-1)^{m+1}(-\Delta)^{1-r'}M_{x_{i}}(-\Delta)^{r'}\widetilde{z}_{\delta},\\
	& \quad \qquad\quad J_{2,\delta}:=\delta^{2}(-1)^{m+1}2(m+1)(-\Delta)^{-r'}\partial_{i}(-\Delta)^{r'}\widetilde{z}_{\delta}.
	\end{align*}
	Lemma \ref{lem:delta_neg_frac_scaling} shows $\norm{J_{1,\delta}}=\norm{(-\Delta_{\delta})^{1-r'}M_{\delta x_{i}}(-\Delta_{\delta})^{r'}\widetilde{z}}_{L^{2}(\Lambda_{\delta})}$, and the moment inequality for the fractional Laplacian, cf. \cite[Chapter 2.7.4]{Yagi10}, gives 
	\[
	\norm{J_{1,\delta}} \lesssim \norm{\Delta_{\delta} M_{\delta x_i}(-\Delta_{\delta})^{r'}\tilde{z}}_{L^2(\Lambda_{\delta})}^{1-r'} \norm{M_{\delta x_i}(-\Delta_{\delta})^{r'}\tilde{z}}_{L^2(\Lambda_{\delta})}^{r'}. 
	\]
	Due to the convergence $(-\Delta_{\delta})^{\tilde{r}}\widetilde{z}\rightarrow(-\Delta_{0})^{\tilde{r}}\widetilde{z}$
	in $L^{2}(\R^{d})$ from Lemma \ref{lem:delta_frac_bound_conv}(i), we
	also have find for $\tilde{r}\geq 0$ that
	\[
	\norm{M_{\delta x_{i}}(-\Delta_{\delta})^{\tilde{r}}\widetilde{z}}_{L^{2}(\Lambda_{\delta})}\lesssim\norm{(-\Delta_{\delta})^{\tilde{r}}\widetilde{z}-(-\Delta_{0})^{\tilde{r}}\widetilde{z}}_{L^{2}(\R^{d})}+\norm{M_{\delta x_{i}}(-\Delta_{0})^{\tilde{r}}\widetilde{z}}_{L^{2}(\Lambda_{\delta})},
	\]
	which converges to zero using the dominated convergence theorem. Applying  the identity $x_{i}\Delta_{\delta} u=\Delta_{\delta}(x_{i}u)-2\partial_{i}u$ to $u=(-\Delta_{\delta})^{r'}\widetilde{z}$ thus yields $\norm{J_{1,\delta}}\rightarrow0$.
	With respect to $J_{2,\delta}$, note that $\partial_{i}(-\Delta)^{-1/2}:L^{2}(\Lambda)\rightarrow L^{2}(\Lambda)$ is continuous, as is its adjoint $(-\Delta)^{-1/2}\partial_{i}$ (extended
	to $L^{2}(\Lambda)$). If $r'<1/2$, then
	\begin{align*}
	 \norm{J_{2,\delta}} & \lesssim\delta^{2}\norm{\partial_{i}(-\Delta)^{r'}\widetilde{z}_{\delta}} \leq \delta^{2}\norm{(-\Delta)^{r'+1/2}\widetilde{z}_{\delta}} \\
	& = \delta^{2-2(r'+1/2)}\norm{(-\Delta_{\delta})^{r'+1/2}\widetilde{z}}_{L^2(\Lambda_{\delta})}.
	\end{align*}
	This vanishes as $\delta\rightarrow0$, using Lemmas \ref{lem:delta_neg_frac_scaling}
	and \ref{lem:delta_frac_bound_conv}. For $r'\geq1/2$, we have similarly
	\[
	\norm{J_{2,\delta}}\lesssim\delta^{2}\norm{(-\Delta)^{-1/2}\partial_{i}(-\Delta)^{r'}\widetilde{z}_{\delta}}\lesssim\delta^{2}\norm{(-\Delta)^{r'}\widetilde{z}_{\delta}}\rightarrow0.
	\]
	This finishes the proof.
	
\end{proof}

\subsection{Proof of Theorem \ref{thm:Burgers_CLT}: Bias in Burgers CLT}\label{sec:ProofOfBurgers}

As in Appendix~\ref{app:RegularityLinear}, let $(\lambda_{k},\Phi_k)_{k \in \N}$ denote the eigensystem of~$-\Delta$ on $\Lambda$ (recall that $\Lambda$ here corresponds to the shifted domain $(0,1)-x_0$, as assumed in the beginning of the proof section) such that in $d=1$, $\lambda_k=\pi^2 k^2$ and $\Phi_k=\sqrt{2}\sin(\pi k (x+x_0))$. Also, throughout this section we will assume that the assumptions of Theorem~\ref{thm:Burgers_CLT} are fulfilled. We frequently use that the space $W^{s,p}(\Lambda)$ is an algebra with respect to pointwise multiplication for $p\geq 2$ and $s>1/p$; cf. proof of Lemma \ref{lemLpReactionDiffusion}.

By Proposition~\ref{prop:pseudo_Fisher_info}(i-iii) and equation \eqref{eq:B_K_convergence}, it is enough to show that
\begin{equation}\label{eq:burgers_bias_claim}
\delta^{1-4\gamma}\norm{B^{*}K_{\delta}}^{2}R_{\delta}\r{\P}0,\quad \delta\to 0.
\end{equation}
Using integration by parts, we write 
\begin{align*}
\norm{B^{*}K_{\delta}}^{2}R_{\delta} & =\frac{1}{2}\int_{0}^{T}X_{\delta}^{\Delta}(t)\sc{X(t)^{2}}{\partial_{x}K_{\delta}}\dif t=\frac{1}{2}(U_{1,\delta}+U_{2,\delta}+U_{3,\delta}),
\end{align*}
where 
\begin{align*}
U_{1,\delta} & :=\int_{0}^{T}\bar{X}_{\delta}^{\Delta}(t)\sc{\bar{X}(t)^{2}}{\partial_{x}K_{\delta}}\dif t,  \\
U_{2,\delta} & :=  \int_{0}^{T}\widetilde{X}_{\delta}^{\Delta}(t)\sc{X(t)^{2}}{\partial_{x}K_{\delta}}\dif t +  \int_{0}^{T}\bar{X}_{\delta}^{\Delta}(t)\sc{\widetilde{X}(t)^{2}}{\partial_{x}K_{\delta}}\dif t\\
&  \quad + 2\int_{0}^{T}\bar{X}_{\delta}^{\Delta}(t)\sc{\bar{X}(t)(\widetilde{X}(t)-\widetilde{X}(t,0))}{\partial_{x}K_{\delta}}\dif t =: V_{1,\delta}+V_{2,\delta}+V_{3,\delta}, \\
U_{3,\delta} &:= 2\int_{0}^{T}\widetilde{X}(t,0)\bar{X}_{\delta}^{\Delta}(t) \sc{\bar{X}(t)}{\partial_{x}K_{\delta}}\dif t.
\end{align*}

We will treat each term separately in a series of lemmas below, and show that  $\delta^{1-4\gamma}U_{j,\delta}\r{\P}0$, for $j=1,2,3$. For  $U_{1,\delta}$, cf. Lemma~\ref{lem:U_1}, we use Gaussian calculus, while for $U_{2,\delta}$, we use the excess spatial regularity of $\widetilde X$ over $\bar X$, cf. Lemma~\ref{lem:U_2}. In Lemma~\ref{lem:U_3}, we treat $U_{3,\delta}$ by a Wiener-chaos decomposition of $\widetilde{X}(t,0)$.

\begin{lemma}\label{lem:Kprime}
Let $1<q<\infty$, $r\leq2+2\lceil\gamma\rceil$. Then, as $\delta\to0$, 
\begin{enumerate}
\item $(-\Delta_{\delta})^{-r/2}\partial_{x}K\rightarrow(-\Delta_{0})^{\lceil\gamma\rceil-r/2+1}L$
		in $L^{p}(\R^{d})$,  
\item $(-\Delta_{\delta})^{-r/2}\Phi_k(\delta\cdot)\partial_{x}K\r{}   \Phi_k(0)(-\Delta_{0})^{\lceil\gamma\rceil-r/2+1}L$
in $L^{2}(\R^{d})$. 
Moreover, $\sup_{0<\delta\leq 1}\norm{(-\Delta_{\delta})^{-r/2}\Phi_k(\delta\cdot)\partial_{x}K}_{L^{2}(\Lambda_{\delta})}\lesssim\lambda_{k}^{r'/2}$ for $r'>r$.
\end{enumerate}
\end{lemma}

\begin{proof}(i) Since $ (-\Delta_{\delta})^{-r/2}\partial_{x}K = (-\Delta_{\delta})^{\lceil\gamma\rceil-r/2+1}L$, the claim follows at once by  Lemma~\ref{lem:delta_frac_bound_conv}(i). 
	
(ii)  Let us write $(-\Delta_{\delta})^{-r/2}\Phi_k(\delta\cdot)\partial_{x}K=G_{\delta}(\Phi_k)K^{(\delta)}$ with $K^{(\delta)}:=(-\Delta_{\delta})^{-r/2}\partial_{x}K$ and $G_{\delta}(\cdot)$ from Lemma~\ref{lem:G_delta}. The two claims follow from (i) and Lemma~\ref{lem:G_delta}.	
\end{proof}

\begin{lemma}\label{lem:regularity_bounds}
For any small $\epsilon>0$, uniformly in $0\leq t\leq T$, $k\geq1$, $r\leq1+2\lceil\gamma\rceil$:
	\begin{enumerate}
		\item $|\widetilde{X}_{\delta}^{\Delta}(t)| \lesssim\delta^{2\gamma-\epsilon}$,  
	$|\bar{X}_{\delta}^{\Delta}(t)|\lesssim\delta^{2\gamma-1-\epsilon}$, 			\\ 	$|\sc{\widetilde{X}(t)^{2}}{\partial_{x}K_{\delta}}|\lesssim\delta^{2\gamma+1-\epsilon}$,
	$|\sc{X(t)^{2}}{\partial_{x}K_{\delta}}|\lesssim\delta^{2\gamma-\epsilon}$, 
	
		\item $|\sc{\Phi_k^2}{\partial_{x}K{}_{\delta}}|\lesssim\lambda_{k}^{r}\delta^{r-1/2-\epsilon}$, $|\sc{\widetilde{X}(t)}{\Phi_k}|\lesssim\lambda_{k}^{-\gamma-3/4+\epsilon}$. 
	\end{enumerate}
\end{lemma}

\begin{proof}(i) We note that by $\gamma>1/4$ and Theorem~\ref{thm:ExistenceBurgersComplete}, for      $\epsilon'>0$ and $p\geq2$, 
\begin{equation}
\bar{X}\in C([0,T];W^{2\gamma+1/2-\epsilon',p}(\Lambda)),\quad\widetilde{X}\in C([0,T];W^{2\gamma+3/2-\epsilon',p}(\Lambda)).\label{eq:burgers_X_bar_tilde_regularities}
\end{equation}
Since the Sobolev spaces appearing herein are also algebras with respect to pointwise multiplication, we conclude that  $\widetilde{X}^{2}$ and, respectively, $X^{2}=(\bar{X}+\widetilde{X})^{2}$ belong to the same spaces as $\widetilde{X}$ and, respectively, $\bar{X}$.  
The first two inequalities follow from 	Lemma~\ref{lem:scalar_ineq}, applied to $\delta^{-2}(\Delta K)_{\delta}$
	with $r=1/2+2\gamma-\varepsilon'$ and by putting $\epsilon=\epsilon'+1/p$ for some small $\varepsilon'$ and large $p$.    
The last two inequalities follow similarly, by applying  Lemma~\ref{lem:scalar_ineq} to $\delta^{-1}(\partial_{x}K)_{\delta}$ with $r= 1+2\gamma -\varepsilon'$ and additionally invoking Lemmas~\ref{lem:fracDelta_K} and \ref{lem:Kprime}(i). 
	
(ii) Using the explict form of $\Phi_k$ and $\lambda_k$, by direct computations we deduce that $\norm{\Phi_k^2}_{r,p}\lesssim\lambda_{k}^{r}$. The first statement follows thus as in (i). The second one holds by \eqref{eq:burgers_X_bar_tilde_regularities} such
	that, with $r=3/2+2\gamma-2\epsilon$, $|\sc{\widetilde{X}(t)}{\Phi_k}|\leq\norm{\widetilde{X}(t)}_{r}\norm{\Phi_k}_{-r}\lesssim\lambda_{k}^{-r/2}$.
\end{proof}

\begin{lemma}\label{lem:U_2}
As $\delta\to0$, we have that  $\delta^{1-4\gamma}U_{2,\delta}\r{\P}0$.
	
\end{lemma}

\begin{proof}Lemma \ref{lem:regularity_bounds}(i) yields $V_{1,\delta}=O_{\P}(\delta^{4\gamma-\epsilon})$, $V_{2,\delta}=O_{\P}(\delta^{4\gamma-\epsilon})$ for any small $\epsilon>0$. With respect to $V_{3,\delta}$ expand $\widetilde{X}(t)=\sum_{k\geq1}\sc{\widetilde{X}(t)}{\Phi_k}\Phi_k$
	such that with $g_{k,\delta}(t):=\bar{X}_{\delta}^{\Delta}(t)\sc{\bar{X}(t)}{(\Phi_k-\Phi_k(0))\partial_{x}K_{\delta}}$, we deduce 
	\begin{align}
	|V_{3,\delta}| & =|2\sum_{k\geq1}\int_{0}^{T}\sc{\widetilde{X}(t)}{\Phi_k}g_{k,\delta}(t)\dif t| \lesssim \sum_{k\geq1}\lambda_{k}^{-3/4-\gamma+\epsilon}\int_{0}^{T}|g_{k,\delta}(t)|\dif t,
	\label{eq:burgers_g_k_delta_claims1}
	\end{align}
	where in the last inequality we used Lemma~\ref{lem:regularity_bounds}(ii). By the Cauchy-Schwarz inequality and Lemma \ref{lem:covFun2} 	with $z=z'=\Delta K$ and for $z=z'=v^{(\delta)}:=(\Phi_k(\delta\cdot)-\Phi_k(0))\partial_{x}K$
	we have
	\begin{align*}
	\delta^{1-4\gamma}\E[|g_{k,\delta}(t)|] &\leq\delta^{-2-4\gamma}\E[\sc{\bar{X}(t)}{(\Delta K)_{\delta}}^{2}]^{1/2}\E[\sc{\bar{X}(t)}{v_{\delta}^{(\delta)}}^{2}]^{1/2}\\
	& \lesssim\norm{(-\Delta_{\delta})^{1/2-\gamma}K}_{L^{2}(\Lambda_{\delta})}^{1/2}\norm{(-\Delta_{\delta})^{-1/2-\gamma}v^{(\delta)}}_{L^{2}(\Lambda_{\delta})}.
	\end{align*}
	Consequently, by Lemmas~\ref{lem:fracDelta_K} and \ref{lem:Kprime}(ii), we get 
	\begin{equation}
	\delta^{1-4\gamma}\E[|g_{k,\delta}(t)|]\rightarrow0,\quad\sup_{0<\delta\leq 1,k\geq1,0\leq t\leq T}(\delta^{1-4\gamma}\lambda_{k}^{-\gamma}\E[|g_{k,\delta}(t)|])<\infty,
	\end{equation}
which combined with \eqref{eq:burgers_g_k_delta_claims1} concludes the proof. 

\end{proof}

To deal with $U_{1,\delta}$ and $U_{3,\delta}$, we will prove two additional technical lemmas. For $x,x'\in\Lambda$ and $0 \leq t'\leq t\leq T$ set
\begin{align*}
c_{t,t'}^{\Delta}(x) & :=\E[\bar{X}_{\delta}^{\Delta}(t)\bar{X}(t',x)],\quad &c_{t,t'}(x,x')&:=\E[\bar{X}(t,x)\bar{X}(t',x')],\\
c_{t,t'}^{(1)}(x,x')&:=c_{t',t}^{\Delta}(x)c_{t,t'}^{\Delta}(x'), \quad &c_{t,t'}^{(2)}(x,x')&:=c_{t,t}^{\Delta}(x)c_{t',t'}^{\Delta}(x'). 
\end{align*}

\begin{lemma}\label{lem:gauss_regularity_bounds}
The following assertions hold true, with $c_{t, t}=c_{t, t}(x,x)$: 
	\begin{enumerate}
		\item $|\sc{c_{t,t}}{\partial_{x}K_{\delta}}|\lesssim\delta^{2\gamma-\epsilon}$
		and   $|\int_{0}^{t}\sc{c_{t,t'}}{\partial_{x}K_{\delta}}\dif t'|\lesssim\delta^{1/2+2\gamma-\epsilon}$,
		
		\item $|\int_{\Lambda^{2}}c_{t,t'}(x,x')^{2}\partial_{x}K_{\delta}(x)\partial_{x}K_{\delta}(x')\dif x\dif x'|\lesssim\delta^{4\gamma-\epsilon}$,
		
		\item $|\sc{c_{t,t}^{\Delta}c_{t',t}^{\Delta}}{\partial_{x}K_{\delta}}| \lesssim\delta^{6\gamma-1}$
		and $|\sc{c_{t',t'}^{\Delta}c_{t,t'}^{\Delta}}{\partial_{x}K_{\delta}}| \lesssim\delta^{6\gamma-1}$,
		
		\item for $i=1,2$, as $\delta\to0$,
		$$
		 \delta^{2-8\gamma}\int_{\Lambda^{2}}\int_{0}^{T}\int_{0}^{t}c_{t,t'}(x,x')c_{t,t'}^{(i)}(x,x')\partial_{x}K_{\delta}(x)\partial_{x}K_{\delta}(x')\dif t'\dif t\dif\, (x,x') \to 0. 
		$$

\end{enumerate}
\end{lemma}

\begin{proof}
	
By \eqref{eq:SPDElinear}, and using the representation $W(t)=\sum_{k\geq1}\Phi_k\beta_{k}(t)$, where  	$\beta_{k}, k\geq 1$, are  independent standard Brownian motions, we have that 
\begin{align*}
	\bar{X}(t,x) & =\sum_{k\geq1}\lambda_{k}^{-\gamma} \Phi_k(x) \int_{0}^{t}e^{-(t-r)\vartheta\lambda_{k}}\dif \beta_{k}(r),\\
	\bar{X}_{\delta}^{\Delta}(t) & =\sum_{k\geq1}\lambda_{k}^{-\gamma}\sc{\Phi_k}{\Delta K_{\delta}}\int_{0}^{t}e^{-(t-r)\vartheta\lambda_{k}}\dif \beta_{k}(r).
	\end{align*}
Consequently, using the independence of the $\beta_{k}$'s, we obtain 
\begin{align}
	c_{t,t'}(x,x') & = \frac{1}{2\vartheta}\sum_{k\geq1}e^{-(t-t')\vartheta\lambda_{k}}(1-e^{-2t'\vartheta\lambda_{k})}\lambda_{k}^{-1-2\gamma}\Phi_k(x)\Phi_k(x'),\label{eq:c_t_tprime_x_xprime}\\
	c_{t,t'}^{\Delta}(x) & =\frac{1}{2\vartheta}\left((S_{\vartheta}(2t')-I)S_{\vartheta}(t-t')(-\Delta)^{-2\gamma}K_{\delta}\right)(x).\label{eq:c_t_tprime_x}
	\end{align}	
\smallskip 
\noindent (i) By \eqref{eq:c_t_tprime_x_xprime} and  Lemma~\ref{lem:regularity_bounds}(ii) with $r'=1/2+2\gamma -\varepsilon'$ and $\varepsilon=\varepsilon'$, we deduce 
\begin{align*}
|\sc{c_{t,t}}{\partial_{x}K_{\delta}}| & \lesssim\sum_{k\geq1}\lambda_{k}^{-1-2\gamma}\left|\sc{\Phi_k^{2}}{\partial_{x}K_{\delta}}\right|
 \lesssim \delta^{2\gamma - 2\varepsilon'}.
\end{align*}
Analogously, the second result follows after integrating \eqref{eq:c_t_tprime_x_xprime} with respect to $t'$, and using  Lemma~\ref{lem:regularity_bounds}(ii) with $r'=1+2\gamma -\varepsilon'$, 
\begin{align*}
|\int_{0}^{t}\sc{c_{t,t'}}{\partial_{x}K_{\delta}}\dif t'|\lesssim\sum_{k\geq1}\lambda_{k}^{-2-2\gamma}\left|\sc{\Phi_k^{2}}{\partial_{x}K_{\delta}}\right| \lesssim \delta^{1/2+2\gamma - 2\varepsilon'}. 
\end{align*}

\smallskip
\noindent
(ii) The proof is analogous to (i).

\smallskip
\noindent 
(iii) By Lemmas~\ref{lem:delta_neg_frac_scaling} and \ref{lem:fracDelta_K},  $\norm{c_{t',t}^{\Delta}}_{2\gamma}\lesssim\norm{(-\Delta)^{-\gamma}K_{\delta}}\lesssim\delta^{2\gamma}$, and consequently by the algebra property of Sobolev spaces   $\norm{c_{t,t}^{\Delta}c_{t',t}^{\Delta}}_{2\gamma}\lesssim\delta^{4\gamma}$.  Using this and $\partial_{x}K_{\delta}=\delta^{-1}(\partial_{x}K)_{\delta}$, the desired result follows by applying Lemma~\ref{lem:scalar_ineq} with $r=2\gamma$ and $p=2$ and consequently using Lemma~\ref{lem:Kprime}(i).

\smallskip
\noindent
(iv)  We consider only the case $i=1$, and one can treat the case $i=2$ similarly.  
Using \eqref{eq:c_t_tprime_x_xprime} and \eqref{eq:c_t_tprime_x} we write
$$	 \delta^{2-8\gamma}\int_{\Lambda^{2}}\int_{0}^{t}c_{t,t'}(x,x')c_{t,t}^{\Delta}(x)c_{t',t'}^{\Delta}(x')\partial_{x}K_{\delta}(x)\partial_{x}K_{\delta}(x')\dif t'\dif x\dif x'=\sum_{k\geq1}a_{k,\delta}
$$
with $a_{k,\delta} := \int_{0}^{t}e^{-(t-t')\vartheta\lambda_{k}} (e^{-2t'\vartheta\lambda_{k}}-1)\lambda_{k}^{-1-2\gamma} b_{t,k,\delta}b_{t',k,\delta}\dif t'$, and where, using  Lemmas~\ref{lem:scaling} and \ref{lem:delta_neg_frac_scaling}, 
\begin{align*}
	b_{t,k,\delta}:= & \delta^{1-4\gamma}\sc{c_{t,t}^{\Delta}}{\Phi_k\partial_{x}K_{\delta}}=\delta^{1-4\gamma}\sc{(-\Delta)^{\gamma}c_{t,t}^{\Delta}}{(-\Delta)^{-\gamma}\Phi_k\partial_{x}K_{\delta}}\\
	= & \frac{1}{2\vartheta}\sc{(S_{\vartheta,\delta}(2t\delta^{-2})-I)(-\Delta_{\delta})^{-\gamma}K}{(-\Delta_{\delta})^{-\gamma}\Phi_k(\delta\cdot)\partial_{x}K}_{L^{2}(\Lambda_{\delta})}.
\end{align*}
Due to Proposition~\ref{prop:semigroup_props}(i) and Lemmas~\ref{lem:fracDelta_K},
\ref{lem:Kprime}(ii), note that $\sup_{\delta>0,0\leq t\leq T}|b_{t,k,\delta}|\lesssim\lambda_{k}^{\gamma}$. Moreover, using in addition Proposition~\ref{prop:semigroup_props}(iii), we also deduce that, 
as $\delta\rightarrow0$,
\begin{equation}\label{eq:b_t_k_delta}
b_{t,k,\delta}\rightarrow  -\frac{1}{2\vartheta}\Phi_k(0)\sc{(-\Delta_{0})^{\lceil\gamma\rceil-\gamma}\widetilde{K}}{(-\Delta_{0})^{\lceil\gamma\rceil-\gamma}\partial_{x}\widetilde{K}}_{L^{2}(\R^{d})}.
\end{equation}
Since the fractional Laplacian on $\R^{d}$ is a convolution operator and therefore commutes
with the derivative $\partial_{x}$, after integration by parts, we deduce that the limit in \eqref{eq:b_t_k_delta} vanishes. In all, we have shown that $\sup_{\delta>0,0\leq t\leq T}|a_{k,\delta}|\lesssim\lambda_{k}^{-2-\gamma}$ and $a_{k,\delta}\rightarrow0$, and hence, by the dominated convergence theorem the result follows.
\end{proof}

\begin{lemma}\label{lem:gauss_cov_bounds}
For any any $0\leq t,t'\leq T$, we have
\begin{enumerate}
\item  $|\E[\bar{X}_{\delta}^{\Delta}(t)\bar{X}_{\delta}^{\Delta}(t')]|\lesssim\delta^{4\gamma-1-\epsilon}|t-t'|^{-1/2+\epsilon}$, for $t\neq t'$, and $\epsilon>0$,  
\item $\delta^{-4\gamma}\E[\sc{\bar{X}(t)}{\partial_{x}K_{\delta}}\sc{\bar{X}(t')}{\partial_{x}K_{\delta}}]\rightarrow0$, 
\item $\delta^{1-4\gamma}\E[\bar{X}_{\delta}^{\Delta}(t')\sc{\bar{X}(t)}{\partial_{x}K_{\delta}}]\rightarrow0$, as $\delta\to 0$,.
\end{enumerate}
\end{lemma}

\begin{proof} Using \eqref{eq:covInequ} with $z_{\delta}=z'_{\delta} = (\Delta K)_{\delta}$ yields 
\begin{align*}
 & |\E[\bar{X}_{\delta}^{\Delta}(t)\bar{X}_{\delta}^{\Delta}(t')]|
 =\delta^{-4}|\E[\sc{\bar{X}(t)}{(\Delta K)_{\delta}}\sc{\bar{X}(t')}{(\Delta K)_{\delta}}]|\\
& \quad \lesssim\delta^{4\gamma-2}\norm{S_{\vartheta,\delta}(|t-t'|\delta^{-2}) (-\Delta_{\delta})^{1/2-\gamma}K}_{L^{2}(\Lambda_{\delta})} 
	 \norm{(-\Delta_{\delta})^{1/2-\gamma}K}_{L^{2}(\Lambda_{\delta})}.
	\end{align*}
Then, (i) follows by Proposition~\ref{prop:semigroup_props}(i) with $h=1/2-\epsilon$, combined with Lemma~\ref{lem:fracDelta_K}, where we take $r=1+\gamma-\varepsilon$. 
Assertion (ii) follows similarly by applying \eqref{eq:covInequ} with $z_\delta=z'_\delta=(\partial_{x}K)_\delta$ and using Lemma~\ref{lem:Kprime}(i). For (iii), in view of Lemma~\ref{lem:covFun1} with $B_{\delta}^{*}=I$, we have
	\begin{align*}
	A_{\delta}&:= \delta^{1-4\gamma}\E[\bar{X}_{\delta}^{\Delta}(t')\sc{\bar{X}(t)}{\partial_{x}K_{\delta}}]=\delta^{-2-4\gamma}\E[\sc{\bar{X}(t')}{(\Delta K)_{\delta}}\sc{\bar{X}(t)}{(\partial_{x}K)_{\delta}}]\\
	& =\int_{0}^{(t\wedge t')\delta^{-2}}\sc{S_{\vartheta,\delta}(|t-t'|\delta^{-2}+2s)(-\Delta_{\delta})^{-\gamma}\Delta K}{(-\Delta_{\delta})^{-\gamma}\partial_{x}K}_{L^2(\Lambda_\delta)}\dif s\\
	& =\frac{1}{2\vartheta}\sc{(S_{\vartheta,\delta}((t+ t')\delta^{-2})-S_{\vartheta,\delta}(|t-t'|\delta^{-2})) (-\Delta_{\delta})^{-\gamma}K}{(-\Delta_{\delta})^{-\gamma}\partial_{x}K}_{L^2(\Lambda_\delta)}.
	\end{align*}
When $t\neq t'$, then both semigroups in the above expression vanish as $\delta\rightarrow0$, and the original claim follows.  If $t=t'$, then  the first semigroup vanishes as $\delta\rightarrow0$, and by Lemmas~\ref{lem:fracDelta_K} and \ref{lem:Kprime}(ii)
	\[
A_{\delta} \to 	-\frac{1}{2\vartheta}\sc{(-\Delta_{0})^{\lceil\gamma\rceil-\gamma}\widetilde{K}}{(-\Delta_{0})^{\lceil\gamma\rceil-\gamma}\partial_{x}\widetilde{K}}_{L^{2}(\R^{d})}, \quad \delta\to0. 
\]
By the same arguments as in \eqref{eq:b_t_k_delta}, we conclude that the limiting term is zero, and thus $A_{\delta}\to0$.  This concludes the proof. 
\end{proof}

\begin{lemma}\label{lem:U_1}
As $\delta\to 0$, we have that $\delta^{1-4\gamma}U_{1,\delta}\r{\P}0$.
\end{lemma}
\begin{proof}
Since $\bar{X}(t,x)$ and $\bar{X}_{\delta}^{\Delta}(t,x)$ are centered Gaussians, using Wick's formula for moments of centered Gaussians, cf. \cite[Theorem 1.28]{Janson:1997uy}, we get 
\begin{align*}	
\E[U_{1,\delta}] & =\int_{\Lambda}\int_{0}^{T}\E[\bar{X}_{\delta}^{\Delta}(t)\bar{X}(t,x)^{2}]
\partial_{x}K_{\delta}(x)\dif x\dif t=0, \\ 
\text{Var}(U_{1,\delta}) & =\sum_{\pi\in\Pi_{2}(6)}V_{\pi}, 
\end{align*}
where $\Pi_{2}(6)$ is the set of partitions of $\{1,\dots,6\}$ into 2-tuples (pairs) and where 
\[
V_{\pi}=2\int_{\Lambda^{2}}\int_{0}^{T}\int_{0}^{t}\prod_{(i,j)\in\pi}\E[Z_{i}Z_{j}]\partial_{x}K_{\delta}(x)\partial_{x}K_{\delta}(x')\dif t'\dif t\dif\,(x,x'),
\]
with $Z_{1}=\bar{X}_{\delta}^{\Delta}(t)$, $Z_{2}=Z_{3}=\bar{X}(t,x)$, 	$Z_{4}=\bar{X}_{\delta}^{\Delta}(t')$, $Z_{5}=Z_{6}=\bar{X}(t',x')$. 

Clearly, it is enough to show that $\delta^{2-8\gamma}V_{\pi}\r{\P}0$ for any $\pi\in\Pi_2(6)$. 
Since $Z_{2}=Z_{3}$ and $Z_{5}=Z_{6}$, by symmetry, it is sufficient to consider only six partitions, conveniently grouped as follows:
\begin{align*}
I_1= & \{((1,2),(3,4),(5,6)), \ ((1,5),(2,3),(4,6))\},\\
I_2= & \{((1,2),(3,5),(4,6)),\ ((1,5),(2,4),(3,6))\},\\
I_3= & \{((1,4),(2,3),(5,6))\},\quad I_4= \{((1,4),(2,5),(3,6))\}.
\end{align*}
All relevant terms were already studied in Lemmas~\ref{lem:gauss_regularity_bounds} and \ref{lem:gauss_cov_bounds}. For $\pi\in I_2$, we apply Lemma~\ref{lem:gauss_regularity_bounds}(iv) and obtain that $\delta^{2-8\gamma}V_{\pi}\r{\P}0$. On the other hand, $V_{\pi}=O_{\P}(\delta^{8\gamma-1-\epsilon})$
for $\pi\in I_1$ by Lemma $\ref{lem:gauss_regularity_bounds}$(i,iii), for $\pi\in I_3$ by Lemmas $\ref{lem:gauss_regularity_bounds}$(i) and \ref{lem:gauss_cov_bounds}(i), and for $\pi \in I_4$ by applying Lemmas~$\ref{lem:gauss_regularity_bounds}$(ii) and \ref{lem:gauss_cov_bounds}(i). The proof is complete. 
\end{proof}

\begin{lemma}\label{lem:U_3} 
As $\delta\to 0$, we  have that $\delta^{1-4\gamma}U_{3,\delta}\r{\P}0$.
\end{lemma}

\begin{proof}
Similar to Lemma~\ref{lem:U_1}, we aim to compute the mean and the variance of $U_{3,\delta}$. 
Since $\widetilde{X}(t,0)$ is not Gaussian, we will study its Wiener chaos decomposition (cf. \cite{Nualart2006}).  

We consider the Hilbert space $\c H:=L^{2}([0,T]\times\Lambda)$ endowed with the norm $\norm z_{\c H}=\int_{[0,T]\times\Lambda}z^{2}(t,x)\dif(t,x)$,  and correspondingly let  $(\widetilde{W}(z))_{z\in\c H}$
be the isonormal Gaussian process $\widetilde{W}(z):=\int_{0}^{T}z(t,\cdot)\dif W(t)$. Also, let $(m_{j})_{j\geq1}$ be an orthonormal basis in $L^{2}([0,T])$ such that $(m_{j}\cdot \Phi_k)_{j,k\geq1}$ forms an orthonormal basis in $\c H$. We denote by $\cG$ the sigma algebra generated by $(\widetilde{W}(z))_{z\in\c H}$. It is well-known (see \cite[Proposition 1.1.1]{Nualart2006}) that there exists a sequence of random variables $(\xi_{i})_{i\geq1}$ forming a complete orthonormal system in $L^{2}(\Omega,\c G,\P)$, where each $\xi_{i}$ is a linear combination of multinomials of the form $\Pi_{l=1}^{M}\widetilde{W}(m_{j_{l}}\cdot \Phi_{k_{l}})^{b_{l}}$ for some $M,j_{l},k_{l}\in\N$, $b_{l}\in\N_{0}$.
	
In view of \cite[Theorem 5.1.3 and Example 5.1.8]{LiuRoeckner2015}, where we use that $B$ is Hilbert-Schmidt for $\gamma>1/4$, we have $\widetilde{X}\in L^2([0,T]\times\Omega; W^{1,2}(\Lambda))$, and hence $\widetilde{X}(t,0)\in L^{2}(\Omega,\c G,\P)$ for $0\leq t\leq T$. This yields the chaos expansion $\widetilde{X}(t,0)=\sum_{i\geq1}b_{i}(t)\xi_{i}$, with some deterministic $b_{i}\in L^{2}([0,T])$. For $N\in\N$, we put 
\begin{align*}
U_{3,\delta,N}& := 2 \sum_{i=1}^N \xi_i\int_{0}^{T}b_{i}(t)\bar{X}_{\delta}^{\Delta}(t)\sc{\bar{X}(t)}{\partial_{x}K_{\delta}}\dif t =:2\sum_{i=1}^{N}\xi_i s_{i,\delta}. 
\end{align*}

For a fixed $\eta>0$, choose $N\equiv N(\eta)\in\N$ sufficiently large such that 
\[
\int_{0}^{T}\E[(\widetilde{X}(t,0)-\sum_{i=1}^{N}b_{i}(t)\xi_{i})^{2}]\dif t=\sum_{i=N+1}^{\infty}\int_{0}^{T}b_{i}(t)^{2}\dif t<\eta.
\]

By the Cauchy-Schwarz inequality and Gaussianity we get that
\begin{align*}
& \delta^{2-8\gamma}\E[|U_{3,\delta}-U_{3, \delta, N}|]^{2} \lesssim\delta^{-4-8\gamma}\eta\int_{0}^{T}\E[\sc{\bar{X}(t)}{(\Delta K)_{\delta}}^{2}]\E[\sc{\bar{X}(t)}{(\partial_{x}K)_{\delta}}^{2}]\dif t\\
& \quad\lesssim\eta\norm{(-\Delta_{\delta})^{1/2-\gamma}K}_{L^{2}(\Lambda_{\delta})}^{2}
\norm{(-\Delta_{\delta})^{-1/2-\gamma}\partial_{x}K}_{L^{2}(\Lambda_{\delta})}^{2},
\end{align*}
using in the last inequality Lemma~\ref{lem:covFun2}(i) with $z=z'=\Delta K$
and $z=z'=\partial_{x}K$. Moreover, by Lemmas~\ref{lem:fracDelta_K} and \ref{lem:Kprime}(i), the terms in the last inequality above are uniformly bounded in $\delta>0$, and hence
\begin{equation}\label{eq:burgers_U_3_delta_claim1}
\sup_{0<\delta\leq 1}\left(\delta^{2-8\gamma}\E[|U_{3,\delta}-U_{3,\delta,N}|]^{2}\right)  \lesssim\eta. 
\end{equation}

Next, we will prove that 
\begin{equation}\label{eq:burgers_U3333}
\delta^{2-8\gamma}\E[s_{i,\delta}^{2}]\rightarrow0, \quad i\in\bN.  
\end{equation}
Analogous to Lemma~\ref{lem:U_1}, by Wick's formula and taking advantage of the symmetry in $t,t'$ we obtain 
\begin{align*}
\E[s_{i,\delta}^{2}]&=2\int_{0}^{T}\int_{0}^{t}  b_{i}(t)b_{i}(t')(\rho_{1,\delta}(t,t')+\rho_{2,\delta}(t,t')+\rho_{3,\delta}(t,t'))\dif t'\dif t,\\
\rho_{1,\delta}(t,t') & =\E[\bar{X}_{\delta}^{\Delta}(t)\sc{\bar{X}(t)}{\partial_{x}K_{\delta}}]\E[\bar{X}_{\delta}^{\Delta}(t')\sc{\bar{X}(t')}{\partial_{x}K_{\delta}}],\\
\rho_{2,\delta}(t,t') & =\E[\bar{X}_{\delta}^{\Delta}(t')\sc{\bar{X}(t)}{\partial_{x}K_{\delta}}]\E[\bar{X}_{\delta}^{\Delta}(t)\sc{\bar{X}(t')}{\partial_{x}K_{\delta}}],\\
\rho_{3,\delta}(t,t') & =\E[\sc{\bar{X}(t)}{\partial_{x}K_{\delta}}\sc{\bar{X}(t')}{\partial_{x}K_{\delta}}]\E[\bar{X}_{\delta}^{\Delta}(t)\bar{X}_{\delta}^{\Delta}(t')].
\end{align*}
Clearly \eqref{eq:burgers_U3333} follows from here by invoking the Cauchy-Schwarz inequality and Lemma~\ref{lem:gauss_cov_bounds}(i-iii). Consequently, using \eqref{eq:burgers_U3333} for $1\leq i\leq N$,  and applying again the Cauchy-Schwarz inequality, we deduce that $\delta^{2-8\gamma}\E[U_{3,\delta,N}^2]<\eta$ for any sufficiently small $\delta$ depending on $N$ and thus on $\eta$. Together with \eqref{eq:burgers_U_3_delta_claim1} and since $\eta$ was arbitrary, we get $\delta^{1-4\gamma}U_{3,\delta}\xrightarrow{\P}0$.
\end{proof}

\section{Well-Posedness and higher regularity of the solutions}\label{append:wellPos}

In this section we provide well-posedness and higher regularity results for the linear and semilinear SPDEs needed for our study. This is a well-established topic with a vast literature, see e.g. \cite{DaPratoZabczykBook2014, LiuRoeckner2015, van_Neerven_2012, Krylov96}. We aim at giving a short and self-contained presentation.

\subsection{Regularity of the solution to the linear equation} \label{app:RegularityLinear}

We start with a result on well-posedness of the linear equation, as well as the 
optimal regularity of its solution. We recall that the Laplace operator on any smooth bounded domain $\Lambda\subset\bR^d$ with Dirichlet boundary conditions has only point spectrum $\set{-\lambda_k}_{k\in\bN}$, and without loss of generality can be arranged such that $0\leq \lambda_1\leq \ldots \leq  \lambda_k \leq \ldots$.  Moreover, the corresponding eigenfunctions, say $\set{\Phi_k}_{k\in\bN}$, form a complete orthonormal system in $L^2(\Lambda)$; cf. \cite{Shubin}.  It is also well known that $\lambda_k\sim k^{2/d}$, as $k\to\infty$. Recall the optimal linear regularity $s^*=1+2\gamma-d/2$ from 
Section \ref{sec:higherReg}.

\begin{proposition}\label{prop:LinearRegularity} Grant Assumption~\ref{assump:B}. Then, the linear equation \eqref{eq:SPDElinear} has a unique mild solution $\bar{X}$ taking values in $L^2(\Lambda)$, and for all $2\leq p<\infty$, $s<s^*$:
\begin{enumerate}
\item $\bar{X}\in C([0,T];W^{s-d/2+d/p,p}(\Lambda))$, in particular, $\bar{X}\in C([0,T];W^{s}(\Lambda))$,
\item $\bar{X}\in C([0,T];W^{s,p}(\Lambda))$, provided that  
\begin{equation}\label{eq:supPhi}
	\sup_{k\geq1}\norm{\Phi_k}_{L^\infty(\Lambda)}<\infty.
\end{equation}
\end{enumerate}
Moreover, $s^*$ is maximal with that property, and $\bar X\notin C([0,T];W^{s^*}(\Lambda))$ with probability one.
\end{proposition}
\begin{proof}
Recall \eqref{eq:mildSolutionLinear} and define for $\alpha\geq 0$ the process
\begin{equation}
Y_{\alpha}(t):=\int_{0}^{t}(t-r)^{-\alpha}S_{\vartheta}(t-r)B\mathrm{d}W(r),\quad 0\leq t \leq T.
\end{equation}
We show below for all $s\geq 0$ that
\begin{align}
\mathbb{E}\left[|(-\Delta)^{\frac{s}{2}}Y_{\alpha}(t)(x)|^{2}\right] \leq C \sum_{k\geq 1}\lambda_{k}^{-2\gamma+s+2\alpha-1}\Phi_{k}^2(x),\quad x\in \Lambda.\label{eq:linearRegClaim}
\end{align}
Taking $s,\alpha=0$ shows by Itô's isometry, with Hilbert-Schmidt norm $\norm{\cdot}_2$ on $L^2(\Lambda)$, that $\int_0^t\norm{S_{\vartheta}(t-s)B}^2_{2}\dif s = \E[\norm{\bar{X}(t)}^2]<\infty$. This means that the stochastic integral in \eqref{eq:mildSolutionLinear} is well-defined. That $\bar{X}$ is the unique mild solution to \eqref{eq:SPDElinear}, follows by general theory \cite[Chapter~5]{DaPratoZabczykBook2014}. 

To establish the regularity of $\bar X$, we argue as in \cite[Theorem 5.25]{DaPratoZabczykBook2014} using the factorization method. We first show (ii). Let $s<s^*$, $p\geq 2$ and set $E_1=E_2=W^{s,p}(\Lambda)$. Recall that if $Z\sim N(0,1)$, then $\E[|Z|^{p}]=c_{p}\E[Z^2]^{p/2}$ for some $c_{p}<\infty$. The H\"older inequality and the inequality in \eqref{eq:linearRegClaim} show for $p'\geq 2$ that
\begin{align*}
	\E\left[\int_0^T\norm{(-\Delta)^\frac{s}{2}Y_\alpha(t)}^{p'}_{E_2}\mathrm{d}t\right]
		& \lesssim \int_0^T\left(\int_\Lambda\mathbb{E}\left[|(-\Delta)^\frac{s}{2}Y_{\alpha}(t)(x)|^{2}\right]^{\frac{p}{2}}\mathrm{d}x\right)^{\frac{p'}{p}}\mathrm{d}t \\
		&\lesssim \left(\sum_{k\geq 1}k^{\frac{2}{d}(-2\gamma+2\alpha-1+s)}\right)^{\frac{p'}{p}},
\end{align*}
where we used \eqref{eq:supPhi} in the last line. Since $-2\gamma+2\alpha-1+s<2\alpha-d/2$, the last line is finite for  sufficiently small $\alpha$. We find that $Y_{\alpha}$ has trajectories in $L^{p'}([0,T];E_2)$. Choosing $p'$ large enough such that $\alpha>1/p'$ and $r=0$ in \cite[Proposition 5.9]{DaPratoZabczykBook2014}, we conclude that $\bar X\in C([0,T];E_1)=C([0,T];W^{s,p}(\Lambda))$. This proves (ii). For (i), it is enough to observe for $p'=p=2$ that the upper bound in the last display equals $\sum_{k\geq 1}k^{\frac{2}{d}(-2\gamma+2\alpha-1+s)}$, which is finite for $s<s^*$ as just discussed. The supplement follows from the Sobolev embedding $W^{s}(\Lambda)\subset W^{s-d/2+d/p,p}(\Lambda)$.
	
We still have to prove \eqref{eq:linearRegClaim}. Let $B_1:=(-\Delta)^{\gamma}B$ and note that by Assumption~\ref{assump:B} the operator\footnote{With slight abuse of notations, we use the same notation for $B_1$ as in Assumption~\ref{assump:B}, although strictly speaking  they are not the same.}  $B_1:L^2(\Lambda)\rightarrow L^2(\Lambda)$ is bounded. 
For $x\in\Lambda$ and $f\in C(\Lambda)$, we define $\delta_x(f) = f(x)$. 
Then $\tilde\delta_x := \delta_x\circ (-\Delta)^{-\frac{d}{2}-\epsilon}$ is a bounded linear functional on $L^2(\Lambda)$ for any $\epsilon>0$. 
Hence, 
	\begin{align*}
\mathbb{E}\left[|(-\Delta)^{\frac{s}{2}}Y_{\alpha}(t)(x)|^{2}\right] 
			&= \E\left[\left|\delta_x\left(\int_0^t (t-r)^{-\alpha}(-\Delta)^{\frac{s}{2}-\gamma}S_{\vartheta}(t-r)B_1dW(r)\right)\right|^2\right] \\
&= \int_0^t (t-r)^{-2\alpha}\norm{B_1^*(-\Delta)^{\frac{d}{2}+\epsilon}S_{\vartheta}(t-r)(-\Delta)^{\frac{s}{2}-\gamma}\tilde\delta_x^*}_{2}^2dr \\
			& \lesssim \int_0^t (t-r)^{-2\alpha}\norm{(-\Delta)^{\frac{d}{2}+\epsilon}S_{\vartheta}(t-r)(-\Delta)^{\frac{s}{2}-\gamma}\tilde\delta_x^*}_{2}^2dr \\
			&= \E\left[\left|\delta_x\left(\int_0^t (t-r)^{-\alpha}(-\Delta)^{\frac{s}{2}-\gamma}S_{\vartheta}(t-r)dW(r)\right)\right|^2\right].
	\end{align*}
This allows us to reduce the argument to $B_1=I$, i.e. $B=(-\Delta)^{-\gamma}$. In this case, 
$$
Y_{\alpha}(t,x)=\sum_{k=1}^{\infty}\lambda_{k}^{-\gamma}\left(\int_{0}^{t}(t-r)^{-\alpha}e^{-\lambda_{k}(t-r)}\mathrm{d}\beta_{k}(r)\right)\Phi_{k}(x),
$$
where the $(\beta_k)_{k\in\N}$ are independent standard Wiener processes. The inequality \eqref{eq:linearRegClaim} follows then from
\begin{align*}
\mathbb{E}|(-\Delta)^{\frac{s}{2}}Y_{\alpha}(t,x)|^{2} & =\sum_{k\geq1}\lambda_{k}^{-2\gamma}\left(\int_{0}^{t}r^{-2\alpha}e^{-2\lambda_{k}r}\mathrm{d}r\right)((-\Delta)^{\frac{s}{2}}\Phi_{k}(x))^{2}\\
 & \lesssim\sum_{k\geq 1}\lambda_{k}^{-2\gamma+2\alpha-1+s}\Phi_{k}^2(x).
\end{align*}
Finally, the optimality of $s^*$ follows as in \cite[Proposition 4.3]{PasemannStannat2019}, taking into account that $(-\Delta)^{\gamma}B$ is an isomorphism on $L^2(\Lambda)$.
\end{proof}

Recall the $L^p$-regularity index $\bar{s}$ from \eqref{eq:sBar}. The proposition shows that $\bar{s}\geq s^*-d/2+d/p$ for all $p\geq 2$. Choosing $\alpha=0$ in \eqref{eq:linearRegClaim} also shows $\bar{s}\leq s^*$. The upper bound $\bar{s} = s^*$ is achieved if\eqref{eq:supPhi} holds. The condition \eqref{eq:supPhi} depends on the geometry of the domain $\Lambda$, but is true for rectangular domains in any dimension, in particular, for bounded intervals in $d=1$; cf. the discussion in  \cite[Remark~5.27]{DaPratoZabczykBook2014}.

\subsection{Well-posedness and regularity of the solution to the semilinear equation} \label{app:WellPosedSemilinear}

In this section we study the well-posedness and higher regularity of the solution to \eqref{eq:SPDENonlinear} in its mild formulation \eqref{eq:NonlinearMild}. We will use a classical fixed point argument, cf. \cite{DaPratoZabczykBook2014}, \cite{DaPratoDebusscheTemam1994}. 
In addition to Assumption~\ref{assump:A} from 
Section \ref{sec:higherReg}, 
we will make use of local Lipschitz and coercivity conditions for $p\geq 2$, and $s, s_1, s_2, \eta \geq 0$:

\begin{assumption}{$A_{s, \eta, p}$}\label{assump:A}
	There exists $\epsilon>0$ and a continuous function $g:[0,\infty)\rightarrow [0,\infty)$ such that for $u\in W^{s, p}(\Lambda)$:
	\begin{equation}
	\norm{F(u)}_{s+\eta-2+\epsilon, p}\leq g(\norm{u}_{s, p}).
	\end{equation}
\end{assumption}
\begin{assumption}{$L_{s, \eta, p}$} \label{assump:L} 
	There exist $\epsilon > 0$ and a continuous function $h:[0,\infty)^2\rightarrow[0,\infty)$ such that for any $u, v\in W^{s+\eta, p}(\Lambda)$:
	\begin{equation}
	\norm{F(u)-F(v)}_{s+\eta-2+\epsilon,p}\leq \norm{u-v}_{s+\eta,p}h(\norm{u}_{s, p}, \norm{v}_{s, p}). 
	\end{equation}
\end{assumption}
	
\begin{assumption}{$C_{s_1, s_2}$}\label{assump:C}   
There exists a continuous function $b:[0,\infty)\rightarrow [0,\infty)$ such that for any $u\in W^{s_1}(\Lambda)$, $v\in W^{s_2}(\Lambda)$ with $F(u+v)\in W^{s_1}(\Lambda)$:
\begin{equation}
	|\langle F(u+v), u\rangle_{W^{s_1}(\Lambda)}| \leq (1+\norm{u}_{s_1}^2)b(\norm{v}_{s_2}).
\end{equation}
\end{assumption}

Next, we present the main result of this section. 
\begin{theorem}\label{thm:SemilinearExistence}
Let $s, s_1,\eta\geq 0$, $p\geq 2$ with $s+\eta\geq s_1+2$, and suppose that 
$$
X_0\in W^{s+\eta, p}(\Lambda), \quad \textrm{and} \quad \bar X\in C([0,T];W^{s, p}(\Lambda)).
$$ 
Suppose that Assumption~$A_{s', \eta, p'}$ is satisfied for $s_1\leq s'\leq s$ and $2\leq p'\leq p$. Furthermore suppose that Assumptions~$L_{s, \eta, p}$ and $C_{s_1,s}$ are fulfilled. Then there exists a unique solution $\widetilde{X}$ to \eqref{eq:NonlinearMild} such that  $\widetilde{X}\in C([0,T];W^{s+\eta,p}(\Lambda))$. 

In particular, there exists a unique mild solution $X \in C([0,T];W^{s, p}(\Lambda))$ to equation \eqref{eq:SPDE}.  
\end{theorem}
\begin{proof}
The statement follows from Proposition~\ref{prop:ExcessRegularity}, provided that $\widetilde{X}\in C([0,T];W^{s_1}(\Lambda))$. This inclusion indeed holds true, as proved in Lemma~\ref{lem:SemilinearWellPosedGlobalInTime} below.  
\end{proof}

For the rest of this section, we fix $s, s_1, \eta\geq 0$ and $p\geq 2$ that satisfy the assumptions from Theorem~\ref{thm:SemilinearExistence}. Since all the statements are pathwise, we also fix $\omega\in\Omega$. For $T',m>0$, let 
$$
M(T', m):=\{u\in C([0,T'];W^{s+\eta, p}(\Lambda))\;|\;\sup_{0\leq t\leq T'}\norm{u(t)}_{s+\eta, p}\leq m\},
$$ 
and define the operator $G:M(T', m)\rightarrow C([0,T']; W^{s+\eta, p}(\Lambda))$ as
\begin{equation}
(Gu)(t) = S_{\vartheta}(t)X_0 + \int_0^t S_{\vartheta}(t-r)F(\bar X(r) + u(r))\mathrm{d}r. 
\end{equation}
Note that $M(T',m)$ is a closed ball in a Banach space, hence complete.

\begin{lemma}\label{lemExistenceLocalInTime}
	Suppose that Assumptions $A_{s,\eta, p}$ and $L_{s, \eta, p}$ are fulfilled,  $X_0\in W^{s+\eta, p}(\Lambda)$ and let $m>\norm{X_0}_{s+\eta, p}$. Then,  there exists $T'>0$ such that equation  \eqref{eq:NonlinearMild} has a unique solution in $M(T', m)$.
\end{lemma}
\begin{proof}
Analogous to the proof of Proposition~\ref{prop:ExcessRegularity} with $\epsilon'=0$, $\epsilon<2$, for any $T'>0$, we deduce 
\begin{align*}
	\norm{(Gu)(t)}_{s+\eta, p} 
		&\leq  \norm{X_0}_{s+\eta, p} +\frac{2}{\epsilon}T'^{\frac{\epsilon}{2}} g\left(\sup_{0\leq t\leq T'}\norm{\bar X(t)}_{s, p} + Cm\right),
	\end{align*}
where $C$ is the embedding constant coming from $W^{s+\eta, p}(\Lambda)\subset W^{s, p}(\Lambda)$. Note that the above estimate holds uniformly in $t\in[0,T']$. Moreover, for sufficiently small $T'>0$,  $G$ maps $M(T',m)$ into itself. The claim follows, once it is proved that  $T'$ can be chosen such that $G$ is a contraction mapping on $M(T',m)$, which we will show next. 
By Proposition~\ref{prop:semigroup_props}(i) for $\delta=1$, and Assumption~\ref{assump:L}, for any $u, v\in M(T',m)$, we have 	
\begin{align*}
\norm{&(Gu-Gv)(t)}_{s+\eta, p}  \\
& \leq \int_0^t\norm{S_{\vartheta}(t-r)(F(\bar X(r)+u(r))-F(\bar X(r)+v(r)))}_{s+\eta,p}\mathrm{d}r \\
	&  \lesssim \int_0^t(t-r)^{-1+\frac{\epsilon}{2}}\norm{F(\bar X(r)+u(r))-F(\bar X(r) + v(r))}_{s+\eta-2+\epsilon,p}\mathrm{d}r \\
	& \leq \int_0^t (t-r)^{-1+\frac{\epsilon}{2}} \norm{u(r)-v(r)}_{s+\eta,p} h(\norm{\bar X(r)+u(r)}_{s, p},\norm{\bar X(r)+v(r)}_{s, p})\mathrm{d}r \\
	&  \leq \frac{2}{\epsilon}T'^{\frac{\epsilon}{2}}
	 \sup_{0\leq t\leq T'}\norm{u(t)-v(t)}_{s+\eta,p}\sup_{0\leq t\leq T'}
	 h\left(\norm{\bar X(t)+u(t)}_{s, p},\norm{\bar X(t)+v(t)}_{s, p}\right).
	\end{align*}
Since $\norm{\bar X(t)+u(t)}_{s, p}\leq \sup_{0\leq t'\leq T'}\norm{\bar X(t')}_{s, p} + Cm$, there exists a (random) constant $\tilde C$ such that
\begin{align*}
\sup_{0\leq t\leq T'}\norm{(Gu-Gv)(t)}_{s+\eta, p}\leq \tilde C 
\frac{2}{\epsilon}T'^{\frac{\epsilon}{2}} \sup_{0\leq t\leq T'}\norm{u(t)-v(t)}_{s+\eta,p},
\end{align*}
and hence, for small enough $T'$ the mapping $G$ is a contraction mapping. The proof is complete. 
\end{proof}

\begin{lemma}\label{lem:SemilinearWellPosedGlobalInTime}
	Suppose that Assumptions $A_{s,\eta, p}$, $L_{s, \eta, p}$ and $C_{s_1,s}$ hold, with $s+\eta\geq s_1 + 2$, and suppose that $X_0\in W^{s+\eta, p}(\Lambda)$, $\bar X\in C([0,T];W^{s,p}(\Lambda))$. 
	Then, the solution $\widetilde X$ to (\ref{eq:NonlinearMild}) exists up to time $T$, and $\widetilde{X}\in C([0,T];W^{s_1,2}(\Lambda))$. 
\end{lemma}
\begin{proof} 
By Lemma \ref{lemExistenceLocalInTime}, there exists a solution $\widetilde X \in W^{s+\eta,p}(\Lambda) \subset W^{s_1}(\Lambda)$, locally in time. Let $0<\bar T\leq T$ be the (random) maximal time of existence of $\widetilde X \in W^{s_1}(\Lambda)$. Whenever $\bar T < T$, we have $\sup_{0\leq t\leq \bar T}\norm{\widetilde X(t)}_{s_1}=\infty$. 
 
Assume $\bar T<T$, and set $\widetilde X^{(n)}:=n(n-\vartheta\Delta)^{-1}\widetilde X$. Then, as $n\to\infty$, $\widetilde X^{(n)}\rightarrow\widetilde X$ in $C([0,\bar T];W^{s+\eta, p}(\Lambda))$. Furthermore, 
	\begin{align*}
	R^{(n)}  &:= \frac{\dif}{\dif t}\widetilde X^{(n)}-\vartheta\Delta\widetilde X^{(n)} - F(\bar X + \widetilde X^{(n)}) \\
	&\hphantom{:}= n(n-\vartheta\Delta)^{-1}F(\bar X+\widetilde X) - F(\bar X + \widetilde X^{(n)})\rightarrow 0
	\end{align*}
	in $C([0,\bar T]; W^{s+\eta-2, p}(\Lambda))$ by $L_{s, \eta, p}$, and hence also in $C([0,\bar T];W^{s_1}(\Lambda))$. Now, we apply the chain rule to $\norm{\widetilde X^{(n)}(t)}_{s_1}^2$ and use that the Laplacian is negative-definite such that $\norm{\widetilde X^{(n)}(t)}_{s_1}^2 - \norm{\widetilde X^{(n)}(0)}_{s_1}^2$ equals
\begin{align*}
	& 2\int_0^t 
\left
	\langle\vartheta\Delta\widetilde X^{(n)}(r) + F(\bar X(r) + \widetilde X^{(n)}(r))   + R^{(n)}(r), \widetilde X^{(n)}(r) \right\rangle_{W^{s_1}(\Lambda)} \mathrm{d}r \\
	&  \lesssim \int_0^t
	\left(\left|\left\langle F(\bar X(r) + \widetilde X^{(n)}(r)), \widetilde X^{(n)}(r)\right\rangle_{W^{s_1}(\Lambda)} \right| +\norm{\widetilde X^{(n)}(r)}_{s_1}^2 +\norm{R^{(n)}(r)}_{s_1}^2 \right)\mathrm{d}r \\
	& \lesssim \int_0^t \left(
	\left(1+\norm{\widetilde X^{(n)}(r)}_{s_1}^2\right)b \left( \norm{\bar X(r)}_{s}\right) +\norm{\widetilde X^{(n)}(r)}_{s_1}^2 +\norm{R^{(n)}(r)}_{s_1}^2\right)\mathrm{d}r,
\end{align*}
where we applied $C_{s_1,s}$ in the last inequality. Applying Gronwall's inequality and letting $n\rightarrow\infty$, we conclude that  $\sup_{0\leq t\leq \bar T}\norm{\widetilde X(t)}_{s_1}^2 <\infty$, in contradiction to $\bar T < T$. 
Hence $T=\bar T$ almost surely. 
\end{proof}

In the next two sections we consider two important examples - stochastic reaction-diffusion equations and Burgers equation - and for each of them we provide simple conditions that guarantee that the conclusions from Theorem~\ref{thm:SemilinearExistence} are true. 
  
\subsubsection{Application to reaction-diffusion equations}

As in 
Section~\ref{sec:ReactionDiffMain} 
we consider reaction-diffusion equations whose nonlinearity is given by a function $f:\R\rightarrow\R$, namely $F(u)(x) = f(u(x))$. First, we deal with the case that $f$ is a polynomial
\begin{equation}
	f(x) = a_mx^m + \dots + a_1x + a_0,\label{eq:polyF}
\end{equation}
with $a_m<0$ and $m\in 2\N+1$. 
We prove an auxiliary result:
\begin{lemma}\label{lem:ReacDiffConditions} Let $p\geq1$ and consider $f$ as in \eqref{eq:polyF}. Then:
	\begin{enumerate}
		\item Assumption $A_{s, \eta, p}$ is true for any $\eta<2$, $p\geq 2$ and $s > d/p$.
		\item Assumption $A_{s, \eta, p}$ is true for $s=0$, with $p>d(m-1)/2$ and $\eta<2-d(m-1)/p$.
		\item Assumption $L_{s, \eta, p}$ holds for any $\eta\in[0,2)$, $s > d/p$.
		\item Assumption $C_{s_1,s_2}$ is satisfied with $s_1=0$, $s_2>d/2$. 
	\end{enumerate}
\end{lemma}
\begin{proof} 
(i) This follows from Lemma \ref{lemLpReactionDiffusion} with $\alpha=0$.

\smallskip
\noindent 
(ii) The argument is similar to Lemma \ref{lemLpReactionDiffusion}. It suffices to bound $\norm{x^l}_{\eta-2+\epsilon,p}$, for $l=2,\ldots,m$. Since  $p>\frac{d}{2}(l-1)$ and $0<\eta<2-\frac{d}{p}(l-1)$, by the Sobolev embedding theorem, we have $\norm{x^l}_{\eta-2+\epsilon, p} \lesssim \norm{x^l}_{0,p/l} = \norm{x}_{0,p}^l$.

\smallskip
\noindent (iii) This follows from $\norm{xy}_{s, p}\lesssim\norm{x}_{s, p}\norm{y}_{s, p}$.  
		
\smallskip
\noindent (iv) This is a well-known property, cf. \cite[Example 7.10]{DaPratoZabczykBook2014}. See e.g. \cite[Proposition 2.5]{PasemannStannat2019} for the calculation.
\end{proof}

\begin{proposition}\label{prop:ReacDiffExistence}
Consider $f$ as in \eqref{eq:polyF}. Suppose that $X_0\in W^{s+2,p}(\Lambda)$ and $\bar X\in C([0,T];W^{s, p}(\Lambda))$ for some $p\geq 2$ and $s>1\vee\frac{d}{p}$. Assume that $d\leq 3$ and $p>\frac{dm}{2}$, and if $d=3$ also assume that $m\leq 3$. Then, the assertions of Theorem~\ref{thm:SemilinearExistence} hold true. 
\end{proposition}

\begin{proof}
By Lemma \ref{lem:ReacDiffConditions}, the conditions of Lemma \ref{lemExistenceLocalInTime} and  \ref{lem:SemilinearWellPosedGlobalInTime} are met with $p$ and $s$ for any $\eta<2$ and $s_1=0$. Thus, there is a solution to (\ref{eq:NonlinearMild}) in $L^2(\Lambda)$. As the leading coefficient of $F$ is negative, $F'$ is bounded from above, and it holds for sufficiently smooth $Y$, e.g. $Y\in W^{(3/2)\vee d}(\Lambda)$, that
	\begin{align*}
		\langle\theta\Delta Y+F(Y), Y\rangle_{W^{1,2}(\Lambda)}
			&\leq -\theta\norm{Y}_{2}^2+\langle F'(Y)\nabla Y,\nabla Y\rangle_{L^2} \\
			&\lesssim -\theta\norm{Y}_{2}^2+C\norm{Y}_{1}^2.
	\end{align*}
	Using this coercivity property, one shows as in \cite[Lemma 4.29]{LiuRoeckner2015} using a suitable approximation sequence that $X=\bar X + \widetilde X$ (and thus $\widetilde X$) has in fact values in $W^{1}(\Lambda)$. For additional regularity, we use the Sobolev embedding theorems: If $d=1$ or $d=2$, then $W^{1}(\Lambda)$ is embedded in $L^p(\Lambda)$. By Lemma \ref{lem:ReacDiffConditions}(ii) and Proposition \ref{prop:ExcessRegularity}, $\widetilde X\in C([0,T];W^{s', p}(\Lambda))$ for some $s'>d/p$ (here we use $p>dm/2$). Now conclude inductively with Lemma \ref{lem:ReacDiffConditions}(i). 
	If $d=3$ and $m=3$, we argue similarly. $W^{1}(\Lambda)$ embeds into $L^6(\Lambda)$, so by Lemma \ref{lem:ReacDiffConditions}(ii) with $d=m=3$, $p=6$ and $\eta=1/2$, $\widetilde X$ has values in $W^{1/2, 6}(\Lambda)$, which in turn embeds into $L^q(\Lambda)$ for any $q\geq 2$. Now conclude as in the case $d\in\{1, 2\}$. 
\end{proof}

In particular, using Proposition \ref{prop:LinearRegularity}(i,ii), we have the following result.

\begin{theorem}\label{thm:ExistenceReactionDiffusionComplete}
Consider $f$ as in \eqref{eq:polyF}. Let $d\leq 3$. In the case $d=3$ also assume $m\leq 3$. Grant Assumption \ref{assump:B} and let
	\begin{equation*}
		\gamma > \left\{
		\begin{matrix}
			\frac{1}{4}, & d=1, \\
			1, & d=2, \\
			\frac{3}{2}, & d=3,
		\end{matrix}
		\right.
		\quad\quad s_d = \left\{
		\begin{matrix}
			\frac{1}{2}+2\gamma, & d=1, \\
			-1 + 2\gamma, & d=2, \\
			-2 + 2\gamma, & d=3.
		\end{matrix}
		\right.
	\end{equation*}
For any $p>dm/2$ and $1\vee d/p<s<s_d$, if $X_0\in W^{s+2, p}(\Lambda)$, then there exists a unique solution $\widetilde X\in C([0,T];W^{s+2, p}(\Lambda))$ to (\ref{eq:NonlinearMild}), and in particular, there exists a unique mild solution $X\in C([0,T];W^{s, p}(\Lambda))$ to (\ref{eq:SPDE}). 
\end{theorem}

\begin{remark}
	In Theorem \ref{thm:ExistenceReactionDiffusionComplete}, $\gamma$ is chosen such that $B$ is always a Hilbert-Schmidt operator. In $d=2, 3$, the condition we pose is more restrictive as we just use minimal regularity for $\bar X$ from Proposition \ref{prop:LinearRegularity}(i). Furthermore, note that in $d=3$, when $d/p>1$, w.l.o.g. we can choose $p$ larger such that $1\vee d/p<s_d$ is satisfied. 
\end{remark}

Next, we test the conditions for reaction terms of the form $f\in C^\infty_b(\R)$. 

\begin{lemma}\label{lem:ConditionsBoundedPerturbation} \
	Consider $f\in C^\infty_b(\R)$. Then:
	\begin{enumerate}
		\item Assumption \ref{assump:A} is true for $p\geq 2$, $s>0$ and $\eta<2$, and we can choose $g(x)=C(1+|x|^{1\vee s})$ for some $C>0$. 
		\item Assumption \ref{assump:L} is true for $p\geq 2$, $s>d/p$ and $\eta < 2$.
		\item Assumption \ref{assump:C} is true for $s_1=1$, $s_2\geq 1$.
	\end{enumerate}
\end{lemma}
\begin{proof}
(i) With $\bar f = f - f(0)$, \cite[Theorem A]{AdamsFrazier1992} gives $\norm{\bar f(u)}_{s,p}\lesssim \norm{u}_{s, p} + \norm{u}_{s, p}^s$ in the case $s>1$ and $\norm{\bar f(u)}_{s, p}\lesssim\norm{u}_{s, p}$ in the case $0<s\leq 1$. The claim follows easily.

(ii) First note that $f'\in C^\infty_b(\R)$ as well. Using the algebra property of $W^{s, p}(\Lambda)$, for $u,v\in W^{s, p}(\Lambda)$ together with part (i):
\begin{align*}
			\norm{f(u)-f(v)}_{s,p} &\lesssim \int_0^1\norm{f'(u+t(v-u))}_{s,p}\mathrm{d}t\;\norm{u-v}_{s,p} \\
				&\lesssim \int_0^1\left(1+\norm{u+t(u-v)}_{s,p}^{1\vee s}\right)\mathrm{d}t\;\norm{u-v}_{s,p} \\
				&\lesssim \left(1+\norm{u}_{s, p}^{1\vee s} + \norm{v}_{s, p}^{1\vee s}\right)\norm{u-v}_{s,p},
\end{align*}
		and the claim follows.
		
(iii) Making use of the boundedness of $f\in C^\infty_b(\R,\R)$, we have
		\begin{align*}
			\langle f(u+v), u\rangle_{W^{1, 2}(\Lambda)} 
				&= \langle f(u+v), (-\Delta)u\rangle \\
				&= \langle f'(u+v)(\nabla u +\nabla v),\nabla u\rangle\\
				&\lesssim \norm{u}_{1}^2 + \norm{u}_{1}\norm{v}_{1},
		\end{align*}
		and we can choose $b(x) = 1 + 2x$. 
\end{proof}

\begin{proposition}
Grant Assumption \ref{assump:B}. Consider $f\in C^\infty_b(\R)$.	Suppose that $X_0\in W^{s+2,p}(\Lambda)$ and $\bar X\in C([0,T]; W^{s,p}(\Lambda))$ for some $p\geq 2$, $s>1$. Then, the assertions of Theorem \ref{thm:SemilinearExistence} hold true. 
\end{proposition}
\begin{proof}
	By Lemma \ref{lem:ConditionsBoundedPerturbation}, the conditions of Lemma \ref{lemExistenceLocalInTime} and \ref{lem:SemilinearWellPosedGlobalInTime} are satisfied with $p$, $s$ for any $\eta<2$ and $s_1=1$. Thus, there is a solution to \eqref{eq:NonlinearMild} in $W^{1}(\Lambda)$. Additional $W^{s, p}(\Lambda)$-regularity for $s<3+2\gamma-d/2$ and $p\geq 2$ now follows by iteratively applying Lemma \ref{lem:ConditionsBoundedPerturbation} (i) and the Sobolev embedding theorem. 
\end{proof}

Using this, we immediately get the next result. 

\begin{theorem}\label{thm:compOperators}
Consider $f\in C^\infty_b(\R)$.	Grant Assumption \ref{assump:B} and let $\gamma>1/4$ in $d=1$ or $\gamma>d/2$ in $d\geq 2$, respectively. Let $s_d=1+2\gamma-d$ for $d\geq 2$ and $s_1=1/2+2\gamma$. For any $p\geq 2$ and $1<s<s_d$, if $X_0\in W^{s+2,p}(\Lambda)$, then there exists a unique solution $\widetilde X\in C([0,T];W^{s+2,p}(\Lambda))$ to \eqref{eq:NonlinearMild}, and in particular, there exists a unique mild solution $X\in C([0,T]; W^{s,p}(\Lambda))$ to \eqref{eq:SPDE}. 
\end{theorem}

\subsubsection{Application to the stochastic Burgers equation}
Let $d=1$ and 
\begin{equation}
	F(u) = -u\partial_xu = -\frac{1}{2}\partial_x(u^2).
\end{equation}
Assume $X_0\in W^{s+1, p}(\Lambda)$ and $\bar X\in C([0,T];W^{s, p}(\Lambda))$ for some $s>1$ and $p\geq 2$. Proposition \ref{prop:LinearRegularity} shows that the latter condition is satisfied if $1+2\gamma-d/2>1$, i.e. $\gamma>d/4=1/4$, independently of $p\geq 2$. We note that this assumption can be further weakened; see for instance the analysis in \cite{DaPratoDebusscheTemam1994} that includes the case $\gamma=0$.

\begin{lemma}\label{lem:BurgersConditions} The following statements hold true: 
	\begin{enumerate}
		\item Assumption $A_{s, \eta, p}$ is true for any $p\geq 2$, $s>1/p$ and $\eta<1$.
		\item Assumption $A_{s, \eta, p}$ is true for $s=0$, $p\geq 2$ with $\eta < 1-1/p$.
		\item Assumption $L_{s, \eta, p}$ holds for $p\geq 2$, $s>1/p$ and $\eta\in[0,1)$. 
		\item Assumption $C_{s_1, s_2}$ is true for $s_1=0$ and $s_2>3/2$. 
	\end{enumerate}
\end{lemma}
\begin{proof}
	(i) - (iii) are shown as in Lemma \ref{lem:ReacDiffConditions}. (iv) is well-known, the calculations can be found e.g. in \cite{PasemannStannat2019}.\end{proof}

\begin{proposition}\label{prop:BurgersExistence}
	The conclusions of Theorem \ref{thm:SemilinearExistence} are applicable in this case. 
\end{proposition}

\begin{proof}
	Lemma \ref{lem:BurgersConditions} implies that Lemma \ref{lemExistenceLocalInTime} and \ref{lem:SemilinearWellPosedGlobalInTime} can be applied with $p\geq 2$, $s>1/p$, $\eta<1$ and $s_1=0$. Therefore, the process $\widetilde X$ is well-posed in $C([0,T]; L^2(\Lambda))$. By Lemma \ref{lem:BurgersConditions}(ii) and Proposition \ref{prop:ExcessRegularity}, $\widetilde X$ has values in $W^{s'}(\Lambda)$ for any $s'<1/2$, and consequently in $L^q(\Lambda)$ for any $q\geq 2$. Now conclude as in Proposition \ref{prop:ReacDiffExistence} by applying Lemma \ref{lem:BurgersConditions}(i,ii) iteratively. 
\end{proof}

Combining the above, we obtain the next result on well-posedness of the stochastic Burgers equation. 

\begin{theorem}\label{thm:ExistenceBurgersComplete}
	Grant Assumption \ref{assump:B} and let $\gamma>\frac{1}{4}$. For any $p\geq 2$ and $1<s<1/2+2\gamma$, if $X_0\in W^{s+1,p}(\Lambda)$, then there exists a unique solution $\widetilde X\in C([0,T];W^{s+1, p}(\Lambda))$ to (\ref{eq:NonlinearMild}) and a unique solution $X\in C([0,T];W^{s, p}(\Lambda))$ to (\ref{eq:SPDE}). 
\end{theorem}

\end{appendix}

\paragraph*{Acknowledgement.}
	We thank Markus Reiß and Wilhelm Stannat for very helpful comments and discussions. The authors are grateful to the editors and the anonymous referees for their helpful comments,
	suggestions, and insightful questions which helped to improve the paper.  This research has been partially funded by Deutsche Forschungsgemeinschaft (DFG) - SFB1294/1 - 318763901.

\bibliographystyle{alpha}
\bibliography{bibliography}

\end{document}